\newif\ifpjm \pjmfalse   
\newif\iftwelve \twelvefalse     
\ifpjm \twelvetrue \fi          
\ifpjm \documentclass[12pt]{amsart} \else \documentclass[11pt]{amsart} \fi
\usepackage{latexml} 

\usepackage{etoolbox} 

\usepackage{amssymb} 
\usepackage{mathtools} 
\usepackage{amsthm}  
\usepackage[mathscr]{euscript} 
\usepackage{stmaryrd} 
\usepackage[dvipsnames]{xcolor} 
\usepackage{tikz} 
\usetikzlibrary{graphs,decorations.pathmorphing,decorations.markings} 
\usetikzlibrary{backgrounds}
\usetikzlibrary{fit} 
\usepackage{float} 
\usepackage{enumerate} 
\usepackage{longtable} 
\usepackage{booktabs} 

\usepackage{todonotes} 
\usepackage{hyperref} 
\hypersetup{hidelinks} 
\iflatexml
  \usepackage[capitalise,nameinlink]{cleveref} 
\else
  \usepackage{zref-clever}
  \zcsetup{cap,nameinlink}
\fi
\usepackage[style=alphabetic,natbib=true,isbn=false,url=false,maxbibnames=99,maxalphanames=4]{biblatex} 
\AtEveryBibitem{\clearfield{issn}}
\usepackage{caption}
\usepackage{subcaption} 
\usepackage{tikz}
\usetikzlibrary{arrows,patterns}
\usepackage{tikz-cd}

\newcommand{\rectcolor}{red!40!white}
\newcommand{\rectcolorX}{blue!40!white}
\newcommand{\dotlabel}[1]{$\scriptstyle{#1}$}

\newcommand{\opendot}[1]{\filldraw[fill=white,draw=black] (#1) circle (2pt)}
\newcommand{\greydot}[1]{\filldraw[fill=LightGrey,draw=black] (#1) circle (2pt)}

\newcommand{\token}[4][black]{
    \filldraw[#1] (#2) circle (1.5pt) node[anchor=#3] {\dotlabel{#4}}
}

\newcommand{\identify}[4]{
    \draw[red,dashed] (#1,#2) -- (#1,#4);
    \draw[red,dashed] (#3,#2) -- (#3,#4)
}
\newcommand{\identifyX}[4]{
    \draw[blue,dashed] (#1,#2) -- (#1,#4);
    \draw[blue,dashed] (#3,#2) -- (#3,#4)
}

\newcommand{\deflink}[2]{%
\tikz[baseline=(X.base)]{
\node (X) {$#1$};
\node[anchor=north, font=\scriptsize, overlay] at ([xshift=-3pt,yshift=5pt]X.south east) {(\textup{Def.} \ref{#2})};
}
}

\newcommand{\puncturedtorus}{
		\path[fill=red,opacity=0.1] (-\torwidth+\torpunc,-\torheight)--(\torskew-\torpunc,-\torheight)--(\torskew-\torpunc,-\torheight+\torthickness) -- (-\torwidth+\torpunc,-\torheight+\torthickness)--cycle;
		\path[fill=red,opacity=0.1] (-\torskew+\torpunc,\torheight) -- (\torwidth-\torpunc,\torheight)--(\torwidth-\torpunc,\torheight+\torthickness) -- (-\torskew+\torpunc,\torheight+\torthickness)--cycle;
		\path[fill=blue,opacity=0.1] (-\torwidth+\torpunc,-\torheight+\torpunc*\torgradient)--(-\torskew-\torpunc,\torheight-\torpunc*\torgradient)--(-\torskew-\torpunc,\torheight+\torthickness-\torpunc*\torgradient) -- (-\torwidth+\torpunc,-\torheight+\torthickness+\torpunc*\torgradient)--cycle;
		\path[fill=blue,opacity=0.1] (\torskew+\torpunc,-\torheight+\torpunc*\torgradient) -- (\torwidth-\torpunc,\torheight-\torpunc*\torgradient)--(\torwidth-\torpunc,\torheight+\torthickness-\torpunc*\torgradient) -- (\torskew+\torpunc,-\torheight+\torthickness+\torpunc*\torgradient)--cycle;
        \draw[\rectcolor] (-\torwidth+\torpunc,-\torheight) -- (\torskew-\torpunc,-\torheight);
        \draw[\rectcolor] (-\torwidth+\torpunc,-\torheight+\torthickness) -- (\torskew-\torpunc,-\torheight+\torthickness);
        \draw[dashed,\rectcolor] (-\torskew+\torpunc,\torheight) -- (\torwidth-\torpunc,\torheight);
        \draw[\rectcolor] (-\torskew+\torpunc,\torheight+\torthickness) -- (\torwidth-\torpunc,\torheight+\torthickness);
        \draw[\rectcolor] (-\torwidth+\torpunc,-\torheight) -- (-\torwidth+\torpunc,-\torheight+\torthickness);
        \draw[dashed,\rectcolor] (-\torskew+\torpunc,\torheight) -- (-\torskew+\torpunc,\torheight+\torthickness);
        \draw[\rectcolor] (\torskew-\torpunc,-\torheight) -- (\torskew-\torpunc,-\torheight+\torthickness);
        \draw[\rectcolor] (\torwidth-\torpunc,\torheight) -- (\torwidth-\torpunc,\torheight+\torthickness);
        \draw[\rectcolorX] (-\torwidth+\torpunc,-\torheight+\torthickness+\torpunc*\torgradient) -- (-\torskew-\torpunc,\torheight+\torthickness-\torpunc*\torgradient);
        \draw[dashed,\rectcolorX] (-\torwidth+\torpunc,-\torheight+\torpunc*\torgradient) -- (-\torskew-\torpunc,\torheight-\torpunc*\torgradient);
        \draw[\rectcolorX] (\torskew+\torpunc,-\torheight+\torthickness+\torpunc*\torgradient) -- (\torwidth-\torpunc,\torheight+\torthickness-\torpunc*\torgradient);
        \draw[\rectcolorX] (\torskew+\torpunc,-\torheight+\torpunc*\torgradient) -- (\torwidth-\torpunc,\torheight-\torpunc*\torgradient);
        \draw[\rectcolorX] (-\torwidth+\torpunc,-\torheight+\torpunc*\torgradient) -- (-\torwidth+\torpunc,-\torheight+\torthickness+\torpunc*\torgradient);
        \draw[dashed,\rectcolorX] (-\torskew-\torpunc,\torheight-\torpunc*\torgradient) -- (-\torskew-\torpunc,\torheight+\torthickness-\torpunc*\torgradient);
        \draw[\rectcolorX] (\torskew+\torpunc,-\torheight+\torpunc*\torgradient) -- (\torskew+\torpunc,-\torheight+\torthickness+\torpunc*\torgradient);
        \draw[\rectcolorX] (\torwidth-\torpunc,\torheight-\torpunc*\torgradient) -- (\torwidth-\torpunc,\torheight+\torthickness-\torpunc*\torgradient);
        \draw[black] (-\torskew-\torpunc,\torheight+\torthickness-\torpunc*\torgradient) to[out=0,in=225] (-\torskew+\torpunc,\torheight+\torthickness);
       \draw[dashed,black] (-\torskew-\torpunc,\torheight-\torpunc*\torgradient) to[out=0,in=225] (-\torskew+\torpunc,\torheight);
        \draw[black] (\torwidth-\torpunc,\torheight+\torthickness) to[out=225,in=180] (\torwidth-\torpunc,\torheight+\torthickness-\torpunc*\torgradient);
        \draw[dashed,black] (\torwidth-\torpunc,\torheight) to[out=225,in=180] (\torwidth-\torpunc,\torheight-\torpunc*\torgradient);
        \draw[black] (-\torwidth+\torpunc,-\torheight+\torthickness+\torpunc*\torgradient) to[out=0,in=225] (-\torwidth+\torpunc,-\torheight+\torthickness);
        \draw[dotted,black] (-\torwidth+\torpunc,-\torheight+\torpunc*\torgradient) to[out=0,in=225] (-\torwidth+\torpunc,-\torheight);
        \draw[black] (\torskew-\torpunc,-\torheight+\torthickness) to[out=45,in=180] (\torskew+\torpunc,-\torheight+\torthickness+\torpunc*\torgradient);
        \draw[black] (\torskew-\torpunc,-\torheight) to[out=45,in=180] (\torskew+\torpunc,-\torheight+\torpunc*\torgradient);
}

\tikzset{anchorbase/.style={>=To,baseline={([yshift=-0.5ex]current bounding box.center)}}}
\tikzset{ 
    centerzero/.style={>=To,baseline={([yshift=-0.5ex]#1)}},
    centerzero/.default={0,0}
}
\tikzset{wipe/.style={white,line width=4pt}}

\newcommand{\kk}{\Bbbk}
\newcommand{\Rdaff}{R}
\newcommand{\epsilontransform}{\varepsilon}

\newcommand{\iu}{\mathrm{i}} 
\newcommand{\NN}{\mathbb{N}}
\newcommand{\ZZ}{\mathbb{Z}}
\newcommand{\QQ}{\mathbb{Q}}
\newcommand{\RR}{\mathbb{R}}
\newcommand{\CC}{\mathbb{C}}

\newcommand{\HH}{\mathbb{H}}

\newcommand{\sigmaop}{\zeta}
\newcommand{\sigmaaut}{\boldsymbol{\sigma}}
\newcommand{\thispaper}[1]{#1}
\newcommand{\thispapernew}[1]{\color{red}{#1}}
\newcommand{\notthispaper}[1]{\color{black}{#1}}
\newcommand{\notthispapernew}[1]{\color{blue}{#1}}
\newcommand{\id}{\textup{id}}
\newcommand{\aff}{\textup{aff}}
\newcommand{\affX}{\textup{aff,X}}
\newcommand{\Daff}{\textup{Daff}}
\newcommand{\GL}{\mathrm{GL}}
\newcommand{\SL}{\mathrm{SL}}
\newcommand{\affine}{\textup{a}}
\newcommand{\extaffine}{\textup{ae}}
\newcommand{\TT}{\mathbb{T}}

\DeclareMathOperator{\tr}{\textup{tr}}
\DeclareMathOperator{\opp}{\textup{op}}

\DeclareMathOperator{\Ind}{\textup{Ind}}
\DeclareMathOperator{\Endo}{\textup{End}}
\DeclareMathOperator{\vectspan}{\textup{span}}

\DeclareMathOperator{\gr}{\textup{gr}}

\DeclareMathOperator{\Ad}{\textup{Ad}}

\pgfmathsetmacro{\ra}{1.1547}
\pgfmathsetmacro{\torwidth}{4}
\pgfmathsetmacro{\torheight}{1.9}
\pgfmathsetmacro{\torthickness}{2}
\pgfmathsetmacro{\torskew}{1}
\pgfmathsetmacro{\torpunc}{0.2}
\pgfmathsetmacro{\torpuncratio}{(\torwidth+\torskew-\torpunc)/(\torwidth+\torskew)}
\pgfmathsetmacro{\torgradient}{(2*\torheight-\torthickness)/(\torwidth-\torskew)}
\pgfmathsetmacro{\torangle}{35} 

\iflatexml
  \newcommand{\newzctheorem}[4]{\newtheorem{#1}[#2]{#4}}
\else
  \newcommand{\cref}[1]{\zcref{#1}}
  \newcommand{\Cref}[1]{\zcref[S]{#1}}
  \newcommand{\labelcref}[1]{\zcref[noname]{#1}} 

  \NewDocumentCommand{\newzctheorem}{mmmm}{%
    \newtheorem{#1}[#2]{#4}%
    \AddToHook{env/#1/begin}{\zcsetup{countertype={#2=#3}}}%
  }

  \zcRefTypeSetup{convention}{
    Name-sg = Convention ,
    name-sg = convention ,
    Name-pl = Conventions ,
    name-pl = conventions ,
  }
\fi

\theoremstyle{plain}
\newtheorem{theorem}{Theorem}[section]

\newzctheorem{proposition}{theorem}{proposition}{Proposition}
\newzctheorem{corollary}{theorem}{corollary}{Corollary}
\newzctheorem{lemma}{theorem}{lemma}{Lemma}

\theoremstyle{definition}
\newzctheorem{definition}{theorem}{definition}{Definition}
\newzctheorem{example}{theorem}{example}{Example}

\theoremstyle{remark}
\newzctheorem{remark}{theorem}{remark}{Remark}

\numberwithin{equation}{section}

\ifpjm
  \usepackage[pagewise]{lineno}\linenumbers
  \newcommand{\pjmbreak}{\\}          
  \newcommand{\pjmbr}{\nonumber\\&\quad}       
  \newcommand{\widebreak}{}
  \newcommand{\torscale}{0.30}
\else
  \usepackage[margin=1in]{geometry}   
  \newcommand{\pjmbreak}{}
  \newcommand{\pjmbr}{}
  \newcommand{\widebreak}{\\}
  \newcommand{\torscale}{0.40}
\fi

\ifpjm
  \newcommand{\pjmand}{\,,\\}
\else
  \newcommand{\pjmand}{\quad\text{ and }\quad}
\fi

\iftwelve
  \newcommand{\twelvebreak}{\\}
  \newcommand{\twelvebr}{\nonumber\\&\quad}
\else
  \newcommand{\twelvebreak}{}
  \newcommand{\twelvebr}{}
\fi

\newtoggle{comments}
\togglefalse{comments}
\newtoggle{details}
\togglefalse{details}
\newtoggle{detailsnote}
\togglefalse{detailsnote}

\toggletrue{detailsnote}   

\iftoggle{comments}{%
  \newcommand{\acomments}[1]{
    \ \\
    {\color{red}
      \textbf{AS:} #1
    }
    \\
    }
  \newcommand{\dcomments}[1]{
    \ \\
    {\color{red}
      \textbf{DR:} #1
    }
    \\
    }
}{%
  \newcommand{\acomments}[1]{}
  \newcommand{\dcomments}[1]{}
}

\iftoggle{details}{%
  \newcommand{\details}[1]{%
    \par\medskip
    {\color{OliveGreen}%
      \noindent\textbf{Details:}\par
      #1\par
    }%
    \medskip
  }%
}{%
  \newcommand{\details}[1]{}%
}

\iflatexml
  \renewcommand{\deflink}[2]{#1\ \text{\scriptsize(Def.~\ref{#2})}}
  \renewcommand{\tfrac}{\frac}  
\fi

\begin{document}

\definecolor{DarkGreen}{HTML}{006400}
\definecolor{DarkRed}{HTML}{8B0000}
\definecolor{LightGrey}{HTML}{C0C0C0}

\title{Double Affine Frobenius Hecke Algebras}
\author{Niccolò Mattarei}
\address[N.M.]{
  Department of Mathematics and Statistics \\
  University of Ottawa\\
  150 Louis-Pasteur Private,\newline
  Ottawa, ON K1N 6N5, Canada
}
\urladdr{\href{https://sites.google.com/view/niccolo-mattarei/}{sites.google.com/view/niccolo-mattarei/}}
\email{nmatt044@uottawa.ca}

\begin{abstract}
We define double affine Hecke algebras associated to the $\textrm{GL}_n$ and $\textrm{SL}_n$ root data, decorated by a supersymmetric Frobenius superalgebra. These are double affine versions of the quantum affine wreath algebras of Rosso and Savage. When the Frobenius algebra is the group algebra of a finite cyclic group, we obtain new double affine versions of the affine Yokonuma--Hecke algebras. We prove a basis theorem, describe the supercentre for generic parameters, and define polynomial representations. We also provide a diagrammatic interpretation, from which we deduce certain involutions and symmetries of the algebra, and use it to describe an $\textrm{SL}_2(\mathbb{Z})$-action by outer automorphisms.
\end{abstract}

\subjclass[2020]{Primary 20C08; Secondary 16S35, 20F36}
\keywords{DAHA, Hecke algebra, Yokonuma--Hecke algebra, symmetric algebra, Frobenius algebra}

\maketitle



\section{Introduction}
Double affine Hecke algebras (DAHAs) were first defined by Ivan Cherednik \cite{DAHAKZMO_1992}. They have a canonical representation on the ring of Laurent polynomials in $n$ variables, in which certain elements act by difference Dunkl operators that have the non-symmetric Macdonald polynomials as eigenvectors, and other elements act by intertwining operators that map between eigenstates. Cherednik used DAHAs to prove Macdonald's constant term conjecture in \cite{DAHAMC_1995}.

The representation theory of DAHAs is closely related to harmonic analysis. Macdonald polynomials are $(q,t)$-deformed interpolations of the classical spherical functions on a symmetric space, and the discrete spherical functions on a $p$-adic symmetric space. This connection is reflected in an $\SL_2(\ZZ)$-symmetry of the DAHA that corresponds to interchanging the $X$ and $Y$ generators of the algebra up to inversion, and plays a role in the study of the $q$-deformed Fourier transform and other applications in harmonic analysis.

The same symmetry has an interpretation in the topology of the torus, in which the $X$ and $Y$ generators of the DAHA correspond to loops around the meridian and longitude of the torus. Since the DAHA can be realised as a quotient of the group algebra of a braid group on a punctured torus, it has been used to study braids and tangles in the torus \cite{DAHAST_2021}. For example, the DAHA--Jones invariant \cite{JPTKDAHA_2013} generalises the Jones polynomial for torus knots, and the DAHA is closely related to algebras appearing in skein theory \cite{DAHAST_2021}.

Beyond harmonic analysis and knot theory, the DAHA and its degenerations have applications to quantum integrable systems and geometry. The commuting quantum Hamiltonians of the many-particle Calogero--Moser--Sutherland systems are related to commutative subalgebras of the DAHA and its degenerations \cite{CMSRT_2007}, and the degenerations of the DAHA are used to study the geometry of Hilbert schemes \cite{HSHACSS_2003}.

Clifford-algebra-decorated versions of the trigonometric degeneration of the DAHA and the rational Cherednik algebras have been studied in \cite{DAHASSG_2009} for $\GL_n$. In this paper, we begin a programme in which we systematically extend these constructions by developing the theory of a full, non-degenerate DAHA $\HH_n(A,z,q)$ that is decorated by a supersymmetric Frobenius superalgebra $A$, and its $\SL_n$ counterpart $\HH_{\SL_n}(A,z,q)$. In a future paper, we will study the trigonometric and rational degenerations of $\HH_n(A,z,q)$ and $\HH_{\SL_n}(A,z,q)$, many of whose properties can be deduced from the non-degenerate setting via the degeneration process, and for which $A$ need not be supersymmetric. Moreover, we deliberately work over a general commutative ring $\kk$, identifying the module-theoretic criteria required at each stage of the construction. \Cref{dia:introduction} gives a summary of the intended research programme. The algebras in red are treated in this paper, and the algebras in blue will be studied in an upcoming paper. None of the red or blue algebras seem to appear in this generality in the current literature. The algebra $H_n(A,z)$ is the quantum wreath algebra of \cite[Def.~2.1]{QAWA_2020}. The algebra $H_n^\affX(A,z)$ is the quantum affine wreath algebra of \cite[Def.~2.5]{QAWA_2020}, which we call the affine Frobenius Hecke algebra. The algebra $H_n^\textup{degen}(A,z)$ is the affine wreath product algebra of \cite[Def.~3.1]{AWPA_2020}.

\begin{figure} 
\begin{tikzcd}[column sep=large]
\deflink{\thispaper{H_n(A,z)}}{def:AFHAs} \arrow[d, hookrightarrow] \arrow[r, hookrightarrow] 
& \deflink{\thispaper{H_n^\affX(A,z)}}{def:AFHAs} \arrow[d, hookrightarrow] & \\
\deflink{\thispaper{H_n^\aff(A,z)}}{def:AFHAs} \arrow[r, hookrightarrow] \arrow[d, rightsquigarrow, "\gr"']
& \deflink{\thispapernew{\HH_n(A,z,q)}}{thm:double_affine_Frobenius_Hecke_algebra} \arrow[d, rightsquigarrow, "\gr"'] \arrow[r, rightsquigarrow, "\text{GL to SL}"]
& \deflink{\thispapernew{\HH_{\SL_n}(A,z,q)}}{def:DAFHA_SLn} \arrow[d, rightsquigarrow, "\gr"'] \\
\notthispaper{H_n^\textup{degen}(A,z)} \arrow[r, hookrightarrow]
& \notthispapernew{\HH_n^\textup{trig}(A,c,\kappa)} \arrow[d, rightsquigarrow, "\gr"'] \arrow[r, rightsquigarrow, "\text{GL to SL}"] 
& \notthispapernew{\HH_{\SL_n}^\textup{trig}(A,c,\kappa)} \arrow[d, rightsquigarrow, "\gr"'] \\
& \notthispapernew{\HH_n^\textup{rat}(A,c,\kappa)} \arrow[r, rightsquigarrow, "\text{GL to SL}"]
& \notthispapernew{\HH_{\SL_n}^\textup{rat}(A,c,\kappa)}
\end{tikzcd}
\caption{Relationship between the algebras of the intended research programme.} \label{dia:introduction}
\end{figure}

Let $\kk$ be a commutative ground ring and $z \in \kk$. Given a supersymmetric Frobenius superalgebra $A$, the corresponding Frobenius Hecke algebra $H_n(A,z)$ of \cite{QAWA_2020} is an $A$-decorated version of the type $A_{n-1}$ Iwahori--Hecke algebra, which we recover at $A=\kk=\CC[\tau,\tau^{-1}]$ and $z=\tau-\tau^{-1}$. One can also view the Frobenius Hecke algebra as a $z$-deformed wreath algebra, as $H_n(A,0) = A^{\otimes n} \rtimes S_n$. The affine Frobenius Hecke algebra $H_n^\aff(A,z)$ of \cite{QAWA_2020} at $A=\kk=\CC[\tau,\tau^{-1}]$ and $z=\tau-\tau^{-1}$ recovers the affine Hecke algebra associated to $\GL_n$.

Fix $q \in \kk^\times$. In this paper, we define double affine Frobenius Hecke algebras $\HH_n(A,z,q)$ and $\HH_{\SL_n}(A,z,q)$ of types $\GL_n$ and $\SL_n$, which, at $A=\kk=\CC[\tau,\tau^{-1}]$ and $z=\tau-\tau^{-1}$, are Cherednik's usual double affine Hecke algebras for $\GL_n$ and $\SL_n$ respectively. The algebra $\HH_n(A,z,q)$ is a double affine version of $H_n^\aff(A,z)$ from \cite{QAWA_2020}. Consequently, $\HH_n(A,z,q)$ and $\HH_{\SL_n}(A,z,q)$ provide double affine versions of several existing algebras:
\begin{itemize}
\item When $A=\kk G$ for a finite group $G$ and $z=\tau-\tau^{-1}$, the double affine Frobenius Hecke algebra is a double affine quantum version of Wan and Wang's wreath Hecke algebra from \cite[Def.~2.4]{MRBRWHA_2008}.
\item When $\kk=\CC[\tau,\tau^{-1}]$, $A=\kk C_d$ is the group algebra of a finite cyclic group of order $d$, and $z=(\tau-\tau^{-1})/d$, the double affine Frobenius Hecke algebra is a double affine version of the affine Yokonuma--Hecke algebra from \cite[Sec.~3.1]{RTYHA_2014}.
\item When $A$ is the zigzag algebra of \cite{CFAR_2001}, the double affine Frobenius Hecke algebra is a double affine quantum version of an even version of the algebra from \cite[Sec.~10.3]{HCHS_2012}.
\end{itemize}

The above specialisations of $\HH_n(A,z,q)$ do not seem to appear in the literature. Moreover, when $A$ is commutative, a trigonometric degeneration of $\HH_n(A,z,q)$ recovers the algebra from \cite[p.~13]{HSHACSS_2003}, which is used to study the cohomology of Hilbert schemes of points on a smooth projective surface $X$. When $\kk=\CC$ and $A=\Endo_\kk(\CC^r)$ for some integer $r > 0$, a rational degeneration of $\HH_{\SL_n}(A,z,q)$ recovers an algebra from \cite[Def.~2.4]{RMSQGV_2021}, which is used to study representations of quantized Gieseker varieties.

We give an overview of the results of this paper. In \cref{sec:DAFHA_presentations}, we define the $\GL_n$ double affine Frobenius Hecke algebras, with two presentations (\cref{thm:double_affine_Frobenius_Hecke_algebra} and \cref{thm:DAFHA_bernstein_presentation}). In \cref{sec:DAFHA_basis}, we prove a basis theorem for $\HH_n(A,z,q)$, \cref{thm:DAFHA_basis_theorem}. In \cref{sec:centre_for_DAFHA}, we describe the supercentre of $\HH_n(A,z,q)$ when $q^m-1$ is not a zero divisor in $\kk$ for all non-zero $m \in \ZZ$, in \cref{thm:dafha_centre}. In \cref{sec:DAFHA_poly_rep_phi}, we define polynomial representations for $\HH_n(A,z,q)$ and its subalgebra $H_n^\aff(A,z)$ as pullbacks of existing polynomial representations, and in \cref{sec:DAFHA_poly_rep_phi_A} we define $A$-decorated polynomial representations of $\HH_n(A,z,q)$ and its subalgebra $H_n^\aff(A,z)$ as induced representations, at $z=0$. In \cref{sec:diagrammatic_interpretation}, we give a brief overview of how $\HH_n(A,z,q)$ can be interpreted diagrammatically, and in \cref{sec:symmetries_for_DAFHA}, this interpretation is used to deduce some symmetries of the double affine Frobenius Hecke algebra in \cref{thm:DAFHA_geometric_isomorphisms}, as well as a $B_3$ braid group action on $\HH_n(A,z,q)$ by Dehn twists in \cref{thm:B_3_action}. The $B_3$ action is shown to factor, modulo inner automorphisms, to an $\SL_2(\ZZ)$ action in \cref{thm:SL_action}, and an example relating this action to the Frobenius structure is given in \cref{ex:combining_modular_actions}. Finally, in \cref{sec:GLn_to_SLn_DAFHA} we define the $\SL_n$ double affine Frobenius Hecke algebra, giving a procedure for obtaining it from the $\GL_n$ version. This yields two presentations of $\HH_{\SL_n}(A,z,q)$ (\cref{thm:double_affine_Frobenius_Hecke_algebra_SLn,thm:DAFHA_bernstein_presentation_SLn}) and a basis (\cref{thm:DAFHA_SLn_bases}).

\subsection*{Further directions}

In an upcoming paper, we will describe the $z \to 0$, $q \to 1$ degeneration process of the superalgebra $\HH_n(A,z,q)$ and obtain a trigonometric Frobenius double affine Hecke algebra that coincides, for certain choices of $A$, with an algebra defined to study Hilbert schemes of points in \cite{HSHACSS_2003}. The further rational degeneration yields a Frobenius rational Cherednik algebra, which generalises an algebra used to construct representations of Gieseker varieties in \cite{RMSQGV_2021}. Both $\GL_n$ and $\SL_n$ versions will be defined, as in \cref{dia:introduction}.

The diagrammatic interpretation from \cref{sec:diagrammatic_interpretation} is used in this paper to describe symmetries of $\HH_n(A,z,q)$, and the resultant $\SL_2(\ZZ)$ action of \cref{thm:SL_action} suggests that constructions based on the $\SL_2(\ZZ)$-symmetry of the DAHA can be extended to the Frobenius generalisation. For example, this suggests the possibility of developing Frobenius analogues of constructions involving DAHA--Jones invariants, torus knots and tangles, and quantum integrable systems. Moreover, the diagrammatic framework can help make clear which properties of Cherednik's DAHA are likely to extend to the double affine Frobenius Hecke algebra.

\iftoggle{detailsnote}{
\subsection*{Hidden details}
For the interested reader, the LaTeX source file of the arXiv version of this paper includes hidden details of some straightforward computations and arguments that are omitted in the PDF file. These can be displayed by toggling \texttt{details} to true in the source file and recompiling.
}{}

\subsection*{Use of AI tools}

The author used OpenAI's GPT-6 Astra and Anthropic's Claude Opus 5.5 to assist with proofreading the manuscript, to obtain feedback on possible omissions or errors in the proofs, and to suggest formatting improvements.  AI was not used to source or generate any citations. All mathematical content was both written and checked by the author. The author takes full responsibility for the content of this paper.

\subsection*{Acknowledgements}

The author would like to thank Alistair Savage for his guidance during this project, and for always taking time to provide valuable feedback and support. The author would also like to thank José Simental for helpful conversations and for pointing out some useful papers. This research was supported by the Natural Sciences and Engineering Research Council of Canada (NSERC), funding reference number RGPIN-2023-03842.

\section{The Double Affine Frobenius Hecke Algebra \texorpdfstring{$\HH_n(A,z,q)$}{ℍₙ(A,z,q)}} \label{sec:DAFHA_presentations}

For $n \geq 2$, denote the extended affine Weyl group for the $\GL_n$ root datum as follows, where the indices $i,j$ range in $0 \leq i,j \leq n-1$, with the convention $s_n=s_0$.
\begin{align*}
W_{\GL_n}^{\extaffine}
&= \left\langle s_0, \dotsc, s_{n-1}, \pi \,\middle|\,
\begin{aligned}
&s_i^2=1,\; \pi s_i\pi^{-1}=s_{i+1}, \, &&\\
&s_is_j=s_js_i, && (j \not\equiv {i\!\pm\! 1} \; \bmod n)\\
&s_is_{i+1}s_i=s_{i+1}s_is_{i+1}, \, && (1 \leq i\leq n-2)
\end{aligned}
\right\rangle.
\end{align*}
For the $\SL_n$ root datum, define $W_{\SL_n}^{\extaffine}=W_{\GL_n}^{\extaffine}/\langle \pi^n=1 \rangle$.
For $n \geq 3$, the defining relations of $W_{\GL_n}^{\extaffine}$ and $W_{\SL_n}^{\extaffine}$ imply $s_0 s_1 s_0 = s_1 s_0 s_1$ and $s_{n-1}s_0s_{n-1} = s_0s_{n-1}s_0$, whereas for $n=2$, no braid relation is imposed. The symmetric group $S_n$ is isomorphic to the subgroup of $W_{\GL_n}^{\extaffine}$ or of $W_{\SL_n}^{\extaffine}$ generated by $s_1,\dotsc,s_{n-1}$, and the affine symmetric group $S_n^{\affine}$ is isomorphic to the subgroup of $W_{\GL_n}^{\extaffine}$ or of $W_{\SL_n}^{\extaffine}$ generated by $s_0,s_1,\dotsc,s_{n-1}$. Denoting the image of $\pi$ in $W_{\SL_n}^{\extaffine}$ by $\pi_{\SL}$, we have $W_{\GL_n}^{\extaffine} \cong \langle \pi \rangle \ltimes S_n^{\affine} \cong \ZZ \ltimes S_n^{\affine}$ and $W_{\SL_n}^{\extaffine} \cong  \langle \pi_{\SL} \rangle  \ltimes S_n^{\affine} \cong \ZZ/n\ZZ \ltimes S_n^{\affine}$.

Throughout this document, let $\kk$ be a commutative ring with characteristic not equal to $2$, and fix a supersymmetric Frobenius $\kk$-superalgebra $A$ that is free over $\kk$, with a parity-preserving trace map $\tr \colon A \to \kk$. That is, for homogeneous $a,b \in A$ we have $\tr(a b) = (-1)^{\bar{a}\bar{b}}\tr(ba)$, and the map
\[
A \to \textup{Hom}_\kk(A,\kk), \qquad a \mapsto (b \mapsto \tr(ab)),
\]
is a parity-preserving $\kk$-module isomorphism. Fix a basis $B$ of $A$ and let $\{b^\vee : b \in B\}$ denote the left dual basis with respect to $\tr$, which satisfies $\tr(b^\vee c) = \delta_{b,c}$ for $b,c \in B$. Define the \emph{teleporters}
\[
t_{i,j}:= \sum_{b \in B}b_ib_j^\vee \in A^{\otimes n}, \quad 1 \leq i,j \leq n,
\]
for $b_i=1^{\otimes i-1} \otimes b \otimes 1^{\otimes n-i} \in A^{\otimes n}$. It can be verified that $t_{i,j}$ is independent of the choice of the basis $B$. Throughout this document, let $S_n$ act on $A^{\otimes n}$ by superpermutation,
\begin{equation} \label{eq:superpermute_A}
s_i(a_1 \otimes \dotsb \otimes a_n) = (-1)^{\bar{a_i}\bar{a}_{i+1}}a_1 \otimes \dotsb \otimes a_{i-1} \otimes a_{i+1} \otimes a_i \otimes a_{i+2} \otimes \dotsb \otimes a_n,
\end{equation}
for homogeneous $a_1, \dotsc ,a_n \in A$. Extend this action to $S_n^\textup{a}$ and $W_{\GL_n}^{\extaffine}$ by letting $s_0$ act by $(1 \; n) \in S_n$ and $\pi$ act by $\sigma = (1 \; 2 \; \dotsb \; n) \in S_n$.

\begin{definition} \label{thm:double_affine_Frobenius_Hecke_algebra}
Let $z \in \kk$, $q \in \kk^\times$ and $n \geq 2$. The \emph{double affine Frobenius Hecke algebra} $\HH_n(A,z,q)$ is the quotient of the $\kk$-superalgebra
\[
A^{\otimes n} \ast \kk [X_1^{\pm 1}, \dotsc ,X_n^{\pm 1}] \ast \langle \pi^{\pm 1} \rangle  \ast \langle T_i \colon 0 \leq i \leq n-1 \rangle,
\]
by the following relations, where the indices $i,j$ range over $0 \leq i \leq n-1$ and $1 \leq j \leq n$, with the convention $t_{0,1}=t_{n,1}$, $T_0=T_n$, $X_0=X_n$:
\begin{align}
T_iT_j&=T_jT_i, && 0 \leq j \leq n-1,\; j \not\equiv i \pm 1 \bmod n, \label{dafha_rel_1}\\
T_iT_{i+1}T_i&=T_{i+1}T_iT_{i+1}, && 1 \leq i \leq n-2, \label{dafha_rel_2}\\
T_i^2&=zt_{i,i+1}T_i+1, && \label{dafha_rel_3}\\
T_i\mathbf{a}&=s_i(\mathbf{a})T_i, && \mathbf{a} \in A^{\otimes n}, \label{dafha_rel_4}\\
\pi T_i \pi^{-1} &= T_{i+1}, && \label{dafha_rel_5}\\
\pi \mathbf{a} &= \sigma(\mathbf{a}) \pi, && \mathbf{a} \in A^{\otimes n}, \label{dafha_rel_6}\\
T_iX_j&=X_jT_i, && j \not\equiv i, {i+1} \bmod n, \label{dafha_rel_7}\\
T_iX_iT_i&=X_{i+1}, && 1 \leq i \leq n-1, \label{dafha_rel_8}\\
T_0X_nT_0&=q^{-1}X_1, &&\label{dafha_rel_8b}\\
X_j\mathbf{a}&=\mathbf{a}X_j, && \mathbf{a} \in A^{\otimes n}, \label{dafha_rel_9}\\
\pi X_j \pi^{-1} &= X_{j+1}, && 1 \leq j \leq n-1,\label{dafha_rel_10}\\
\pi X_n \pi^{-1} &= q^{-1}X_1. &&\label{dafha_rel_10b}
\end{align}
In the above, we define $\sigma = (1 \; 2 \; \dotsb \; n) \in S_n$. For any permutation $w \in S_n$ with reduced decomposition $w=s_{i_1} \dotsb s_{i_k}$ into simple transpositions, define $T_w:=T_{i_1} \dotsb T_{i_k}$. This is independent of the reduced expression by Matsumoto's Theorem, as the generators $T_i$ satisfy the braid relations. We have $T_i^{-1}=T_i-zt_{i,i+1}$ for all $0 \leq i \leq n-1$. All generators not in $A^{\otimes n}$ have parity $0$.
\end{definition}

Note that for $n=2$, Relation \eqref{dafha_rel_2} is empty, so there is no braid relation between $T_0$ and $T_1$. For $n \geq 3$, we can deduce
\[
T_0T_1T_0 \overset{\eqref{dafha_rel_5}}{=} \pi^{-1}T_1T_2T_1\pi \overset{\eqref{dafha_rel_2}}{=} \pi^{-1}T_2T_1T_2\pi \overset{\eqref{dafha_rel_5}}{=} T_1T_0T_1.
\]
Thus for $n \geq 3$, given a reduced decomposition $w^\textup{a} = s_{i_1} \dotsb s_{i_k}\in S_n^{\affine}$, Matsumoto's Theorem implies that $T_{w^{\affine}}:=T_{i_1} \dotsb T_{i_k} \in \HH_n(A,z,q)$ is well-defined.  For $n=2$, reduced words in $S_2^{\affine}$ are unique, so the same convention can be used. Given $w^{\extaffine} = \pi^r w^\textup{a} \in W_{\GL_n}^{\extaffine}=\langle \pi \rangle \ltimes S_n^{\affine}$, let $T_{w^{\extaffine}} = \pi^r T_{w^{\affine}} \in \HH_n(A,z,q)$.

\begin{example} \label{ex:Cherednik_DAHA}
Let $\kk=\CC$, $k \in \QQ$, $q \in \CC^\times$, $t=q^k$ and $z=t^{\tfrac{1}{2}}-t^{-\tfrac{1}{2}}$. Then $\HH_n(\CC,z,q)$ is precisely the $\GL_n$ double affine Hecke algebra given in \cite[(3.7.1)]{DAHA_2005}, with identical notation. The generator $T_0$ is omitted in Cherednik's presentation, but can be recovered by $T_0 := \pi T_{n-1} \pi^{-1}$.
\end{example}

\begin{proposition} \label{thm:DAFHA_bernstein_presentation}
Let $z \in \kk$, $q \in \kk^\times$ and $n \geq 2$. The \emph{double affine Frobenius Hecke algebra} $\HH_n(A,z,q)$ has a presentation as the quotient $\HH_n(A,z,q)'$ of the free product
\[
A^{\otimes n} \ast \kk [X_1^{\pm 1}, \dotsc ,X_n^{\pm 1}] \ast \kk [Y_1^{\pm 1}, \dotsc ,Y_n^{\pm 1}] \ast \langle T_i \colon 1 \leq i \leq n-1 \rangle
\]
by the following relations, where the indices $i,j$ range over $1 \leq i \leq n-1$ and $1 \leq j \leq n$:
\begin{align}
T_iT_j&=T_jT_i, && 1 \leq j \leq n-1,\; |i-j|>1, \label{qwa_rel_1}\\
T_iT_{i+1}T_i&=T_{i+1}T_iT_{i+1}, && 1 \leq i \leq n-2, \label{qwa_rel_2}\\
T_i^2&=zt_{i,i+1}T_i+1, && \label{qwa_rel_3}\\
T_i\mathbf{a}&=s_i(\mathbf{a})T_i, && \mathbf{a} \in A^{\otimes n}, \label{qwa_rel_4}\\
T_iX_j=X_jT_i &, \quad T_iY_j=Y_jT_i, && j \notin \{i,i+1\}, \label{daqwa_sym_rel_1}\\
T_iX_iT_i=X_{i+1} &, \quad T_iY_{i+1}T_i=Y_i, && \label{daqwa_sym_rel_2}\\
X_j\mathbf{a}=\mathbf{a}X_j &, \quad Y_j\mathbf{a}=\mathbf{a}Y_j, && \mathbf{a} \in A^{\otimes n}, \label{daqwa_sym_rel_3}\\
X_1^{-1}Y_2X_1Y_2^{-1}&=T_1^2, && \label{daqwa_sym_rel_4}\\
Y_1X_1 \dotsb X_n &= q^{-1}X_1 \dotsb X_nY_1, && \label{daqwa_sym_rel_5}
\end{align}
via the map $\alpha \colon \HH_n(A,z,q) \to \HH_n(A,z,q)'$ sending $X_1^{\pm 1}, \dotsc, X_n^{\pm 1}$, $T_1^{\pm 1}, \dotsc, T_{n-1}^{\pm 1}$ and $\mathbf{a} \in A^{\otimes n}$ to themselves, and sending
\begin{equation}
\pi \mapsto Y_1 T_1^{-1} \dotsb T_{n-1}^{-1} \quad \text{ and } \quad T_0 \mapsto Y_n^{-1}Y_1 T_1^{-1} \dotsb T_{n-1}^{-1} \dotsb T_1^{-1}, \label{eq:pi}
\end{equation}
where all generators not in $A^{\otimes n}$ have parity $0$. The inverse is a $\kk$-superalgebra homomorphism ${\alpha' \colon \HH_n(A,z,q)' \to \HH_n(A,z,q)}$ sending
\[
Y_i \mapsto T_{i-1}^{-1} \dotsb T_1^{-1}\pi T_{n-1} \dotsb T_{i}.
\]
\end{proposition}
\begin{proof}
We begin by verifying that $\alpha$ is well-defined. Only relations of $\HH_n(A,z,q)$ involving $\pi$ or $T_0$ need be checked, as the rest are trivially true. We verify that Relation \eqref{dafha_rel_5} for $0 \leq i \leq n-1$ is preserved by $\alpha$ in three cases. For $1 \leq i \leq n-2$, using Relation \labelcref{qwa_rel_2} rewritten as $T_i^{-1}T_{i+1}^{-1}T_iT_{i+1}T_i=T_{i+1}$, we have
\begin{multline*}
\alpha(\pi)\alpha(T_i)\alpha(\pi)^{-1} = Y_1T_1^{-1} \dotsb T_{n-1}^{-1}T_iT_{n-1} \dotsb T_1Y_1^{-1}\\ \overset{\labelcref{qwa_rel_1}}{=} Y_1T_1^{-1} \dotsb T_{i-1}^{-1}T_i^{-1}T_{i+1}^{-1}T_iT_{i+1}T_iT_{i-1} \dotsb T_1Y_1^{-1}\\
\overset{\labelcref{qwa_rel_2}}{=} Y_1T_1^{-1} \dotsb T_{i-1}^{-1}T_{i+1}T_{i-1} \dotsb T_1Y_1^{-1} \underset{\labelcref{daqwa_sym_rel_1}}{\overset{\labelcref{qwa_rel_1}}{=}} T_{i+1} = \alpha(T_{i+1}).
\end{multline*}
Next, we have
\begin{multline*}
\alpha(\pi)\alpha(T_{n-1})\alpha(\pi)^{-1} = Y_1T_1^{-1} \dotsb T_{n-1}^{-1}T_{n-1}T_{n-1} \dotsb T_1Y_1^{-1}\twelvebreak
= Y_1T_1^{-1} \dotsb T_{n-2}^{-1} T_{n-1} T_{n-2} \dotsb T_1Y_1^{-1}\\
\overset{\labelcref{daqwa_sym_rel_2}}{=} Y_1T_1^{-1} \dotsb T_{n-2}^{-1} Y_n^{-1}T_{n-1}^{-1} T_{n-2}^{-1} \dotsb T_1^{-1}\twelvebreak
\overset{\labelcref{daqwa_sym_rel_1}}{=} Y_n^{-1}Y_1T_1^{-1} \dotsb T_{n-2}^{-1}T_{n-1}^{-1} T_{n-2}^{-1} \dotsb T_1^{-1} = \alpha(T_0).
\end{multline*}
For the $i=0$ case, we prove the equivalent relation
\begin{multline*}
\alpha(\pi)^{-1}\alpha(T_1)\alpha(\pi) = (T_{n-1} \dotsb T_1Y_1^{-1})\,T_1\,(Y_1T_1^{-1} \dotsb T_{n-1}^{-1})\\
\overset{\labelcref{daqwa_sym_rel_2}}{=} Y_n^{-1}(T_{n-1}^{-1} \dotsb T_2^{-1} T_1^{-1}) T_1(Y_1 T_1^{-1} \dotsb T_{n-1}^{-1})\pjmbreak
= Y_n^{-1}(T_{n-1}^{-1} \dotsb T_2^{-1} )(Y_1 T_1^{-1} \dotsb T_{n-1}^{-1})\\
\overset{\labelcref{daqwa_sym_rel_1}}{=} Y_n^{-1}Y_1(T_{n-1}^{-1} \dotsb T_2^{-1})( T_1^{-1} \dotsb T_{n-1}^{-1})\pjmbreak
=Y_n^{-1}Y_1(T_1^{-1} \dotsb T_{n-1}^{-1} \dotsb T_1^{-1})=\alpha(T_0),
\end{multline*}
where the penultimate equality uses
\[
T_1^{-1} \dotsb T_{n-1}^{-1} \dotsb T_1^{-1} =\allowbreak T_{n-1}^{-1} \dotsb T_1^{-1} \dotsb T_{n-1}^{-1},
\]
which follows from Matsumoto's theorem, as $(1 \; n) \in S_n$ can be written in a reduced expression as both $s_1 \dotsb s_{n-1} \dotsb s_1$ and $s_{n-1} \dotsb s_1 \dotsb s_{n-1}$. We verify Relation \eqref{dafha_rel_6} by
\[
\alpha(\pi)\alpha(\mathbf{a})=Y_1 T_1^{-1} \dotsb T_{n-1}^{-1}\mathbf{a} \underset{\eqref{daqwa_sym_rel_3}}{\overset{\eqref{qwa_rel_4}}{=}} \sigma(\mathbf{a})Y_1 T_1^{-1} \dotsb T_{n-1}^{-1} = \alpha(\sigma(\mathbf{a}))\alpha(\pi).
\]
The remaining relations involving $T_0$, namely the $i=0$ cases of Relations \eqref{dafha_rel_1}, \eqref{dafha_rel_3}, \eqref{dafha_rel_4} and \eqref{dafha_rel_7}, may be proved by conjugating their $i=1$ analogues by $\alpha(\pi)$, since Relations \eqref{dafha_rel_5} and \eqref{dafha_rel_6} have already been verified. For example, if $2 \leq j \leq n-2$,
\[
\begin{multlined}
\alpha(T_0)\alpha(T_j) = \alpha(\pi)^{-1}\alpha(T_1)\alpha(T_{j+1})\alpha(\pi)\pjmbreak
\overset{\eqref{qwa_rel_1}}{=} \alpha(\pi)^{-1}\alpha(T_{j+1})\alpha(T_1)\alpha(\pi) = \alpha(T_j)\alpha(T_0).
\end{multlined}
\]
To verify Relation \eqref{dafha_rel_10}, we check it for $i=1$,
\begin{multline*}
\alpha(\pi)\alpha(X_1)\alpha(\pi)^{-1}=(Y_1 T_1^{-1} \dotsb T_{n-1}^{-1})X_1(T_{n-1} \dotsb T_1 Y_1^{-1})\\
\overset{\eqref{daqwa_sym_rel_1}}{=} (Y_1 T_1^{-1} X_1 T_2^{-1} \dotsb T_{n-1}^{-1})(T_{n-1} \dotsb T_1 Y_1^{-1}) = Y_1 T_1^{-1} X_1 T_1 Y_1^{-1}\\
\overset{\eqref{daqwa_sym_rel_2}}{=} T_1Y_2X_1Y_2^{-1}T_1^{-1} \overset{\eqref{daqwa_sym_rel_4}}{=} T_1(X_1T_1^2)T_1^{-1}
= T_1X_1T_1 \overset{\eqref{daqwa_sym_rel_2}}{=} X_2 = \alpha(X_2),
\end{multline*}
and use the already verified Relations \eqref{dafha_rel_5} and \eqref{dafha_rel_8} to prove \eqref{dafha_rel_10} for any $1 \leq i \leq n-1$:
\begin{multline*}
\alpha(\pi)\alpha(X_i)\alpha(\pi)^{-1}\pjmbreak
= \alpha(\pi)\alpha(T_{i-1}) \dotsb \alpha(T_1)\alpha(X_1)\alpha(T_1)\dotsb \alpha(T_{i-1}) \alpha(\pi)^{-1}\\
\overset{\eqref{dafha_rel_5}}{=} \alpha(T_i) \dotsb \alpha(T_2)\alpha(\pi)\alpha(X_1)\alpha(\pi)^{-1}\alpha(T_2)\dotsb \alpha(T_i)\\
\overset{\eqref{dafha_rel_10}}{=}\alpha(T_i) \dotsb \alpha(T_2)\alpha(X_2)\alpha(T_2)\dotsb \alpha(T_i)
\overset{\eqref{dafha_rel_8}}{=} \alpha(X_{i+1}).
\end{multline*}
To verify Relation \eqref{dafha_rel_10b}, denote $\tilde{X}:=X_1\dotsb X_n \in \HH_n(A,z,q)'$. Note that since $T_iX_iX_{i+1}=T_iX_i(T_iX_iT_i)=(T_iX_iT_i)X_iT_i=X_{i+1}X_iT_i$ by Relation \eqref{daqwa_sym_rel_2}, and $T_iX_j=X_jT_i$ for all $j \notin \{i,i+1\}$ by Relation \eqref{daqwa_sym_rel_1}, the element $\tilde{X}$ commutes with all $T_i$, and likewise, with all $T_i^{-1}$, for $1 \leq i \leq n-1$. Thus,
\begin{multline*}
X_2 \dotsb X_n \alpha(\pi)\alpha(X_n)\alpha(\pi)^{-1}\pjmbreak
= \alpha(\pi)\alpha(X_1)\alpha(\pi)^{-1} \dotsb \alpha(\pi)X_{n-1}\alpha(\pi)^{-1} \alpha(\pi)\alpha(X_n)\alpha(\pi)^{-1}\\
= \alpha(\pi)\alpha(X_1)\dotsb \alpha(X_n)\alpha(\pi)^{-1} = (Y_1 T_1^{-1} \dotsb T_{n-1}^{-1})\tilde{X}(T_{n-1} \dotsb T_1 Y_1^{-1})\\
 = (Y_1 \tilde{X} T_1^{-1} \dotsb T_{n-1}^{-1})(T_{n-1} \dotsb T_1 Y_1^{-1}) = Y_1 \tilde{X}Y_1^{-1} \overset{\eqref{daqwa_sym_rel_5}}{=} q^{-1}\tilde{X} Y_1 Y_1^{-1}\pjmbreak
= q^{-1} X_1 \dotsb X_n.
\end{multline*}
Multiplying on the left by $X_2^{-1} \dotsb X_n^{-1}$ yields the required Relation \eqref{dafha_rel_10b},
\[
\alpha(\pi)\alpha(X_n)\alpha(\pi)^{-1}=q^{-1} X_1 = q^{-1}\alpha(X_1).
\]
This can then be used to check Relation \eqref{dafha_rel_8b},
\begin{multline*}
\alpha(T_0)\alpha(X_n)\alpha(T_0)\pjmbreak
\underset{\eqref{dafha_rel_5}}{\overset{\eqref{dafha_rel_10}}{=}} \alpha(\pi)\alpha(T_{n-1})\alpha(\pi)^{-1}\alpha(\pi)\alpha(X_{n-1})\alpha(\pi)^{-1}\alpha(\pi)\alpha(T_{n-1})\alpha(\pi)^{-1}\\
= \alpha(\pi)\alpha(T_{n-1})\alpha(X_{n-1})\alpha(T_{n-1})\alpha(\pi)^{-1} = \alpha(\pi)T_{n-1}X_{n-1}T_{n-1}\alpha(\pi)^{-1}\\
\overset{\eqref{daqwa_sym_rel_2}}{=} \alpha(\pi)X_n\alpha(\pi)^{-1} = \alpha(\pi)\alpha(X_n)\alpha(\pi)^{-1} \overset{\eqref{dafha_rel_10b}}{=} q^{-1}\alpha(X_1).
\end{multline*}
We have verified all relations necessary to ensure that $\alpha$ is a well-defined  homomorphism of $\kk$-superalgebras. It is clearly parity-preserving. We now verify that $\alpha'$ is well-defined by ensuring the generating relations of $\HH_n(A,z,q)'$ hold on the images of the generators under $\alpha'$. Only relations of $\HH_n(A,z,q)'$ involving $Y_i$ for $1 \leq i \leq n$ need be checked, as the rest are trivially true. Firstly, $\alpha'(Y_i)$ is clearly invertible for each $1 \leq i \leq n$. Relation \eqref{daqwa_sym_rel_1} is shown in two cases, for $1 \leq i \leq n-1$ and $j \notin \{i,i+1\}$. If $1 \leq i < j-1$, then
\begin{multline*}
\alpha'(T_i)\alpha'(Y_j) = T_i(T_{j-1}^{-1} \dotsb T_1^{-1} \pi T_{n-1} \dotsb T_j)\pjmbreak \overset{\eqref{dafha_rel_2}}{=} T_{j-1}^{-1} \dotsb T_1^{-1} T_{i+1}\pi T_{n-1} \dotsb T_j\\ \overset{\eqref{dafha_rel_5}}{=} T_{j-1}^{-1} \dotsb T_1^{-1} \pi T_i T_{n-1} \dotsb T_j\pjmbreak
\overset{\eqref{dafha_rel_1}}{=} T_{j-1}^{-1} \dotsb T_1^{-1} \pi T_{n-1} \dotsb T_j T_i = \alpha'(Y_j)\alpha'(T_i),
\end{multline*}
In the above, Relation \eqref{dafha_rel_2} is used in the form $T_i(T_{i+1}^{-1}T_i^{-1})=(T_{i+1}^{-1}T_i^{-1})T_{i+1}$. If $j < i \leq n-1$, then
\begin{multline*}
\alpha'(T_i)\alpha'(Y_j) = T_i(T_{j-1}^{-1} \dotsb T_1^{-1} \pi T_{n-1} \dotsb T_j)\pjmbreak \overset{\eqref{dafha_rel_1}}{=} T_{j-1}^{-1} \dotsb T_1^{-1} T_i\pi T_{n-1} \dotsb T_j\widebreak \overset{\eqref{dafha_rel_5}}{=} T_{j-1}^{-1} \dotsb T_1^{-1} \pi T_{i-1} T_{n-1} \dotsb T_j\pjmbreak \overset{\eqref{dafha_rel_2}}{=} T_{j-1}^{-1} \dotsb T_1^{-1} \pi T_{n-1} \dotsb T_j T_i = \alpha'(Y_j)\alpha'(T_i).
\end{multline*}
Relation \eqref{daqwa_sym_rel_2} is easily verified, as for $1 \leq i \leq n-1$,
\[
\begin{multlined}
\alpha'(T_i)\alpha'(Y_{i+1})\alpha'(T_i)=T_i(T_i^{-1} \dotsb T_1^{-1} \pi T_{n-1} \dotsb T_{i+1})T_i\pjmbreak
= T_{i-1}^{-1} \dotsb T_1^{-1} \pi T_{n-1} \dotsb T_i = \alpha'(Y_i).
\end{multlined}
\]
Relation \eqref{daqwa_sym_rel_3} holds, as for $1 \leq i \leq n$,
\begin{multline*}
\alpha'(Y_i)\alpha'(\mathbf{a}) = (T_{i-1}^{-1} \dotsb T_1^{-1} \pi T_{n-1} \dotsb T_i) \mathbf{a}\\
= (s_{i-1} \circ \dotsb \circ s_1 \circ \sigma \circ s_{n-1} \circ \dotsb \circ s_i)(\mathbf{a})(T_{i-1}^{-1} \dotsb T_1^{-1} \pi T_{n-1} \dotsb T_i)\\
= \mathbf{a} \alpha'(Y_i) = \alpha'(\mathbf{a}) \alpha'(Y_i),
\end{multline*}
as $s_{i-1} \circ \dotsb \circ s_1 \circ \sigma \circ s_{n-1} \circ \dotsb \circ s_i = \textup{id} \in S_n$. Relation \eqref{daqwa_sym_rel_4} holds in the image of $\alpha'$:
\begin{multline*}
\alpha'(X_1)^{-1}\alpha'(Y_2)\alpha'(X_1)\alpha'(Y_2)^{-1}\pjmbreak
= X_1^{-1} (T_1^{-1} \pi T_{n-1} \dotsb T_2) X_1 (T_2^{-1} \dotsb T_{n-1}^{-1} \pi^{-1} T_1) \widebreak\overset{\eqref{dafha_rel_7}}{=} X_1^{-1} T_1^{-1} \pi X_1 \pi^{-1} T_1\pjmbreak
\overset{\eqref{dafha_rel_10}}{=} X_1^{-1} T_1^{-1} X_2 T_1 \overset{\eqref{dafha_rel_8}}{=} T_1^2 = \alpha'(T_1)^2.
\end{multline*}
Relation \eqref{daqwa_sym_rel_5} holds. Using once again that $\tilde{X}=X_1 \dotsb X_n$ commutes with all $T_i$, we have
\begin{multline*}
\alpha'(Y_1)\alpha'(X_1)\dotsb \alpha'(X_n) = (\pi T_{n-1} \dotsb T_1) \tilde{X} = \pi \tilde{X} T_{n-1} \dotsb T_1\\
\underset{\eqref{dafha_rel_10}}{\overset{\eqref{dafha_rel_10b}}{=}} q^{-1}\tilde{X}(\pi T_{n-1} \dotsb T_1) = q^{-1}\alpha'(X_1)\dotsb \alpha'(X_n) \alpha'(Y_1).
\end{multline*}
One can use the relations proved thus far to show that the elements $\alpha'(Y_i)$ for $1 \leq i \leq n$ are pairwise commutative, by proving that $\alpha'(Y_1)\alpha'(Y_2)=\alpha'(Y_2)\alpha'(Y_1)$ and using Relations \eqref{daqwa_sym_rel_1} and \eqref{daqwa_sym_rel_2} to show the rest. We conclude that $\alpha'$ is also a well-defined $\kk$-superalgebra homomorphism. It is the inverse of $\alpha$, as for $1 \leq i \leq n$, we have
\begin{align*}
\alpha(\alpha'(Y_i))&=\alpha(T_{i-1}^{-1} \dotsb T_1^{-1} \pi T_{n-1} \dotsb T_i)\pjmbr
= T_{i-1}^{-1} \dotsb T_1^{-1} (Y_1 T_1^{-1} \dotsb T_{n-1}^{-1}) T_{n-1} \dotsb T_i\\
&= T_{i-1}^{-1} \dotsb T_1^{-1}Y_1T_1^{-1} \dotsb T_{i-1}^{-1} \overset{\eqref{daqwa_sym_rel_2}}{=} Y_i,\\
\alpha'(\alpha(T_0))&=\alpha'(\alpha(\pi^{-1} T_1 \pi)) = \alpha'(\alpha(\pi))^{-1}T_1\alpha'(\alpha(\pi)) = \pi^{-1} T_1 \pi = T_0,\\
\alpha'(\alpha(\pi))&= \alpha'(Y_1 T_1^{-1} \dotsb T_{n-1}^{-1}) = (\pi T_{n-1} \dotsb T_1) T_1^{-1} \dotsb T_{n-1}^{-1} = \pi. \qedhere
\end{align*}
\end{proof}

Henceforth we identify $\HH_n(A,z,q)'$ and $\HH_n(A,z,q)$ by the isomorphism in \cref{thm:DAFHA_bernstein_presentation}, with the identification
\begin{equation} \label{eq:Y_i_formula}
Y_i = T_{i-1}^{-1} \dotsb T_1^{-1}\pi T_{n-1} \dotsb T_{i},
\end{equation}
for $1 \leq i \leq n$.

\begin{lemma} \label{thm:pi_n}
In $\HH_n(A,z,q)$, we have $Y_1 \dotsb Y_n=\pi^n$.
\end{lemma}
\begin{proof}
This is a standard DAHA relation. It is not difficult to prove the identity
\begin{equation} \label{eq:pi_n}
Y_1 \dotsb Y_k=\pi^k (T_{n-k} \dotsb T_1) (T_{n-k+1} \dotsb T_2) \dotsb (T_{n-2} \dotsb T_{k-1})(T_{n-1} \dotsb T_k)
\end{equation}
inductively for all $1 \leq k \leq n$ using Relations \eqref{dafha_rel_5} and \eqref{dafha_rel_1}, and taking $k=n$ concludes the proof. The base case $Y_1=\pi T_{n-1} \dotsb T_1$ is simply the identity \eqref{eq:Y_i_formula}.\details{ Assuming \eqref{eq:pi_n} holds for some $1 \leq k < n$, we have
\begin{align*}
Y_1 \dotsb Y_k Y_{k+1} &\overset{\eqref{eq:Y_i_formula}}{=} \pi^k (T_{n-k} \dotsb T_1) (T_{n-k+1} \dotsb T_2) \\
&\qquad \dotsb (T_{n-2} \dotsb T_{k-1})(T_{n-1} \dotsb T_k) (T_{k}^{-1} \dotsb T_1^{-1}\pi T_{n-1} \dotsb T_{k+1})\\
&\overset{\eqref{dafha_rel_5}}{=} \pi^{k+1} (T_{n-k-1} \dotsb T_0) (T_{n-k} \dotsb T_1)\\
&\qquad \dotsb (T_{n-3} \dotsb T_{k-2})(T_{n-2} \dotsb T_{k-1}) (T_{k-1}^{-1} \dotsb T_0^{-1} T_{n-1} \dotsb T_{k+1})\\
&= \pi^{k+1} (T_{n-k-1} \dotsb T_1)T_0(T_{n-k} \dotsb T_2)T_1\\
&\qquad \dotsb (T_{n-3} \dotsb T_{k-1})T_{k-2}(T_{n-2} \dotsb T_k)T_{k-1} (T_{k-1}^{-1} \dotsb T_0^{-1} T_{n-1} \dotsb T_{k+1})\\
&\overset{\eqref{dafha_rel_1}}{=}\pi^{k+1} (T_{n-k-1} \dotsb T_1)(T_{n-k} \dotsb T_2)\\
&\qquad \dotsb (T_{n-3} \dotsb T_{k-1})T_{k-1}(T_{n-2} \dotsb T_k)T_0 \dotsb T_{k-1} (T_{k-1}^{-1} \dotsb T_0^{-1} T_{n-1} \dotsb T_{k+1})\\
&=\pi^{k+1} (T_{n-k-1} \dotsb T_1)(T_{n-k} \dotsb T_2)\\
&\qquad \dotsb (T_{n-3} \dotsb T_{k-1})T_{k-1}(T_{n-2} \dotsb T_k)(T_{n-1} \dotsb T_{k+1}),
\end{align*}
concluding the induction step.}
\end{proof}

There are some additional relations that can be deduced from the relations in \cref{thm:DAFHA_bernstein_presentation}.

\begin{lemma} \label{thm:extra_twist_relations}
In $\HH_n(A,z,q)$, we have
\[
Y_iX_1 \dotsb X_n = q^{-1}X_1 \dotsb X_nY_i,
\]
for all $1 \leq i \leq n$, and
\[
X_jY_1 \dotsb Y_n = qY_1 \dotsb Y_n X_j,
\]
for all $1 \leq j \leq n$.
\end{lemma}
\begin{proof}
This is a standard DAHA relation. Since $\tilde{X}:=X_1\dotsb X_n \in \HH_n(A,z,q)$ commutes with $T_i$ for all $1 \leq i \leq n-1$ as shown in the proof of \cref{thm:DAFHA_bernstein_presentation}, and $\pi \tilde{X} = q^{-1}\tilde{X}\pi$ by Relations \eqref{dafha_rel_10} and \eqref{dafha_rel_10b}, we have
\[
\begin{multlined}
Y_i\tilde{X} \overset{\eqref{eq:Y_i_formula}}{=} (T_{i-1}^{-1} \dotsb T_1^{-1}\pi T_{n-1} \dotsb T_{i}) \tilde{X}\pjmbreak
\underset{\eqref{dafha_rel_10b}}{\overset{\eqref{dafha_rel_10}}{=}} q^{-1}\tilde{X} (T_{i-1}^{-1} \dotsb T_1^{-1}\pi T_{n-1} \dotsb T_{i}) = q^{-1}\tilde{X} Y_i,
\end{multlined}
\]
for any $1 \leq i \leq n$, proving the first identity. Relations \eqref{dafha_rel_10} and \eqref{dafha_rel_10b} give us $X_j\pi^n=q \pi^n X_j$, and since $\pi^n=Y_1 \dotsb Y_n$ by \cref{thm:pi_n}, we have $X_jY_1 \dotsb Y_n = qY_1 \dotsb Y_n X_j$ for all $1 \leq j \leq n$, proving the second identity.
\end{proof}


Recall that for a $\kk$-superalgebra $C$, the opposite superalgebra $C^{\opp}$ has the same underlying $\kk$-supermodule as $C$, and multiplication given for homogeneous $b,c \in C$ by $b^{\opp} c^{\opp} := (-1)^{\bar{b}\bar{c}}(cb)^{\opp}$.

\begin{lemma} \label{thm:XY_automorphisms}
There is an isomorphism of $\kk$-superalgebras defined by
\begin{equation} \label{eq:XY_antiautomorphism}
\begin{gathered}
\sigmaop \colon \HH_n(A,z,q) \to \HH_n(A^\textup{op},z,q)^\textup{op},\\
\mathbf{a} \mapsto \mathbf{a}^\textup{op}, \qquad T_i \mapsto T_i, \qquad X_j \mapsto Y_j^{-1}, \qquad Y_j \mapsto X_j^{-1},
\end{gathered}
\end{equation}
for $\mathbf{a} \in A^{\otimes n}$, $1 \leq i \leq n-1$ and $1 \leq j \leq n$, and an isomorphism of $\kk$-superalgebras defined by
\begin{equation} \label{eq:XY_automorphism}
\begin{gathered}
\epsilontransform \colon \HH_n(A,z,q) \to \HH_n(A,-z,q^{-1}),\\
\mathbf{a} \mapsto \mathbf{a}, \qquad T_i \mapsto T_i^{-1}, \qquad X_j \mapsto Y_j, \qquad Y_j \mapsto X_j,
\end{gathered}
\end{equation}
for $\mathbf{a} \in A^{\otimes n}$, $1 \leq i \leq n-1$ and $1 \leq j \leq n$.
\end{lemma}
\begin{proof}
Most relations are easy to check for both maps. In particular, Relation \eqref{daqwa_sym_rel_5} follows immediately from \cref{thm:extra_twist_relations}. For Relation \eqref{daqwa_sym_rel_4} under the map $\sigmaop$,
\begin{multline*}
X_2Y_1^{-1}X_2^{-1}Y_1=T_1T_1^{-1}(X_2Y_1^{-1}X_2^{-1}Y_1)T_1^{-1}T_1\pjmbreak
\overset{\eqref{daqwa_sym_rel_2}}{=} T_1(X_1Y_2^{-1}T_1^{-2}X_1^{-1}Y_2)T_1\\
\overset{\eqref{daqwa_sym_rel_4}}{=} T_1X_1Y_2^{-1}(Y_2X_1^{-1}Y_2^{-1}X_1)X_1^{-1}Y_2T_1 = T_1^2,
\end{multline*}
as required. For Relation \eqref{daqwa_sym_rel_4} under the map $\epsilontransform$,
\begin{multline*}
Y_1^{-1}X_2Y_1X_2^{-1}=T_1^{-1}T_1(Y_1^{-1}X_2Y_1X_2^{-1})T_1T_1^{-1}\pjmbreak
\overset{\eqref{daqwa_sym_rel_2}}{=} T_1^{-1}(Y_2^{-1}X_1T_1^2Y_2X_1^{-1})T_1^{-1} \\
\overset{\eqref{daqwa_sym_rel_4}}{=} T_1^{-1}(Y_2^{-1}X_1(X_1^{-1}Y_2X_1Y_2^{-1})Y_2X_1^{-1})T_1^{-1} = T_1^{-2},
\end{multline*}
as required. Bijectivity of $\sigmaop$ follows from noting that
\[
\HH_n(A,z,q) \overset{\sigmaop}{\to} \HH_n(A^\textup{op},z,q)^\textup{op} \overset{\sigmaop}{\to} (\HH_n((A^\textup{op})^{\opp},z,q)^\textup{op})^{\opp}=\HH_n(A,z,q)
\]
is the identity, and $\epsilontransform \colon \HH_n(A,-z,q^{-1}) \to \HH_n(A,z,q)$ provides its own inverse.
\end{proof}

\begin{example} \label{ex:double_affine_Wan_Wang}
Let $\kk$ have characteristic not equal to $2$ and let $\tau \in \kk^\times$. Let $A=\kk G$, for $G$ a finite group, with the trace map $\sum_{g \in G} C_g g \mapsto C_{e_G}$, and let $z=(\tau-\tau^{-1})$. The algebra $\HH_n(\kk G,(\tau-\tau^{-1}),q)$ is a double affine version of the non-degenerate version \cite[Introduction,~Example~(b)]{QAWA_2020} of the wreath Hecke algebra of Wan and Wang \cite[Def.~2.4]{MRBRWHA_2008}. This algebra seems not to have appeared in the current literature.
\end{example}

\begin{example} \label{ex:double_affine_Yokonuma_Hecke}
This example is a special case of \cref{ex:double_affine_Wan_Wang}, with a rescaled parameter. Let $\kk=\CC[\tau,\tau^{-1}]$, $A=\kk C_d$, for $C_d$ the cyclic group of order $d$, with trace as in \cref{ex:double_affine_Wan_Wang}, and let $z=(\tau-\tau^{-1})/d$. We call the algebra $\HH_n(\kk C_d,(\tau-\tau^{-1})/d,q)$ the \emph{double affine Yokonuma--Hecke algebra}, as it is a double affine version of the affine Yokonuma--Hecke algebra from \cite[Sec.~3.1]{RTYHA_2014}. This algebra seems not to have appeared in the current literature.
\end{example}

\begin{definition} \label{def:AFHAs}
Let $z \in \kk$. The \emph{affine Frobenius Hecke algebras} are the subalgebras
\begin{align*}
H_n^\aff(A,z) &:= \langle \mathbf{a}, Y_1^{\pm 1}, \dotsc, Y_n^{\pm 1}, T_1^{\pm 1}, \dotsc, T_{n-1}^{\pm 1}\rangle\\
&= \langle \mathbf{a}, \pi^{\pm 1}, T_0^{\pm 1}, \dotsc, T_{n-1}^{\pm 1}\rangle,\\
H_n^\affX(A,z) &:= \langle \mathbf{a}, X_1^{\pm 1}, \dotsc, X_n^{\pm 1}, T_1^{\pm 1}, \dotsc, T_{n-1}^{\pm 1}\rangle,
\end{align*}
of $\HH_n(A,z,q)$, where $\mathbf{a} \in A^{\otimes n}$. The \emph{Frobenius Hecke algebra} is the subalgebra
\[
H_n(A,z) := \langle \mathbf{a}, T_1^{\pm 1}, \dotsc, T_{n-1}^{\pm 1}\rangle,
\]
of $\HH_n(A,z,q)$, where $\mathbf{a} \in A^{\otimes n}$.
\end{definition}

The quantum wreath algebra of \cite[Def.~2.1]{QAWA_2020} clearly surjects onto $H_n(A,z)$ as a $\kk$-superalgebra by $\mathbf{a} \mapsto \mathbf{a}$ and $T_i \mapsto T_i$. The quantum affine wreath algebra of \cite[Def.~2.5]{QAWA_2020} surjects onto $H_n^\affX(A,z)$ by $\mathbf{a} \mapsto \mathbf{a}$, $X_j \mapsto X_j$ and $T_i \mapsto T_i$, and onto $H_n^\aff(A,z)$ by $\mathbf{a} \mapsto \mathbf{a}$, $X_j \mapsto Y_j^{-1}$ and $T_i \mapsto T_i$. These surjections will be shown to be isomorphisms in \cref{thm:AFHA_isomorphisms}.

\section{Basis for \texorpdfstring{$\HH_n(A,z,q)$}{ℍₙ(A,z,q)}} \label{sec:DAFHA_basis}

We now prove basis theorems for $\HH_n(A,z,q)$, namely \cref{thm:DAFHA_basis_theorem} and \cref{thm:DAFHA_basis}. In the $A=\kk$ case, the polynomial representation as in \cref{ex:DAHA_poly_rep} is faithful at generic parameters, and can therefore be used to prove a basis theorem for the double affine Hecke algebra $\HH_n(\kk,\tau-\tau^{-1},q)$. We define polynomial representations of $\HH_n(A,z,q)$ in \cref{sec:DAFHA_poly_rep_phi}, but these are in general not faithful, and in \cref{sec:DAFHA_poly_rep_phi_A}, but these are only defined for $z=0$. We therefore take a different approach in this section; in \cref{thm:DAFHA_representation}, we explicitly construct a representation of $\HH_n(A,z,q)$ by defining the action of the generators on a suitable $\kk$-module and verifying that the defining relations are satisfied. We will then use this representation to prove a basis theorem for $\HH_n(A,z,q)$ for any supersymmetric Frobenius superalgebra $A$.

Recall that $S_n$ acts on $\ZZ^n$ by permutation, and let $\sigma$ denote $(1 \; 2 \; \dotsb \; n) \in S_n$. Extend the action of $S_n$ by permutation on $\kk[X_1^{\pm 1},\dotsc,X_n^{\pm 1}]$ to an action of $W_{\GL_n}^{\extaffine}$, by
\begin{align}
s_0(X^\mu) &= q^{\mu_1-\mu_n}X^{(n \; 1) \cdot(\mu)}= q^{\mu_1-\mu_n}X^{(\mu_n,\mu_2,\dotsc,\mu_{n-1},\mu_1)}\pjmbr
= (qX_nX_1^{-1})^{\mu_1-\mu_n}X^\mu, \label{eq:s_0_operator}\\
\pi(X^\mu) &= q^{-\mu_n}X^{\sigma(\mu)}=q^{-\mu_n}X^{(\mu_n,\mu_1,\dotsc,\mu_{n-1})}, \label{eq:pi_operator}
\end{align}
for $\mu=(\mu_1,\dotsc,\mu_n) \in \ZZ^n$. Note that if $q^m=1$ for some $m \in \ZZ \setminus \{0\}$, then $\pi^{mn}$ acts trivially. In this section, we make no assumptions on $q \in \kk$ other than its invertibility. For $0 \leq i \leq n-1$, define the \emph{Demazure operators}
\begin{equation} \label{eq:Demazure_operators}
\Delta_i \colon \kk[X_1^{\pm 1},\dotsc,X_n^{\pm 1}] \to \kk[X_1^{\pm 1},\dotsc,X_n^{\pm 1}], \qquad
X^\mu \mapsto \frac{s_i(X^\mu)-X^\mu}{X_is_i(X_i)^{-1}-1},
\end{equation}
where $X_0:=X_n$. These can be extended to $A^{\otimes n}$-linear operators on $A^{\otimes n} \otimes \kk[X_1^{\pm 1},\dotsc,X_n^{\pm 1}]$, and for $1 \leq i \leq n-1$, they then coincide with the Demazure operators of \cite[(3.2)]{QAWA_2020}. However, in order to keep the actions of $W_{\GL_n}^{\extaffine}$ on $A^{\otimes n}$ and on $\kk[X_1^{\pm 1},\dotsc,X_n^{\pm 1}]$ separate, the present paper avoids this $A^{\otimes n}$-linear extension. Nonetheless, the identities \cite[(3.3)--(3.9)]{QAWA_2020} hold for the present definition, simply by specialising those results to $A=\kk$, and noting that the proofs also work for $i=0$. The Demazure operators are well-defined, as for $1 \leq i \leq n-1$,
\begin{align}
\Delta_i(X^\mu)= \frac{(X_iX_{i+1}^{-1})^{\mu_{i+1}-\mu_i}X^\mu-X^\mu}{X_i(X_{i+1})^{-1}-1} = \frac{(X_iX_{i+1}^{-1})^{\mu_{i+1}-\mu_i}-1}{(X_iX_{i+1}^{-1})-1}X^\mu, \label{eq:Delta_i_well_def}\\ 
\Delta_0(X^\mu)= \frac{(qX_nX_1^{-1})^{\mu_1-\mu_n}X^\mu-X^\mu}{X_n(q^{-1}X_1)^{-1}-1} = \frac{(qX_nX_1^{-1})^{\mu_1-\mu_n}-1}{(qX_nX_1^{-1})-1}X^\mu, \label{eq:Delta_0_well_def}
\end{align}
are both clearly in $\kk[X_1^{\pm 1},\dotsc,X_n^{\pm 1}]$. We have
\begin{align}
\Delta_0(X_j)&=\frac{s_0(X_j)-X_j}{X_n s_0(X_n)^{-1}-1}=0, \qquad 1<j<n, \label{eq:Delta_0_X_j}\\
\Delta_0(X_n)&=\frac{s_0(X_n)-X_n}{X_n s_0(X_n)^{-1}-1}=\frac{q^{-1}X_1-X_n}{X_n(q^{-1}X_1)^{-1}-1} \pjmbr
=q^{-1}X_1\frac{X_1-qX_n}{qX_n-X_1}=-q^{-1}X_1, \label{eq:Delta_0_X_n}\\
\Delta_0(X_1)&=\frac{s_0(X_1)-X_1}{X_n s_0(X_n)^{-1}-1}=\frac{qX_n-X_1}{X_n(q^{-1}X_1)^{-1}-1} \pjmbr
=X_1\frac{qX_n-X_1}{qX_n-X_1}=X_1. \label{eq:Delta_0_X_1}
\end{align}

\begin{lemma} \label{thm:lusztig_relation_Daff}
In $\HH_n(A,z,q)$, for $\mathbf{a} \in A^{\otimes n}$, $0 \leq i \leq n-1$ and $f \in \kk[X_1^{\pm 1},\dotsc,X_n^{\pm 1}]$ we have
\begin{equation}
T_i\mathbf{a}f-s_i(\mathbf{a})s_i(f)T_i=zt_{i,i+1}\mathbf{a}\Delta_i(f). \label{eq:lusztig_relation_Daff}
\end{equation}
\end{lemma}
\begin{proof}
A version of this statement is proved for $H_n^\affX(A,z)$ in \cite[Lemma~3.3]{QAWA_2020}. The same proof works for \cref{thm:lusztig_relation_Daff}, once we verify that for $j \notin \{1, n\}$,
\begin{align*}
T_0X_j &\overset{\eqref{dafha_rel_7}}{=} X_jT_0 \overset{\eqref{eq:Delta_0_X_j}}{=} X_jT_0+zt_{n,1}\Delta_0(X_j),\\
T_0X_n &\overset{\eqref{dafha_rel_8b}}{=} q^{-1}X_1T_0^{-1} \overset{\eqref{dafha_rel_3}}{=} q^{-1}X_1(T_0-zt_{n,1})\twelvebr
=q^{-1}X_1T_0-q^{-1}zt_{n,1}X_1\overset{\eqref{eq:Delta_0_X_n}}{=} q^{-1}X_1T_0+zt_{n,1}\Delta_0(X_n),\\
T_0X_1 &\overset{\eqref{dafha_rel_3}}{=} (T_0^{-1}+zt_{n,1})X_1\overset{\eqref{dafha_rel_8b}}{=} qX_nT_0+zt_{n,1}X_1 \pjmbr
\overset{\eqref{eq:Delta_0_X_1}}{=} qX_nT_0+zt_{n,1}\Delta_0(X_1). \qedhere
\end{align*}
\end{proof}

\begin{lemma} \label{thm:pi_dem_pi}
As operators on $\kk[X_1^{\pm 1},\dotsc,X_n^{\pm 1}]$, $s_i$ from \eqref{eq:s_0_operator}, $\pi$ from \eqref{eq:pi_operator} and $\Delta_i$ from \eqref{eq:Demazure_operators} satisfy
\begin{align} 
\pi \circ s_i \circ \pi^{-1} &= s_{i+1}, \label{eq:pi_s_i_pi}\\
\pi \circ \Delta_i \circ \pi^{-1} &= \Delta_{i+1}, \label{eq:pi_dem_pi}
\end{align}
taking indices $0 \leq i \leq n-1$ modulo $n$.
\end{lemma}
\begin{proof}
Relation \eqref{eq:pi_s_i_pi} is trivially true for $1 \leq i < n-1$, and needs only a simple verification for $i=n-1$ and $i=0$. \details{For $i=n-1$, we have
\begin{multline*}
\pi \circ s_{n-1} \circ \pi^{-1}(X^\mu)=\pi \circ s_{n-1} \left(q^{\mu_1}X^{\sigma^{-1}(\mu)}\right) = \pi \left(q^{\mu_1}X^{(n-1 \; n) \circ \sigma^{-1}(\mu)}\right) = q^{\mu_1}q^{-\mu_n}X^{\sigma \circ (n-1 \; n) \circ \sigma^{-1}(\mu)}\\ = q^{\mu_1-\mu_n}X^{(n \; 1) \cdot (\mu)} = s_0(X^\mu).
\end{multline*}
For $i=0$, we have
\begin{multline*}
\pi \circ s_0 \circ \pi^{-1}(X^\mu)=\pi \circ s_0 \left(q^{\mu_1}X^{\sigma^{-1}(\mu)}\right) = \pi \left(q^{\mu_1}q^{\mu_2-\mu_1}X^{(n \; 1) \circ \sigma^{-1}(\mu)}\right) = q^{\mu_1}q^{\mu_2-\mu_1}q^{-\mu_2}X^{\sigma \circ (n \; 1) \circ \sigma^{-1}(\mu)}\\ = X^{(1 \; 2) \cdot (\mu)} = s_1(X^\mu).
\end{multline*}}
To prove Relation \eqref{eq:pi_dem_pi}, we similarly verify the same cases. \details{Let $1 \leq i < n-1$ and $\mu=(\mu_1,\dotsc,\mu_n) \in \ZZ^n$. Then,
\begin{multline*}
\pi \circ \Delta_i \circ \pi^{-1} (X^\mu) = \pi \circ \Delta_i(q^{\mu_1} X^{\sigma^{-1}(\mu)}) = \pi \left(q^{\mu_1}\frac{(X_iX_{i+1}^{-1})^{\mu_{i+2}-\mu_{i+1}}-1}{(X_iX_{i+1}^{-1})-1}X^{\sigma^{-1}(\mu)} \right)\\= q^{\mu_1}q^{-\mu_1}\frac{(X_{i+1}X_{i+2}^{-1})^{\mu_{i+2}-\mu_{i+1}}-1}{(X_{i+1}X_{i+2}^{-1})-1}X^{\sigma \circ \sigma^{-1}(\mu)} = \frac{(X_{i+1}X_{i+2}^{-1})^{\mu_{i+2}-\mu_{i+1}}-1}{(X_{i+1}X_{i+2}^{-1})-1}X^{\mu} = \Delta_{i+1}(X^\mu).
\end{multline*}
For $i=n-1$, we have
\begin{multline*}
\pi \circ \Delta_{n-1} \circ \pi^{-1} (X^\mu) = \pi \circ \Delta_{n-1}(q^{\mu_1} X^{\sigma^{-1}(\mu)}) = \pi \left(q^{\mu_1}\frac{(X_{n-1}X_n^{-1})^{\mu_{1}-\mu_{n}}-1}{(X_{n-1}X_n^{-1})-1}X^{\sigma^{-1}(\mu)} \right)\\ = q^{\mu_1}q^{-\mu_1}\frac{(X_n(q^{-1}X_{1})^{-1})^{\mu_{1}-\mu_{n}}-1}{(X_n(q^{-1}X_1)^{-1})-1}X^{\sigma \circ \sigma^{-1}(\mu)} = \frac{(qX_nX_{1}^{-1})^{\mu_{1}-\mu_{n}}-1}{(qX_nX_1^{-1})-1}X^{\mu} = \Delta_{0}(X^\mu).
\end{multline*}
For $i=0$, we have
\begin{multline*}
\pi \circ \Delta_{0} \circ \pi^{-1} (X^\mu) = \pi \circ \Delta_{0}(q^{\mu_1} X^{\sigma^{-1}(\mu)}) = \pi \left(q^{\mu_1}\frac{(qX_{n}X_1^{-1})^{\mu_{2}-\mu_{1}}-1}{(qX_nX_1^{-1})-1}X^{\sigma^{-1}(\mu)} \right)\\ =  q^{\mu_1}q^{-\mu_1}\frac{(qq^{-1}X_1X_{2}^{-1})^{\mu_{2}-\mu_{1}}-1}{(qq^{-1}X_1X_2^{-1})-1}X^{\sigma \circ \sigma^{-1}(\mu)} =  \frac{(X_1X_2^{-1})^{\mu_{2}-\mu_{1}}-1}{(X_1X_2^{-1})-1}X^{\mu} = \Delta_{1}(X^\mu). \qedhere
\end{multline*}}
\end{proof}

\begin{proposition} \label{thm:DAFHA_representation}
Let $V$ be the $\kk$-module $A^{\otimes n} \otimes \kk[X_1^{\pm 1},\dotsc,X_n^{\pm 1}] \otimes \kk W_{\GL_n}^{\extaffine}$. Then $V$ is an $\HH_n(A,z,q)$-module, with the action $\beta \colon \HH_n(A,z,q) \to \Endo_\kk(V)$ given by
\begin{align}
\mathbf{a} &\mapsto \mathbf{a} \otimes 1 \otimes 1,\\
X_j &\mapsto \mathbf{1} \otimes X_j \otimes 1,\\
\beta(T_i)(\mathbf{a} \otimes X^\mu \otimes w) &= \left(s_i(\mathbf{a}) \otimes s_i(X^\mu) \otimes s_i w \right)\pjmbr
+ \left(z t_{i,i+1}\mathbf{a} \otimes \Delta_{i,w}(X^\mu) \otimes w\right),  \label{eq:T_i_basic_action}\\
\beta(\pi)(\mathbf{a} \otimes X^\mu \otimes w) &= \sigma(\mathbf{a}) \otimes \pi(X^{\mu}) \otimes \pi w,  \label{eq:pi_basic_action}
\end{align}
for $1 \leq j \leq n$, $0 \leq i \leq n-1$, $\mathbf{a} \in A^{\otimes n}$, $\mu \in \ZZ^n$ and $w \in W_{\GL_n}^{\extaffine}$, where
\[
\Delta_{i,w}(X^\mu) := \begin{cases} \Delta_i(X^\mu), &\textup{ if }l(s_iw)>l(w),\\
X_{i+1}^{-1}\Delta_i(X_{i+1}X^\mu), &\textup{ if } l(s_iw)<l(w).
\end{cases}
\]
In the above, we define $s_0(X^\mu) := q^{\mu_1-\mu_n}X^{s_0(\mu)}$, $t_{0,1}:=t_{n,1}$ and $X_0:=X_n$. See \eqref{eq:superpermute_A} for the action of $s_0$ and of $\sigma = (1 \; 2 \; \dotsb \; n)$ on $A^{\otimes n}$ by superpermutation. The length of an element $w \in W_{\GL_n}^{\extaffine}$ is denoted $l(w)$, with $l(\pi)=0$.
\end{proposition}
\begin{proof}
We must verify the relations of \cref{thm:double_affine_Frobenius_Hecke_algebra} to ensure that $\beta$ is well-defined. In \cite[Prop.~3.8]{QAWA_2020}, the $\kk$-module
\[
V':=A^{\otimes n} \otimes \kk[X_1^{\pm 1},\dotsc,X_n^{\pm 1}] \otimes H_n(\kk,z) \cong_\kk A^{\otimes n} \otimes \kk[X_1^{\pm 1},\dotsc,X_n^{\pm 1}] \otimes \kk S_n
\]
is shown to be an $H_n^\affX (A,z)$-module, where $\mathbf{a}$, $X_1^{\pm 1}, \dotsc, X_n^{\pm 1}$ and $T_1^{\pm 1}, \dotsc, T_{n-1}^{\pm 1}$ act very similarly to their actions in our proposed module, albeit restricted to the submodule $V'$. Indeed, consider the right-coset decomposition $W_{\GL_n}^{\extaffine} = \bigsqcup_{d \in D}S_n d$ for some set $D$ of coset representatives of $S_n \backslash W_{\GL_n}^{\extaffine}$ that are minimal by length, and hence $V=\bigoplus_{d \in D} V' d$. By minimality of length of elements of $D$, for all $1 \leq i \leq n-1$, $w \in S_n$ and $d \in D$, we have $l(s_i wd)-l(wd)=l(s_i w)-l(w)$. Consequently, each $\kk$-module isomorphism
\[
V' \to V'd, \qquad \mathbf{a} \otimes X^{\mu} \otimes w \mapsto \mathbf{a} \otimes X^{\mu} \otimes wd
\]
preserves the cases that determine the action of $T_i$ \eqref{eq:T_i_basic_action}. It is then easily seen that the action of $H_n^\affX(A,z)$ on $V'$ described in \cite[Prop.~3.8]{QAWA_2020} is identical to the action given by $\beta$ on each summand $V'd$ for $d \in D$. Hence all required relations in the image of $\beta$ that do not involve $\beta(\pi)$ or $\beta(T_0)$ follow from \cite[Prop.~3.8]{QAWA_2020}. For example, the verifications of Relations \labelcref{dafha_rel_1,dafha_rel_2,dafha_rel_3,dafha_rel_4,dafha_rel_7,dafha_rel_8,dafha_rel_9}, for $i \neq 0$, need not be repeated. We thus start with Relation \eqref{dafha_rel_5}. Considering $0 \leq i \leq n-1$, let $\mathbf{a} \in A^{\otimes n}$, $\mu=(\mu_1, \dotsc , \mu_n) \in \ZZ^n$ and $w \in W_{\GL_n}^{\extaffine}$ such that $l(s_i\pi^{-1}w)<l(\pi^{-1}w)$. Then, taking indices modulo $n$,
\begin{align*}
&\; \beta(\pi)\beta(T_i)\beta(\pi)^{-1}(\mathbf{a} \otimes X^{\mu} \otimes w)\\
&= \beta(\pi)\beta(T_i)\left(\sigma^{-1}(\mathbf{a}) \otimes \pi^{-1}(X^\mu) \otimes \pi^{-1} w \right) \\
&= \beta(\pi)\Big( (s_i \circ \sigma^{-1})(\mathbf{a}) \otimes (s_i \circ \pi^{-1})(X^\mu) \otimes s_i \pi^{-1} w \displaybreak[1]\\
&\quad + zt_{i,i+1}\sigma^{-1}(\mathbf{a}) \otimes X_{i+1}^{-1}\Delta_i(X_{i+1}\pi^{-1}(X^\mu)) \otimes \pi^{-1} w \Big)\\
&= (\sigma \circ s_i \circ \sigma^{-1})(\mathbf{a}) \otimes (\pi \circ s_i \circ \pi^{-1})(X^\mu) \otimes \pi s_i \pi^{-1} w \displaybreak[1]\\
&\quad + zt_{i+1,i+2}(\sigma \circ \sigma^{-1})(\mathbf{a}) \otimes \pi(X_{i+1}^{-1}\Delta_i(X_{i+1}\pi^{-1}(X^\mu))) \otimes \pi \pi^{-1} w \displaybreak[1]\\
&\overset{\eqref{dafha_rel_10}}{=} (\sigma \circ s_i \circ \sigma^{-1})(\mathbf{a}) \otimes (\pi \circ s_i \circ \pi^{-1})(X^\mu) \otimes \pi s_i \pi^{-1} w\\
&\quad+ zt_{i+1,i+2}\mathbf{a} \otimes X_{i+2}^{-1}(\pi \circ \Delta_i \circ \pi^{-1})(X_{i+2}X^\mu) \otimes \pi \pi^{-1} w\\
&\underset{\eqref{eq:pi_dem_pi}}{\overset{\eqref{eq:pi_s_i_pi}}{=}} (s_{i+1})(\mathbf{a}) \otimes s_{i+1}(X^\mu) \otimes s_{i+1} w\pjmbr
+ zt_{i+1,i+2}\mathbf{a} \otimes X_{i+2}^{-1}\Delta_{i+1}(X_{i+2}X^{\mu}) \otimes w\\
&= \beta(T_{i+1})(\mathbf{a} \otimes X^\mu \otimes w).
\end{align*}
In the case $i=n-1$ of the above verification, the step using Relation \eqref{dafha_rel_10} uses Relation \eqref{dafha_rel_10b} instead, and the resulting factors of $q$ cancel immediately. The case $l(s_i\pi^{-1}w)>l(\pi^{-1}w)$ is similar.\details{
\begin{align*}
&\; \beta(\pi)\beta(T_i)\beta(\pi)^{-1}(\mathbf{a} \otimes X^{\mu} \otimes w)\\
&= \beta(\pi)\beta(T_i)\left(\sigma^{-1}(\mathbf{a}) \otimes \pi^{-1}(X^\mu) \otimes \pi^{-1} w \right) \\
&= \beta(\pi)\Big( (s_i \circ \sigma^{-1})(\mathbf{a}) \otimes (s_i \circ \pi^{-1})(X^\mu) \otimes s_i \pi^{-1} w \\
&\qquad + zt_{i,i+1}\sigma^{-1}(\mathbf{a}) \otimes (\Delta_i \circ \pi^{-1})(X^\mu) \otimes \pi^{-1} w \Big)\\
&= (\sigma \circ s_i \circ \sigma^{-1})(\mathbf{a}) \otimes (\pi \circ s_i \circ \pi^{-1})(X^\mu) \otimes \pi s_i \pi^{-1} w\\
&\quad + zt_{i+1,i+2}(\sigma \circ \sigma^{-1})(\mathbf{a}) \otimes (\pi \circ \Delta_i \circ \pi^{-1})(X^\mu) \otimes \pi \pi^{-1} w\\
&\underset{\eqref{eq:pi_dem_pi}}{\overset{\eqref{eq:pi_s_i_pi}}{=}} (s_{i+1})(\mathbf{a}) \otimes s_{i+1}(X^\mu) \otimes s_{i+1} w + zt_{i+1,i+2}\mathbf{a} \otimes \Delta_{i+1}(X^{\mu}) \otimes w\\
&= \beta(T_{i+1})(\mathbf{a} \otimes X^\mu \otimes w).
\end{align*}}
Thus, $\beta(\pi)\beta(T_i)\beta(\pi)^{-1} = \beta(T_{i+1})$, and Relation \labelcref{dafha_rel_5} is verified. An expression in $\beta(T_0)$ can be rewritten using the shifting operator $\pi$ into an expression in the operators $\beta(T_i)$, for $1 \leq i \leq n-1$, in order to prove the $i=0$ cases of Relations \labelcref{dafha_rel_1,dafha_rel_2,dafha_rel_3,dafha_rel_4,dafha_rel_7}. For example, for $1 \leq j \leq n-1$ with $j \notin \{n-1,1\}$, Relation \labelcref{dafha_rel_1} holds for $i=0$, as
\begin{multline*}
\beta(T_0)\beta(T_j) = \beta(\pi)^{-1}\beta(T_1)\beta(\pi)\beta(T_j) = \beta(\pi)^{-1}\beta(T_1)\beta(T_{j+1})\beta(\pi)\\
= \beta(\pi)^{-1}\beta(T_{j+1})\beta(T_1)\beta(\pi) = \beta(\pi)^{-1}\beta(T_{j+1})\beta(\pi)\beta(T_0)
= \beta(T_j)\beta(T_0).
\end{multline*}
\details{We verify Relations \labelcref{dafha_rel_2,dafha_rel_3,dafha_rel_4,dafha_rel_7} for $i=0$ similarly.
For $\mathbf{a} \in A^{\otimes n}$:
\begin{align*}
\beta(T_0)\beta(T_1)\beta(T_0) &= \beta(\pi)^{-1}\beta(T_1)\beta(\pi)\beta(\pi)^{-1}\beta(T_2)\beta(\pi)\beta(\pi)^{-1}\beta(T_1)\beta(\pi)\\
&= \beta(\pi)^{-1}\beta(T_1)\beta(T_2)\beta(T_1)\beta(\pi) = \beta(\pi)^{-1}\beta(T_2)\beta(T_1)\beta(T_2)\beta(\pi)\\
&= \beta(\pi)^{-1}\beta(T_2)\beta(\pi)\beta(\pi)^{-1}\beta(T_1)\beta(\pi)\beta(\pi)^{-1}\beta(T_2)\beta(\pi) = \beta(T_1)\beta(T_0)\beta(T_1),\\
\beta(T_0)^2 &= (\beta(\pi)^{-1}\beta(T_1)\beta(\pi))^2 = \beta(\pi)^{-1}\beta(T_1)^2\beta(\pi)\\
&= \beta(\pi)^{-1}(z\beta(t_{1,2})\beta(T_1)+1)\beta(\pi) = z\beta(\pi)^{-1}\beta(t_{1,2})\beta(T_1)\beta(\pi)+1\\
&= z\beta(\sigma^{-1}(t_{1,2}))\beta(\pi)^{-1}\beta(T_1)\beta(\pi)+1 = z\beta(t_{0,1})\beta(T_0)+1,\\
\beta(T_0)\beta(\mathbf{a}) &= \beta(\pi)^{-1}\beta(T_1)\beta(\pi)\beta(\mathbf{a}) = \beta(\pi)^{-1}\beta(T_1)\beta(\sigma(\mathbf{a}))\beta(\pi)\\
&= \beta(\pi)^{-1}\beta((s_1 \circ \sigma)(\mathbf{a}))\beta(T_1)\beta(\pi) = \beta((\sigma^{-1} \circ s_1 \circ \sigma)(\mathbf{a}))\beta(\pi)^{-1}\beta(T_1)\beta(\pi)\\
&= \beta(s_0(\mathbf{a}))\beta(T_0).\\
\beta(T_0)\beta(X_j) &= \beta(\pi)^{-1}\beta(T_1)\beta(\pi)\beta(X_j) = \beta(\pi)^{-1}\beta(T_1)\beta(X_{j+1})\beta(\pi)\\
&=\beta(\pi)^{-1}\beta(X_{j+1})\beta(T_1)\beta(\pi) = \beta(X_j)\beta(\pi)^{-1}\beta(T_1)\beta(\pi)
=\beta(X_j)\beta(T_0).
\end{align*}
Note that the first of these is only possible and necessary for $n \geq 3$. It is clear from \eqref{eq:pi_basic_action} that Relation \labelcref{dafha_rel_6} holds for the operators, and we use this fact in the second of the four proofs above.}
Relations \labelcref{dafha_rel_10} and \labelcref{dafha_rel_10b} hold for the operators, as for $\mathbf{a} \in A^{\otimes n}$, $\mu = (\mu_1, \dotsc, \mu_n) \in \ZZ^n$ and $w \in W_{\GL_n}^{\extaffine}$, we have
\begin{align*}
&\beta(\pi)\beta(X_i)\beta(\pi)^{-1}(\mathbf{a} \otimes X^\mu \otimes w)\\
&= \beta(\pi)\beta(X_i)(\sigma^{-1}(\mathbf{a}) \otimes q^{\mu_1}X^{\sigma^{-1}(\mu)} \otimes \pi^{-1}w)\\
&=\beta(\pi)(\sigma^{-1}(\mathbf{a}) \otimes q^{\mu_1}X^{\sigma^{-1}(\mu)+\epsilon_i} \otimes \pi^{-1}w)\\
&=\mathbf{a} \otimes q^{\mu_1-\mu_1-\delta_{i,n}}X^{\mu+\epsilon_{i+1}} \otimes  w\\
&=\begin{cases}\beta(X_{i+1})(\mathbf{a} \otimes X^\mu \otimes w) &\text{if } 1 \leq i \leq n-1,\\ q^{-1}\beta(X_1)(\mathbf{a} \otimes X^\mu \otimes w) &\text{if } i=n,\end{cases}
\end{align*}
where $\epsilon_i=(0,\dotsc,0,1,0,\dotsc,0) \in \ZZ^n$ is non-zero only in the $i$-th component, and $i$ is reduced modulo $n$. Finally, Relation \labelcref{dafha_rel_8b} follows from Relation \labelcref{dafha_rel_10b} by
\begin{align*}
\beta(T_0)\beta(X_n)\beta(T_0)&=\beta(\pi)^{-1}\beta(T_1)\beta(\pi)\beta(X_n)\beta(\pi)^{-1}\beta(T_1)\beta(\pi)\\
&= q^{-1}\beta(\pi)^{-1}\beta(T_1)\beta(X_1)\beta(T_1)\beta(\pi)\pjmbr
= q^{-1}\beta(\pi)^{-1}\beta(X_2)\beta(\pi) = q^{-1}\beta(X_1). \qedhere
\end{align*}
\end{proof}

\begin{theorem} \label{thm:DAFHA_basis_theorem}
The following $\kk$-module homomorphism is an isomorphism of $\HH_n(A,z,q)$-modules:
\begin{equation} \label{eq:DAFHA_basis}
\begin{gathered}
V=A^{\otimes n} \otimes \kk[X_1^{\pm 1},\dotsc,X_n^{\pm 1}] \otimes \kk W_{\GL_n}^{\extaffine} \to \HH_n(A,z,q),\\
\mathbf{a} \otimes X^\mu \otimes w \mapsto \mathbf{a} X^\mu T_w,
\end{gathered}
\end{equation}
for $\mathbf{a} \in A^{\otimes n}$, $\mu \in \ZZ^n$ and $w \in W_{\GL_n}^{\extaffine}$.
\end{theorem}
\begin{proof}
If $B$ is a basis of $A$, then $B^{\otimes n} = \{a_1 \otimes \dotsb \otimes a_n \in A^{\otimes n}: a_1,\dotsc,a_n \in B\}$ is a basis of $A^{\otimes n}$. Thus
\[
\mathcal{B}_1 := \{\mathbf{b} \otimes X^\mu \otimes w: \mathbf{b} \in B^{\otimes n}, \mu \in \ZZ^n, w \in W_{\GL_n}^{\extaffine}\},
\]
is a basis of $V$. By considering the relations in \cref{thm:double_affine_Frobenius_Hecke_algebra}, it is clear that the set
\[
\mathcal{B}_2 := \{\mathbf{b} X^\mu \pi^k T_w: \mathbf{b} \in B^{\otimes n}, \mu \in \ZZ^n, k \in \ZZ, w \in S_n^{\affine}\}
\]
spans $\HH_n(A,z,q)$. The homomorphism $\Psi \colon \HH_n(A,z,q) \to V$ of $\HH_n(A,z,q)$-modules defined by $h \mapsto \beta(h)(1 \otimes 1 \otimes 1)$ maps $\mathbf{b} X^\mu \pi^k T_w \in \mathcal{B}_2$ to $\mathbf{b} \otimes X^\mu \otimes \pi^k w \in \mathcal{B}_1$. Thus, the set $\mathcal{B}_2$ is linearly independent, and so a basis of $\HH_n(A,z,q)$. Consequently, as $\Psi$ maps the elements of $\mathcal{B}_2$ one-to-one to elements of $\mathcal{B}_1$, it is an isomorphism. It is clearly inverse to the required homomorphism of \eqref{eq:DAFHA_basis}.
\end{proof}

\begin{lemma} \label{thm:AFHA_isomorphisms}
Let $z \in \kk$, and let $\mathcal{H}_n(A,z)$ and $\mathcal{H}_n^\affX(A,z)$ denote the quantum wreath algebra and the quantum affine wreath algebra of \cite[Def.~2.1, Def.~2.5]{QAWA_2020} respectively.
\begin{enumerate}[i)]
\item The assignments $\mathbf{a} \mapsto \mathbf{a}$, $T_i \mapsto T_i$ and $X_i \mapsto X_i$ define an isomorphism $\mathcal{H}_n^\affX(A,z) \to H_n^\affX(A,z)$ that restricts to an isomorphism $\mathcal{H}_n(A,z) \to H_n(A,z)$.
\item The assignments $\mathbf{a} \mapsto \mathbf{a}$, $X_i \mapsto X_i^{-1}$ and $T_i \mapsto T_i^{-1}$ define an isomorphism $H_n^\affX(A,z) \to H_n^\affX(A,-z)$.
\item The assignments $\mathbf{a} \mapsto \mathbf{a}$, $X_i \mapsto Y_i^{-1}$ and $T_i \mapsto T_i$ define an isomorphism $H_n^\affX(A,z) \to H_n^\aff(A,z)$.
\end{enumerate}
\end{lemma}
\begin{proof}
\begin{enumerate}[i)]
\item The assignments clearly define surjective homomorphisms. As they map the basis from \cite[Cor.~3.11]{QAWA_2020} to linearly independent elements by \cref{thm:DAFHA_basis_theorem}, these are isomorphisms.
\item By i), the Relations \eqref{qwa_rel_1}--\eqref{daqwa_sym_rel_3} not containing $Y_j$ are sufficient to present $H_n^\affX(A,z)$. It is then easy to verify that the assignments define a homomorphism. For example, Relation \eqref{qwa_rel_3} maps to itself, but with $-z$ instead of $z$, and \eqref{daqwa_sym_rel_2} is mapped to its own inverse. The assignment is an isomorphism, as the same assignment $H_n^\affX(A,-z) \to H_n^\affX(A,z)$ gives its inverse.
\item By considering Relations \eqref{qwa_rel_1}--\eqref{daqwa_sym_rel_3}, the algebra $H_n^\affX(A,z)$ does not depend on $q$. Thus, $\epsilontransform$ clearly restricts to an isomorphism $H_n^\affX(A,-z) \to H_n^\aff(A,z)$. Composing this with the isomorphism from ii), we obtain the isomorphism from iii). \qedhere
\end{enumerate}
\end{proof}

\begin{remark} \label{rem:embed_AHA}
The isomorphism iii) from \cref{thm:AFHA_isomorphisms} extends to the automorphism $\sigmaaut$ of $\HH_n(A,z,q)$ given in \eqref{eq:sigma}, by the proof of \cref{thm:B_3_action}. Note that \cref{thm:extra_twist_relations} prevents extension by the assignment $Y_i \mapsto X_i^{-1}$.
\end{remark}

\begin{corollary} \label{thm:DAFHA_basis}
The sets
\begin{align*}
&\{\mathbf{b} X^\mu T_w: \mathbf{b} \in B^{\otimes n}, \mu \in \ZZ^n, w \in W_{\GL_n}^{\extaffine}\} \quad \text{ and } \quad \twelvebr
\{\mathbf{b} X^\mu Y^\lambda T_w: \mathbf{b} \in B^{\otimes n}, \mu \in \ZZ^n, \lambda \in \ZZ^n, w \in S_n\}
\end{align*}
are $\kk$-bases of $\HH_n(A,z,q)$.
\end{corollary}
\begin{proof}
The first set is a basis by \cref{thm:DAFHA_basis_theorem}, with the indexing by the group $W_{\GL_n}^{\extaffine}$ described in the comment after \cref{thm:double_affine_Frobenius_Hecke_algebra}. This implies that multiplication gives an isomorphism of $\kk$-modules $\kk[X_1^{\pm 1},\dotsc,X_n^{\pm 1}] \otimes H_n^\aff(A,z) \to \HH_n(A,z,q)$. By the isomorphisms from \textit{i)} and \textit{iii)} of \cref{thm:AFHA_isomorphisms}, the basis from \cite[Cor.~3.11]{QAWA_2020} maps to the elements $\{\mathbf{b} Y^{-\lambda} T_w: \mathbf{b} \in B^{\otimes n}, \lambda \in \ZZ^n, w \in S_n\}$, which therefore form a basis of $H_n^\aff(A,z)$. The second basis of this corollary follows from these two observations.
\end{proof}

\begin{corollary} \label{thm:AFHA_coxeter_basis}
The sets
\begin{align*}
&\{\mathbf{b} T_w: \mathbf{b} \in B^{\otimes n}, w \in W_{\GL_n}^{\extaffine}\} \quad \text{ and } \quad \pjmbr 
\{\mathbf{b} Y^\lambda T_w: \mathbf{b} \in B^{\otimes n}, \lambda \in \ZZ^n, w \in S_n\}
\end{align*}
are $\kk$-bases of $H_n^\aff(A,z)$.
\end{corollary}
\begin{proof}
It is simple to verify that the proposed bases span $H_n^\aff(A,z)$. Linear independence follows from \cref{thm:DAFHA_basis}.
\end{proof}

The analogue to the second basis of \cref{thm:AFHA_coxeter_basis} is given for $H_n^\affX(A,z)$ in \cite[Cor.~3.11]{QAWA_2020}.

\begin{corollary} \label{thm:HA_basis_general}
The set
\[
\{\mathbf{b} T_w: \mathbf{b} \in B^{\otimes n}, w \in S_n\} \subsetneq H_n(A,z)
\]
is a $\kk$-basis of $H_n(A,z)$.
\end{corollary}
\begin{proof}
It is simple to verify that the proposed basis spans $H_n(A,z)$. Linear independence follows from \cref{thm:DAFHA_basis}.
\end{proof}

\section{Supercentre for \texorpdfstring{$\HH_n(A,z,q)$}{ℍₙ(A,z,q)}} \label{sec:centre_for_DAFHA}

We will find the supercentre of $\HH_n(A,z,q)$ in the cases where $q^m-1$ is not a zero divisor in $\kk$ for all $m \in \ZZ \setminus \{0\}$. In the standard $A=\kk$ case, this is well-known. By contrast, for $q=1$, if $A=\kk=\CC$ and $z=\tau-\tau^{-1}$ for some invertible $\tau \in \CC$ that is not a root of unity, \cite[Thm.~5.1]{DAHACMS_2004} gives an isomorphism from the centre of the DAHA $\HH_n(\CC,\tau-\tau^{-1},1)$ to the spherical subalgebra $e\HH_n(\CC,\tau-\tau^{-1},1) e$ of the DAHA, where $e \in H_n(\CC,z)$ is a symmetrising idempotent satisfying $T_ie=eT_i=\tau e$. Such an idempotent does not seem to exist in $H_n(A,\tau-\tau^{-1})$ for general $A$, and such a `Satake isomorphism' is unlikely to describe the centre of $\HH_n(A,z,q)$ for most Frobenius algebras $A$.

Given a subalgebra $S$ of $A$, let $P_n(S)$ refer to the subalgebra of $\HH_n(A,z,q)$ generated by $S^{\otimes n}$, $X_1^{\pm 1},\dotsc,X_n^{\pm 1}$ and $Y_1^{\pm 1},\dotsc,Y_n^{\pm 1}$. Let $P_n^X(S)$ refer to the subalgebra of $\HH_n(A,z,q)$ generated by $S^{\otimes n}$ and $X_1^{\pm 1},\dotsc,X_n^{\pm 1}$. Let $P_n^Y(S)$ refer to the subalgebra of $\HH_n(A,z,q)$ generated by $S^{\otimes n}$ and $Y_1^{\pm 1},\dotsc,Y_n^{\pm 1}$. Let $P_n:=P_n(\kk)$, $P_n^X:=P_n^X(\kk)$ and $P_n^Y:=P_n^Y(\kk)$.

Note that the strong Bruhat order on $S_n$ extends naturally to the affine symmetric group $S_n^{\affine}$, as this is also a Coxeter group. We define the action of $W_{\GL_n}^{\extaffine}$ on $P_n^X(S)$ as acting by superpermutation on the generators $S^{\otimes n}$ with $s_0$ acting by $(1 \; n)$ and $\pi$ acting by $(1 \; 2 \; \dotsb \; n)$, and on the generators $X_1^{\pm 1},\dotsc,X_n^{\pm 1}$ as in \eqref{eq:s_0_operator} and \eqref{eq:pi_operator}.


The proof given in \cite[Sec.~3.3]{QAWA_2020} for the supercentre of $H_n^\affX(A,z)$ needs some adaptation to also work for $\HH_n(A,z,q)$. Throughout this section, we define the supercentre of a superalgebra $D=D_0 \oplus D_1$ to be $Z(D)=Z(D)_0 \oplus Z(D)_1$, where for $\epsilon \in \{0,1\}$,
\[
Z(D)_{\epsilon} := \{d \in D_{\epsilon} \mid dc = (-1)^{\epsilon\bar{c}}cd \textup{ for all homogeneous } c \in D\}.
\]
Similarly, for a subset $S \subseteq D$, we define its supercentraliser in $D$ by $C_D(S)=C_D(S)_0 \oplus C_D(S)_1$, where for $\epsilon \in \{0,1\}$,
\[
C_D(S)_{\epsilon} := \{d \in D_{\epsilon} \mid ds_\eta=(-1)^{\epsilon\eta}s_\eta d \textup{ for all } s=s_0+s_1 \in S, \eta \in \{0,1\}\},
\]
with $s=s_0+s_1$ denoting the decomposition of $s \in S$ into homogeneous elements $s_\eta \in D_\eta$.

\begin{lemma} \label{thm:decompose_supercentre}
Let $A$ be a $\kk$-superalgebra that is a free $\kk$-module over a homogeneous basis. If the supercentre $Z(A)$ is a direct summand of $A$, then $Z(A^{\otimes n})=Z(A)^{\otimes n}$.
\end{lemma}
\begin{proof}
Suppose $A=Z(A) \oplus M$ for some $\kk$-supermodule $M$. For $1 \leq i \leq n$, denote $A_i := 1^{\otimes(i-1)} \otimes A \otimes 1^{\otimes(n-i)} \subseteq A^{\otimes n}$. Then
\[
A^{\otimes(i-1)} \otimes Z(A) \otimes A^{\otimes(n-i)} \subseteq C_{A^{\otimes n}}(A_i)
\]
is obvious. We will show the reverse inclusion, using the decomposition
\begin{equation} \label{eq:decompose_supercentre}
A^{\otimes n}=\left(A^{\otimes(i-1)} \otimes  Z(A) \otimes A^{\otimes(n-i)} \right) \oplus \left(A^{\otimes(i-1)} \otimes M \otimes A^{\otimes(n-i)} \right).
\end{equation}
Any homogeneous $\mathbf{z} \in C_{A^{\otimes n}}(A_i)$ can thus be decomposed uniquely by this direct sum as $\mathbf{z}=\mathbf{z}_Z + \mathbf{z}_M$. Fix a homogeneous basis $B$ of $A$, and expand $\mathbf{z}_M$ by this basis in positions other than $i$, into a sum indexed by a set $S$, with $m_s \in M$ in position $i$ of the $s$-th summand. By comparing summands, we can deduce that each $m_s$ supercommutes with all homogeneous $a \in A$.\details{ Concretely, we can write $\mathbf{z}_M \in A^{i-1} \otimes M \otimes A^{n-i}$ as a unique sum 
\begin{equation} \label{eq:z_M}
\mathbf{z}_M=\sum_{s \in S}b_{1,s} \otimes \dotsb \otimes b_{i-1,s} \otimes m_s \otimes b_{i+1,s} \otimes \dotsb \otimes b_{n,s}
\end{equation}
indexed by $S=B^{n-1}$, for homogeneous $m_s \in M$ for $s \in S$, and $b_{j,s} \in B$ for $s \in S$ and $j \neq i$. Given any homogeneous $a \in A$, both $\mathbf{z}$ and $\mathbf{z}_Z$ supercommute with $a_i=1^{i-1} \otimes a \otimes 1^{n-i} \in A_i$. Thus, so must $\mathbf{z}_M$, whence $\mathbf{z}_Ma_i=(-1)^{\bar{\mathbf{z}}_M\bar{a}}a_i\mathbf{z}_M$. Defining $l_s:=\sum_{1 \leq j < i}\bar{b}_{j,s}$ and $r_s:=\sum_{i < j \leq n}\bar{b}_{j,s}$ for each $s \in S$, by uniqueness of the expansion  \eqref{eq:z_M} with respect to $B$ in summands other than $i$,
\begin{align*}
&(b_{1,s} \otimes \dotsb \otimes b_{i-1,s} \otimes m_s \otimes b_{i+1,s} \otimes \dotsb \otimes b_{n,s})a_i\\
=&(-1)^{\bar{a}r_s}(b_{1,s} \otimes \dotsb \otimes b_{i-1,s} \otimes m_s a \otimes b_{i+1,s} \otimes \dotsb \otimes b_{n,s})
\end{align*}
must equal
\begin{align*}
&(-1)^{\bar{\mathbf{z}}_M\bar{a}} \left( a_i(b_{1,s} \otimes \dotsb \otimes b_{i-1,s} \otimes m_s \otimes b_{i+1,s} \otimes \dotsb \otimes b_{n,s}) \right)\\
= &(-1)^{\bar{a}(\bar{\mathbf{z}}_M+l_s)} (b_{1,s} \otimes \dotsb \otimes b_{i-1,s} \otimes a m_s \otimes b_{i+1,s} \otimes \dotsb \otimes b_{n,s}).
\end{align*}
Therefore, using $\bar{\mathbf{z}}_M=l_s+\bar{m}_s+r_s$ to simplify,
\begin{align*}
(-1)^{\bar{a}r_s}m_sa=(-1)^{\bar{a}(\bar{\mathbf{z}}_M+l_s)}am_s\\
\Leftrightarrow (-1)^{\bar{a}r_s}m_sa=(-1)^{\bar{a}(l_s+\bar{m}_s+r_s+l_s)}am_s=(-1)^{\bar{a}(\bar{m}_s+r_s)}am_s\\
\Leftrightarrow m_sa = (-1)^{\bar{a}\bar{m}_s}am_s,
\end{align*}
showing that for all $s \in S$, $m_s$ supercommutes with all homogeneous $a \in A$.} Thus $m_s \in Z(A) \cap M = \{0\}$ for all $s \in S$, and so $\mathbf{z}=\mathbf{z}_Z + 0$ lies in the left-hand set of \eqref{eq:decompose_supercentre}. Finally, ranging over $1 \leq j \leq n$, the subalgebras $A_j \subseteq A^{\otimes n}$ generate $A^{\otimes n}$, so $Z(A^{\otimes n})=\bigcap_{1 \leq j \leq n} C_{A^{\otimes n}}(A_j) = Z(A)^{\otimes n}$.
\end{proof}

\begin{lemma} \label{thm:triangularity}
For $w \in S_n^{\affine}$, $k \in \ZZ$ and $f \in P_n^X(A)$,
\begin{equation} \label{eq:commute_T_and_f}
\pi^kT_wf=((\pi^k w) \cdot f)\pi^kT_w+\sum_{v < w}f_v\pi^kT_v,
\end{equation}
for some $f_v \in P_n^X(A)$ and $v \in S_n^{\affine}$, where $<$ denotes the strong Bruhat order on $S_n^{\affine}$.
\end{lemma}
\begin{proof}
The proof in \cite[Lemma~3.4]{QAWA_2020} works here with slight modification. By considering Relation \eqref{eq:lusztig_relation_Daff} from \cref{thm:lusztig_relation_Daff} and performing induction on the length of $w$, we obtain
\begin{equation} \label{eq:commute_T_and_f_proof}
\pi^kT_wf=\pi^k w(f)T_w+\sum_{v < w}\pi^kf_v'T_v
\end{equation}
for some $f_v' \in P_n^X(A)$. By Relations \eqref{dafha_rel_10} and \eqref{dafha_rel_10b}, we have $\pi^k\mathbf{a}X_i=(\pi^k \cdot(\mathbf{a}X_{i}))\pi^k$, where $\pi^k \cdot (\mathbf{a}X_{i})$ is defined as in \eqref{eq:pi_operator}. In particular, given $f_v' \in P_n^X(A)$, we can write $\pi^k f_v' = f_v \pi^k$ for some $f_v \in P_n^X(A)$. Applying this to Equation \eqref{eq:commute_T_and_f_proof} yields the identity \eqref{eq:commute_T_and_f}.
\end{proof}

\begin{lemma} \label{thm:X_centraliser}
For $q \in \kk^\times$ such that $q^m-1$ is not a zero divisor in $\kk$ for all $m \in \ZZ \setminus \{0\}$, the supercentraliser of $P_n^X$ in $\HH_n(A,z,q)$ is the subalgebra $P_n^X(A)$.
\end{lemma}
\begin{proof}
We follow a similar proof as \cite[Lemma~3.14]{QAWA_2020}. Note that $P_n^X$ is purely even. Firstly, elements of $P_n^X(A)$ clearly supercommute with elements of $P_n^X$. Let $h \in \HH_n(A,z,q)$, so that by \cref{thm:DAFHA_basis}, $h = \sum_{w \in S_n^{\affine}, k \in \ZZ}f_{w,k} \pi^kT_w $ for some $f_{w,k} \in P_n^X(A)$, and suppose that $h$ centralises $P_n^X$. Let $v \in S_n^{\affine}$ be a maximal element with respect to the Bruhat order such that ${K:=\{k \in \ZZ : f_{v,k} \neq 0\}} \neq \emptyset$. Choose $k' \in K$. By \cref{thm:triangularity}, for any $f \in P_n^X$, we have
\[
f h - h f = (f - (\pi^{k'} v) \cdot f)f_{v,k'}\pi^{k'} T_v + \sum_{k \in K, k \neq k'}g_{v,k}\pi^kT_v + \sum_{u \not\geq v, l \in \ZZ}g_{u,l}\pi^lT_u,
\]
for some $g_{v,k},g_{u,l} \in P_n^X(A)$. As $h$ centralises $P_n^X$, we have $f h - h f = 0$ for all $f \in P_n^X$. By \cref{thm:DAFHA_basis_theorem}, this implies $(f - (\pi^{k'} v) \cdot f)f_{v,k'}=0$. Assume towards a contradiction that $k' \neq 0$ or $v \neq 1$, so that $\pi^{k'} v \neq 1$ in $W_{\GL_n}^{\extaffine}$. Since $q^m \neq 1$ for all $m \in \ZZ  \setminus \{0\}$, the group $W_{\GL_n}^{\extaffine}$ acts faithfully on monomials by \eqref{eq:s_0_operator} and \eqref{eq:pi_operator}. Thus, there must be some $1 \leq j \leq n$ such that $(\pi^{k'} v) \cdot X_j = q^mX_l$ for some $m \in \ZZ$ and $1 \leq l \leq n$, with either $m \neq 0$ or $l \neq j$. Taking $f=X_j$, we get $(X_j - (\pi^{k'} v) \cdot X_j)f_{v,k'}=0$. If $l \neq j$, this immediately implies by \cref{thm:DAFHA_basis} that $f_{v,k'}=0$, a contradiction. If $l=j$ and $m \neq 0$, then $(X_j - q^mX_j)f_{v,k'}=0$, and since $1-q^m$ is not a zero-divisor in $\kk$, this also implies by \cref{thm:DAFHA_basis} that $f_{v,k'}=0$, a contradiction. Therefore, $k' = 0$ and $v = 1$, by which $h \in P_n^X(A)$.
\end{proof}

\begin{lemma} \label{thm:XA_centraliser}
For $q \in \kk^\times$ such that $q^m-1$ is not a zero divisor in $\kk$ for all $m \in \ZZ \setminus \{0\}$, the supercentraliser of $P_n^X(A)$ in $\HH_n(A,z,q)$ is $Z(A^{\otimes n}) \otimes P_n^X$. If $A$ and $Z(A)$ satisfy the conditions of \cref{thm:decompose_supercentre}, then this is equal to $P_n^X(Z(A))$.
\end{lemma}
\begin{proof}
The supercentraliser of $P_n^X(A)$ in $\HH_n(A,z,q)$ is contained in the supercentraliser of $P_n^X$, which is $P_n^X(A)$ by \cref{thm:X_centraliser}. Thus, the supercentraliser of $P_n^X(A)$ is equal to
\[
Z(P_n^X(A))\overset{\eqref{dafha_rel_9}}{=} Z(A^{\otimes n} \otimes P_n^X)=Z(A^{\otimes n}) \otimes Z(P_n^X).
\]
If $A$ and $Z(A)$ satisfy the conditions of \cref{thm:decompose_supercentre}, then \cref{thm:decompose_supercentre} implies $Z(A^{\otimes n}) \otimes Z(P_n^X) = Z(A)^{\otimes n} \otimes P_n^X\overset{\eqref{dafha_rel_9}}{=} P_n^X(Z(A))$.
\end{proof}

\begin{corollary} \label{thm:Y_YA_centraliser}
For $q \in \kk^\times$ such that $q^m-1$ is not a zero divisor in $\kk$ for all $m \in \ZZ \setminus \{0\}$, the supercentraliser of $P_n^Y$ in $\HH_n(A,z,q)$ is $P_n^Y(A)$. The supercentraliser of $P_n^Y(A)$ in $\HH_n(A,z,q)$ is $Z(A^{\otimes n}) \otimes P_n^Y$. If $A$ and $Z(A)$ satisfy the conditions of \cref{thm:decompose_supercentre}, then this is equal to $P_n^Y(Z(A))$.
\end{corollary}
\begin{proof}
The two statements follow from \cref{thm:X_centraliser} and \cref{thm:XA_centraliser} respectively, by the symmetry of the $X$ and the $Y$ generators of $\HH_n(A,z,q)$. In particular, the $\kk$-superalgebra isomorphism $\epsilontransform$ from \cref{thm:XY_automorphisms} maps $P_n^X$ to $P_n^Y$, $P_n^X(A)$ to $P_n^Y(A)$ and $P_n^X(Z(A))$ to $P_n^Y(Z(A))$, whilst changing the parameter $z$ to $-z$ and $q$ to $q^{-1}$. The definitions of the supercentralisers are independent of the parameters, and $q^{-m}-1$ is not a zero divisor in $\kk$ if $q^m-1$ is not one.
\end{proof}

\begin{proposition} \label{thm:dafha_centre}
For $q \in \kk^\times$ such that $q^m-1$ is not a zero divisor in $\kk$ for all $m \in \ZZ \setminus \{0\}$, the supercentre of $\HH_n(A,z,q)$ is $Z\left(A^{\otimes n} \right)^{S_n}$. If $A$ and $Z(A)$ satisfy the conditions of \cref{thm:decompose_supercentre}, then this is equal to $\left(Z(A)^{\otimes n} \right)^{S_n}$.
\end{proposition}
\begin{proof}
It is clear that elements of $Z\left(A^{\otimes n} \right)^{S_n}$ are central in $\HH_n(A,z,q)$. By \cref{thm:XA_centraliser} and \cref{thm:Y_YA_centraliser}, the supercentre of $\HH_n(A,z,q)$ must lie in $\left( Z\left(A^{\otimes n} \right) \otimes P_n^X \right) \cap \left( Z\left(A^{\otimes n} \right) \otimes P_n^Y \right) = Z(A^{\otimes n})$, where the equality follows from \cref{thm:DAFHA_basis}. By Relation \eqref{dafha_rel_4}, elements of the supercentre must be $S_n$-invariant, hence in $Z\left(A^{\otimes n} \right)^{S_n}$. If $A$ and $Z(A)$ satisfy the conditions of \cref{thm:decompose_supercentre}, then \cref{thm:decompose_supercentre} implies $Z\left(A^{\otimes n} \right)^{S_n} = \left(Z(A)^{\otimes n} \right)^{S_n}$.
\end{proof}

\section{Polynomial Representations} \label{sec:DAFHA_poly_rep_phi}

We assume throughout the section that $A$ has a one-dimensional representation $\phi \colon A \to \kk$ that is an even homomorphism of $\kk$-superalgebras, and we let $\phi^{\otimes n} \colon A^{\otimes n} \to \kk$ be the representation of $A^{\otimes n}$ defined by $a_1 \otimes \dotsb \otimes a_n \mapsto \phi(a_1) \dotsb \phi(a_n)$. We also assume that the quadratic equation
\begin{equation}
r^2=z\phi^{\otimes n}(t_{i,i+1})r+1, \label{eq:r_quadratic}
\end{equation}
has a solution $r \in \kk$, and fix this solution. We define projection maps of the Frobenius Hecke algebras defined in this document, onto the usual $A=\kk$ Hecke algebras. It is easily verified by checking Relations \eqref{qwa_rel_1} to \eqref{qwa_rel_4} that this defines a homomorphism of $\kk$-superalgebras
\begin{gather}  \label{eq:one_dim_r_rep} 
\Phi \colon H_n(A,z) \to H_n(\kk,\phi^{\otimes n}(t_{1,2})z),\\
\mathbf{a} \mapsto \phi^{\otimes n}(\mathbf{a}), \qquad T_i \mapsto T_i,
\end{gather}
for $\mathbf{a} \in A^{\otimes n}$ and $1 \leq i \leq n-1$. This map is surjective, as all the generators $T_i$ are mapped onto. Moreover, it is easily verified that $\Phi$ extends to surjections of $\kk$-superalgebras,
\begin{gather*} 
\Phi^\aff \colon H_n^\aff(A,z) \to H_n^\aff(\kk,\phi^{\otimes n}(t_{1,2})z), \\
\mathbf{a} \mapsto \phi^{\otimes n}(\mathbf{a}), \qquad T_i \mapsto T_i, \qquad Y_j \mapsto Y_j,
\end{gather*}
and
\begin{gather*} 
\Phi^\affX \colon H_n^\affX(A,z) \to H_n^\affX(\kk,\phi^{\otimes n}(t_{1,2})z), \\
\mathbf{a} \mapsto \phi^{\otimes n}(\mathbf{a}), \qquad T_i \mapsto T_i, \qquad X_j \mapsto X_j,
\end{gather*}
which in turn extend to a surjection of $\kk$-superalgebras,
\begin{gather*} 
\Phi^\Daff \colon \HH_n(A,z,q) \to \HH_n(\kk,\phi^{\otimes n}(t_{1,2})z,q), \\
\mathbf{a} \mapsto \phi^{\otimes n}(\mathbf{a}), \qquad T_i \mapsto T_i, \qquad Y_j \mapsto Y_j, \qquad X_j \mapsto X_j,
\end{gather*}
for $\mathbf{a} \in A^{\otimes n}$, $1 \leq i \leq n-1$ and $1 \leq j \leq n$.

\begin{remark} \label{rem:usual_poly_rep}
We recall the definition of the usual polynomial representation of $H_n^\aff(\kk,\tau-\tau^{-1})$. As $H_n(\kk,\tau-\tau^{-1})$ is an ordinary Iwahori--Hecke algebra, it acts on $\kk$ by $T_i \mapsto \tau$ for $1 \leq i \leq n-1$, which induces a $\kk$-linear representation $\gamma^\aff$ of $H_n^\aff(\kk,\tau-\tau^{-1})$ with underlying space $\kk[Y_1^{\pm 1},\dotsc,Y_n^{\pm 1}]$, called the polynomial representation. The action of $T_i$ in the polynomial representation follows from \cref{thm:lusztig_relation_Daff} with $A=\kk$ and $z=\tau-\tau^{-1}$.
\end{remark}

By choosing $\tau=r$ satisfying \eqref{eq:r_quadratic}, we may identify $H_n^\aff(\kk,\phi^{\otimes n}(t_{1,2})z)$ with $H_n^\aff(\kk,\tau-\tau^{-1})$, and pull back the usual polynomial representation $\gamma^\aff$ from \cref{rem:usual_poly_rep} by $\Phi^\aff$ to obtain a polynomial representation of $H_n^\aff(A,z)$ that depends on the choice of $\phi$. This can be done similarly with the usual polynomial representation $\gamma^\affX$ of $H_n^\affX(\kk,\tau-\tau^{-1})$.

\begin{definition} \label{def:affine_poly_rep_phi}
The \emph{$\phi$-polynomial representation of $H_n^\aff(A,z)$} is the representation with underlying space $\kk[Y_1^{\pm 1},\dotsc,Y_n^{\pm 1}]$, given by
\begin{equation} \label{eq:affine_poly_rep_phi_X}
\begin{aligned}
\gamma_\phi^\aff := \gamma^\aff \circ \Phi^\aff \colon H_n^\aff(A,z) &\to \Endo_\kk(\kk[Y_1^{\pm 1},\dotsc,Y_n^{\pm 1}]),\\
\mathbf{a} &\mapsto \phi^{\otimes n}(\mathbf{a}),\\
T_i &\mapsto rs_i + z \phi^{\otimes n}(t_{i,i+1}) \frac{s_i-1}{Y_i^{-1}Y_{i+1}-1},\\
Y_j &\mapsto Y_j,
\end{aligned}
\end{equation}
for $\mathbf{a} \in A^{\otimes n}$, $1 \leq i \leq n-1$ and $1 \leq j \leq n$. The \emph{$\phi$-polynomial representation of $H_n^\affX(A,z)$} has underlying space $\kk[X_1^{\pm 1},\dotsc,X_n^{\pm 1}]$, and is given by
\begin{align*}
\gamma_\phi^\affX := \gamma^\affX \circ \Phi^\affX \colon H_n^\affX(A,z) &\to \Endo_\kk(\kk[X_1^{\pm 1},\dotsc,X_n^{\pm 1}]),\\
\mathbf{a} &\mapsto \phi^{\otimes n}(\mathbf{a}),\\
T_i &\mapsto rs_i + z \phi^{\otimes n}(t_{i,i+1}) \frac{s_i-1}{X_iX_{i+1}^{-1}-1},\\
X_j &\mapsto X_j,
\end{align*}
for $\mathbf{a} \in A^{\otimes n}$, $1 \leq i \leq n-1$ and $1 \leq j \leq n$.
\end{definition}

\begin{remark} \label{rem:usual_DAHA_poly_rep}
We recall the definition of the polynomial representation of $\HH_n(\kk,\tau-\tau^{-1},q)$, which is an ordinary double affine Hecke algebra. As $H_n^\aff(\kk,\tau-\tau^{-1})$ acts on $\kk$ by $T_i \mapsto \tau$ for $0 \leq i \leq n-1$ and $\pi \mapsto 1$ (or equivalently, $Y_i \overset{\eqref{eq:Y_i_formula}}{\mapsto} \tau^{n+1-2i}$), this induces a $\kk$-linear representation $\gamma$ of $\HH_n(\kk,\tau-\tau^{-1},q)$ with underlying space $\kk[X_1^{\pm 1},\dotsc,X_n^{\pm 1}]$, called the polynomial representation. The operators $\gamma(T_1),\dotsc,\gamma(T_{n-1})$ are analogous to those in \cref{def:affine_poly_rep_phi}, the action of the operator $\gamma(T_0)$ can be immediately deduced using \cref{thm:lusztig_relation_Daff}, and the action of $\gamma(\pi)$ can be deduced using Relations \eqref{dafha_rel_10} and \eqref{dafha_rel_10b}.
\end{remark}

By choosing $\tau=r$ satisfying \eqref{eq:r_quadratic}, we may identify $\HH_n(\kk,\phi^{\otimes n}(t_{1,2})z,q)$ with $\HH_n(\kk,\tau-\tau^{-1},q)$, and pull back the usual polynomial representation $\gamma$ from \cref{rem:usual_DAHA_poly_rep} by $\Phi^\Daff$ to obtain a polynomial representation of $\HH_n(A,z,q)$ that depends on the choice of $\phi$.

\begin{definition} \label{def:DAFHA_poly_rep_phi}
Let $q \in \kk^\times$ be an invertible element. The \emph{$\phi$-polynomial representation} of $\HH_n(A,z,q)$ has underlying space $\kk[X_1^{\pm 1},\dotsc,X_n^{\pm 1}]$, and is given by
\begin{equation}\label{eq:DAFHA_poly_rep_phi}
\begin{aligned}
\gamma_\phi \colon \HH_n(A,z,q) &\to \Endo_\kk(\kk[X_1^{\pm 1}, \dotsc ,X_n^{\pm 1}]),\\
\mathbf{a} &\mapsto \phi^{\otimes n}(\mathbf{a}),\\
X_j &\mapsto X_j,\\
T_i &\mapsto rs_i + z\phi^{\otimes n}(t_{i,i+1})\frac{s_i-1}{X_is_i(X_i)^{-1}-1},\\
\pi &\mapsto \pi,
\end{aligned}
\end{equation}
for $\mathbf{a} \in A^{\otimes n}$, $0 \leq i \leq n-1$ and $1 \leq j \leq n$. In the above, we define $t_{0,1}:=t_{n,1}$ and $X_0:=X_n$. The action by $s_0$ and $\pi$ on $\kk[X_1^{\pm 1},\dotsc,X_n^{\pm 1}]$ involves multiplication by a power of $q$, as described in \eqref{eq:s_0_operator} and \eqref{eq:pi_operator}.
\end{definition}

The operators
\begin{equation} \label{eq:Dunkl_operators}
\gamma_\phi(Y_i) \overset{\eqref{eq:Y_i_formula}}{=} \gamma_\phi(T_{i-1}^{-1})\dotsb \gamma_\phi(T_1^{-1})\gamma_\phi(\pi)\gamma_\phi(T_{n-1})\dotsb \gamma_\phi(T_i),
\end{equation}
are the usual $q$-difference Dunkl operators.

\begin{definition} \label{def:phi_basic_representation}
Let $q \in \kk^\times$ be an invertible element. The \emph{$\phi$-Cherednik's basic representation} of $H_n^\aff(A,z)$ is the restriction
\[
\gamma_\phi |_{H_n^\aff(A,z)} \colon H_n^\aff(A,z) \to \Endo_\kk(\kk[X_1^{\pm 1},\dotsc,X_n^{\pm 1}]).
\]
When $r=\tau \in \kk^\times$, this is the pullback of the usual basic representation by $\Phi^\aff$. Note that both $\gamma_\phi$ and this restriction depend on the choice of $q$, though $H_n^\aff(A,z)$ does not. 
\end{definition}

\begin{example} \label{ex:DAHA_poly_rep}
Let $A=\kk$, $z=\tau-\tau^{-1}$. Then $r^2=(\tau-\tau^{-1})r+1$ has solutions $r_1=\tau$ and $r_2=-\tau^{-1}$. Take $r=r_1=\tau$. Let $\phi = \id_{\kk} \colon \kk \to \kk$. Then the polynomial representations from \cref{def:affine_poly_rep_phi} and from \cref{def:DAFHA_poly_rep_phi} are the usual polynomial representation of the affine Hecke algebra and of the usual DAHA, respectively. The basic representation from \cref{def:phi_basic_representation} is the usual basic representation.
\end{example}

\begin{example} \label{ex:Wan_Wang_poly_rep}
Let $\kk$ have characteristic $p \neq 2$ and let $\tau \in \kk^\times$. Let $A=\kk G$, for $G$ a finite group with $p \nmid |G|$, with trace as in \cref{ex:double_affine_Wan_Wang}. Take $z=(\tau-\tau^{-1})/|G|$, which is rescaled from the parameter of \cref{ex:double_affine_Wan_Wang} by the factor $\tfrac{1}{|G|}$, as in \cref{ex:double_affine_Yokonuma_Hecke}. For $g \in G$, denote $g_i=1^{i-1} \otimes g \otimes 1^{n-i} \in A^{\otimes n}$. Then the teleporter is 
\[
t_{i,i+1} = \sum_{g \in G} g_i g_{i+1}^{-1}.
\]
Given a linear character of $G$, which extends to a one-dimensional representation $\phi \colon A \to \kk$, we have $\phi^{\otimes n}(t_{i,i+1}) = \sum_{g \in G} \phi(g) \phi(g^{-1}) = \sum_{g \in G} \phi(g) \phi(g)^{-1} = |G|$. The quadratic equation \eqref{eq:r_quadratic} is thus $r^2=(\tau-\tau^{-1})r+1$ in this example, with solutions $r_1=\tau$ and $r_2=-\tau^{-1}$. Thus, having chosen a linear character of $G$ and either $r_1$ or $r_2$, \cref{def:affine_poly_rep_phi} and \cref{def:phi_basic_representation} are polynomial representations of the non-degenerate versions $H_n^\affX(\kk G,(\tau-\tau^{-1})/|G|)$ and $H_n^\aff(\kk G,(\tau-\tau^{-1})/|G|)$ of Wan and Wang's algebra from \cite[Introduction,~Example~(b)]{QAWA_2020}, with a rescaled parameter. The representation from \cref{def:DAFHA_poly_rep_phi} is of the algebra $\HH_n(\kk G,(\tau-\tau^{-1})/|G|,q)$, which is defined in \cref{ex:double_affine_Wan_Wang} with a rescaled parameter. Neither polynomial representation seems to have appeared in the current literature.
\end{example}

\begin{example} \label{ex:Yokonuma_poly_rep}
This example is a special case of \cref{ex:Wan_Wang_poly_rep}. Let $\kk=\CC[\tau,\tau^{-1}]$, $A=\kk C_d$, for $C_d$ the cyclic group of order $d$ and generated by $g \in C_d$, and $z=(\tau-\tau^{-1})/d$, with trace as in \cref{ex:double_affine_Wan_Wang}. Denoting $g_i=1^{i-1} \otimes g \otimes 1^{n-i} \in A^{\otimes n}$, the teleporter is 
\[
t_{i,i+1} = \sum_{k=0}^{d-1} g_i^k g_{i+1}^{-k}.
\]
The one-dimensional representations of $A$ are precisely of the form 
\[
\phi_j \colon A \to \kk, \qquad g \mapsto e^{2\pi \iu j/d},
\]
for $0 \leq j \leq d-1$, where $\iu$ is the imaginary unit. Moreover, we have
\[
\phi_j^{\otimes n}(t_{i,i+1}) = \sum_{k=0}^{d-1} e^{2\pi \iu jk/d} e^{2\pi \iu j(-k)/d} = d.
\]
The quadratic equation $r^2=(\tau-\tau^{-1})r+1$ has solutions $r_1=\tau$ and $r_2=-\tau^{-1}$. Having chosen $j$ and either $r_1$ or $r_2$, \cref{def:affine_poly_rep_phi} and \cref{def:DAFHA_poly_rep_phi} give us one of $2d$ distinct polynomial representations of the affine Yokonuma--Hecke algebra $H_n^\affX(\kk C_d,(\tau-\tau^{-1})/d)$ from \cite[Sec.~3.1]{RTYHA_2014} and of the double affine Yokonuma--Hecke algebra $\HH_n(\kk C_d,(\tau-\tau^{-1})/d,q)$ from \cref{ex:double_affine_Yokonuma_Hecke}, respectively. Neither polynomial representation seems to have appeared in the current literature.
\end{example}

\section{\texorpdfstring{$A$}{A}-decorated Polynomial Representations at \texorpdfstring{$z=0$}{z=0}} \label{sec:DAFHA_poly_rep_phi_A}

The representations of \cref{sec:DAFHA_poly_rep_phi} are defined given a $\kk$-linear representation $\phi \colon A \to \kk$, as pullbacks of the polynomial representations of the usual $A=\kk$ affine Hecke algebras and DAHAs. When $z=0$, there are additionally non-collapsed analogues of these representations, that do not depend on the existence of such $\phi$. We can obtain these so-called $A$-decorated polynomial representations much like the $A=\kk$ versions from \cref{rem:usual_poly_rep} and \cref{rem:usual_DAHA_poly_rep}, by defining representations of $H_n(A,0)$ and $H_n^\aff(A,0)$ on $A^{\otimes n}$ and inducing to $H_n^\aff(A,0)$ and $\HH_n(A,0,q)$ respectively. As $H_n^\aff(\kk,0)$ appears as a subalgebra of the trigonometric degeneration of the usual DAHA, the representations from \cref{def:affine_poly_rep_A} will appear as restrictions of a polynomial representation of the algebra $\HH_n^\textup{trig}(A,c,\kappa)$ from \cref{dia:introduction} in future work.

It is easily verified by checking Relations \eqref{qwa_rel_1} to \eqref{qwa_rel_4} that we can define a representation of the $z=0$ Frobenius Hecke algebra $H_n(A,0)$, given by
\begin{align} 
\dot{\gamma}_A \colon H_n(A,0) \to \Endo_\kk(A^{\otimes n}), \qquad  \mathbf{a} \mapsto \mathbf{a}, \qquad T_i \mapsto s_i, \label{eq:one_dim_r_rep_A}
\end{align}
for $\mathbf{a} \in A^{\otimes n}$ and $1 \leq i \leq n-1$. As there is a canonical map of $\kk$-superalgebras, $H_n(A,0) \hookrightarrow{} H_n^\aff(A,0)$, which is an embedding by \cref{thm:AFHA_coxeter_basis}, we may induce a representation of $H_n^\aff(A,0)$ from $\dot{\gamma}_A$. Moreover, \cref{thm:lusztig_relation_Daff} can be used to deduce a description of this representation, which is analogous to \cref{def:affine_poly_rep_phi}.

\begin{definition} \label{def:affine_poly_rep_A}
The \emph{$A$-decorated polynomial representation of $H_n^\aff(A,0)$} is the representation $\Ind_{H_n(A,0)}^{H_n^\aff(A,0)}(\dot{\gamma}_A)$, given by
\begin{equation} \label{eq:affine_poly_rep_A}
\begin{aligned}
\gamma_A^\aff \colon H_n^\aff(A,0) &\to \Endo_\kk(\kk[Y_1^{\pm 1},\dotsc,Y_n^{\pm 1}] \otimes A^{\otimes n}),\\
\mathbf{a} &\mapsto 1 \otimes \mathbf{a},\\
T_i &\mapsto s_i \otimes s_i,\\
Y_j &\mapsto Y_j \otimes 1,
\end{aligned}
\end{equation}
for $\mathbf{a} \in A^{\otimes n}$, $1 \leq i \leq n-1$ and $1 \leq j \leq n$. The \emph{$A$-decorated polynomial representation of $H_n^\affX(A,0)$} is $\Ind_{H_n(A,0)}^{H_n^\affX(A,0)}(\dot{\gamma}_A)$, given by
\begin{align*}
\gamma_A^\affX \colon H_n^\affX(A,0) &\to \Endo_\kk(\kk[X_1^{\pm 1},\dotsc,X_n^{\pm 1}] \otimes A^{\otimes n}),\\
\mathbf{a} &\mapsto 1 \otimes \mathbf{a},\\
T_i &\mapsto s_i \otimes s_i,\\
X_j &\mapsto X_j \otimes 1,
\end{align*}
for $\mathbf{a} \in A^{\otimes n}$, $1 \leq i \leq n-1$ and $1 \leq j \leq n$.
\end{definition}

It is easily verified that $\dot{\gamma}_A$ can be extended to the representation
\begin{align} 
\gamma_A' \colon H_n^\aff(A,0) \to \Endo_\kk(A^{\otimes n}), \quad \mathbf{a} \mapsto \mathbf{a}, \quad T_i \mapsto s_i, \quad \pi \mapsto \sigma, \label{eq:one_dim_r_rep_AY}
\end{align}
for $\mathbf{a} \in A^{\otimes n}$ and $0 \leq i \leq n-1$. Note that $\gamma_A'(Y_i) \overset{\eqref{eq:Y_i_formula}}{=} 1$, since $s_{i-1} \dotsb s_1 \sigma s_{n-1} \dotsb s_i = 1 \in S_n$. Using this, we can define $A$-decorated analogues to \cref{def:DAFHA_poly_rep_phi} and \cref{def:phi_basic_representation} for $z=0$.

\begin{definition} \label{def:double_affine_poly_rep_A}
The \emph{$A$-decorated polynomial representation of $\HH_n(A,0,q)$} is the representation $\Ind_{H_n^\aff(A,0)}^{\HH_n(A,0,q)}(\gamma_A')$, given by
\begin{equation} \label{eq:double_affine_poly_rep_A}
\begin{aligned}
\gamma_A \colon \HH_n(A,0,q) &\to \Endo_\kk(\kk[X_1^{\pm 1},\dotsc,X_n^{\pm 1}] \otimes A^{\otimes n}),\\
\mathbf{a} &\mapsto 1 \otimes \mathbf{a},\\
T_i &\mapsto s_i \otimes s_i,\\
X_j &\mapsto X_j \otimes 1,\\
\pi &\mapsto \pi \otimes \sigma,
\end{aligned}
\end{equation}
for $\mathbf{a} \in A^{\otimes n}$, $0 \leq i \leq n-1$ and $1 \leq j \leq n$. The action by $s_0$ and $\pi$ on $\kk[X_1^{\pm 1},\dotsc,X_n^{\pm 1}]$ involves a factor of a power of $q$, as described in \eqref{eq:s_0_operator} and \eqref{eq:pi_operator}. We denote $\sigma = (1 \; 2 \; \dotsb \; n) \in S_n$.
\end{definition}

\begin{definition} \label{def:basic_representation_A}
Let $q \in \kk^\times$ be an invertible element. The \emph{$A$-decorated Cherednik's basic representation} of $H_n^\aff(A,0)$ is the restriction
\[
\gamma_A |_{H_n^\aff(A,0)} \colon H_n^\aff(A,0) \to \Endo_\kk(\kk[X_1^{\pm 1},\dotsc,X_n^{\pm 1}] \otimes A^{\otimes n}).
\]
Note that both $\gamma_A$ and this restriction depend on the choice of $q$, though $H_n^\aff(A,0)$ does not.
\end{definition}

\section{Diagrammatic Interpretation of \texorpdfstring{$\HH_n(A,z,q)$}{ℍₙ(A,z,q)}} \label{sec:diagrammatic_interpretation}

In this section, we give a heuristic diagrammatic interpretation of $\HH_n(A,z,q)$, which forms a useful motivation for the symmetries of \cref{sec:symmetries_for_DAFHA}, and a framework for future work. We do not, however, prove the faithfulness of this interpretation. The interpretation can be restricted to diagrammatic interpretations of the subalgebras $H_n(A,z)$, $H_n^\aff(A,z)$ and $H_n^\affX(A,z)$, for which the corresponding diagrams can be realised as endomorphisms in suitable monoidal categories. We can view the monomial elements of $\HH_n(A,z,q)$ as string diagrams lying on a thickened punctured torus $\mathbb{T}^2_\text{punc} \times [0,1]$, with arbitrary elements of $\HH_n(A,z,q)$ represented by formal $\kk$-linear sums of such diagrams. To illustrate this, consider the diagram
\begin{equation} \label{dia:torus_3D_diagram}
T_{\text{id}} =
    \begin{tikzpicture}[anchorbase,scale=0.58]
    	\draw (-\torwidth/4,-\torheight/4) to[out=up,in=down] (-\torwidth/4,-\torheight/4+\torthickness);
        \draw (-\torwidth/2,-\torheight/2) to[out=up,in=down] (-\torwidth/2,-\torheight/2+\torthickness);
        \draw (0,0) to[out=up,in=down] (0,\torthickness);
        \draw (\torwidth/2,\torheight/2) to[out=up,in=down] (\torwidth/2,\torheight/2+\torthickness);
        \draw (\torwidth/4,\torheight/4) to[out=up,in=down] (\torwidth/4,\torheight/4+\torthickness);
        \draw[DarkGreen,dashed] (-\torwidth,-\torheight+\torthickness) -- (\torwidth,\torheight+\torthickness);
        \draw[DarkGreen,dashed] (-\torwidth,-\torheight) -- (\torwidth,\torheight);
		\puncturedtorus
    \end{tikzpicture} \; ,
\end{equation}
representing the identity element of $\HH_5(A,z,q)$. The red faces are to be identified point-by-point to each other, as are the blue faces, to make a thickened torus, with a puncture where the corners of the above diagram are. The strings are placed on the diagonal, as illustrated by the dashed line, so that this three-dimensional image can be considered in plan-view from either the front red face or the right-hand-side blue face of Diagram \eqref{dia:torus_3D_diagram}, to describe certain braids more clearly. Viewing from the front or the right face, for an elementary tensor $\mathbf{a} = a_1 \otimes \dotsb \otimes a_n \in A^{\otimes n}$, let
\begin{equation} \label{eq:notation_tilde_a}
\mathbf{a} := \begin{tikzpicture}[centerzero]
        \draw (0,-0.3) node[anchor=north] {\dotlabel{1}} -- (0,0.3);
        \token{0,0}{west}{a_1};
        \draw (1,-0.3) node[anchor=north] {\dotlabel{2}} -- (1,0.3);
        \token{1,0}{west}{a_2};
    \end{tikzpicture}
    \ \hdots \
    \begin{tikzpicture}[centerzero]
        \draw (0,-0.3) node[anchor=north] {\dotlabel{n-1}} -- (0,0.3);
        \token{0,0}{west}{a_{n-1}};
        \draw (1,-0.3) node[anchor=north] {\dotlabel{n}} -- (1,0.3);
        \token{1,0}{west}{a_n};
    \end{tikzpicture}\; ,
\end{equation}
and for $1 \leq i \leq n-1$, let
\begin{equation} \label{eq:notation_tilde_T_i}
T_i := \begin{tikzpicture}[centerzero]
        \draw (0.3,-0.3) node[anchor=north] {\dotlabel{1}} -- (0.3,0.3);
    \end{tikzpicture}
    \dotsb
	\begin{tikzpicture}[centerzero]
		\draw (-0.7,-0.3) -- (-0.7,0.3);
        \draw (-0.3,-0.3) node[anchor=north] {\dotlabel{i}} -- (0.3,0.3);
        \draw[wipe] (0.3,-0.3) -- (-0.3,0.3);
        \draw (0.3,-0.3) node[anchor=north] {\dotlabel{i+1}} -- (-0.3,0.3);
        \draw (0.7,-0.3) -- (0.7,0.3);
    \end{tikzpicture}
    \dotsb
    \begin{tikzpicture}[centerzero]
        \draw (0.3,-0.3) node[anchor=north] {\dotlabel{n}} -- (0.3,0.3);
    \end{tikzpicture}
    \quad \text{ and } \quad 
    T_i^{-1} := \begin{tikzpicture}[centerzero]
        \draw (0.3,-0.3) node[anchor=north] {\dotlabel{1}} -- (0.3,0.3);
    \end{tikzpicture}
    \dotsb
    \begin{tikzpicture}[centerzero]
    	\draw (-0.7,-0.3) -- (-0.7,0.3);
        \draw (0.3,-0.3) node[anchor=north] {\dotlabel{i+1}} -- (-0.3,0.3);
        \draw[wipe] (-0.3,-0.3) -- (0.3,0.3);
        \draw (-0.3,-0.3) node[anchor=north] {\dotlabel{i}} -- (0.3,0.3);
        \draw (0.7,-0.3) -- (0.7,0.3);
    \end{tikzpicture}
    \dotsb
    \begin{tikzpicture}[centerzero]
        \draw (0.3,-0.3) node[anchor=north] {\dotlabel{n}} -- (0.3,0.3);
    \end{tikzpicture}\; .
\end{equation}
Note that these define the same braids in the three-dimensional image whether it is viewed from the front red face or from the right-hand blue face of Diagram \eqref{dia:torus_3D_diagram}. We use a dot $\begin{tikzpicture}[centerzero]
        \draw (0,-0.3) -- (0,0.3);
        \opendot{0,0} node[anchor=west] {\dotlabel{1}};
    \end{tikzpicture}$ on a string to denote a loop around the meridian of the torus, through the red faces of Diagram \eqref{dia:torus_3D_diagram}, and we use a dot $\begin{tikzpicture}[centerzero]
        \draw (0,-0.3) -- (0,0.3);
        \greydot{0,0} node[anchor=west] {\dotlabel{1}};
    \end{tikzpicture}$ on a string to denote a loop around the longitude of the torus, through the blue faces of the diagram. For example, in $H_5^\aff(A,z)$ we have, viewing from the right-hand-side blue face, 
\[
    \begin{tikzpicture}[centerzero]
    	\draw (-0.8,-0.5) -- (-0.8,0.5);
        \draw (-0.4,-0.5) -- (-0.4,0.5);
        \draw (0,-0.5) -- (0,0.5);
        \draw (0.4,-0.5) -- (0.4,0.5);
        \draw (0.8,-0.5) -- (0.8,0.5);
        \opendot{0,0} node[anchor=west] {\dotlabel{1}};
    \end{tikzpicture}
    \ = \
    \begin{tikzpicture}[centerzero,xscale=-1]
        \draw (0.8,-0.5) -- (0.8,0.5);
        \draw (0.4,-0.5) -- (0.4,0.5);
        \draw (0,-0.5) to[out=up,in=-20] (-1.2,0);
        \draw[wipe] (1.2,0) to[out=160,in=down] (0,0.5);
        \draw (1.2,0) to[out=160,in=down] (0,0.5);
        \draw[wipe] (-0.4,-0.5) -- (-0.4,0.5);
        \draw (-0.4,-0.5) -- (-0.4,0.5);
        \draw[wipe] (-0.8,-0.5) -- (-0.8,0.5);
        \draw (-0.8,-0.5) -- (-0.8,0.5);
        \identify{-1.2}{-0.5}{1.2}{0.5};
    \end{tikzpicture}\;
    \qquad \text{ and } \qquad
    \begin{tikzpicture}[centerzero]
    	\draw (-0.8,-0.5) -- (-0.8,0.5);
        \draw (-0.4,-0.5) -- (-0.4,0.5);
        \draw (0,-0.5) -- (0,0.5);
        \draw (0.4,-0.5) -- (0.4,0.5);
        \draw (0.8,-0.5) -- (0.8,0.5);
        \opendot{0,0} node[anchor=west,xshift=-4pt,yshift=4pt] {\dotlabel{-1}};
    \end{tikzpicture}
    \ = \
    \begin{tikzpicture}[centerzero,xscale=-1]
    	\draw (0.8,-0.5) -- (0.8,0.5);
        \draw (0.4,-0.5) -- (0.4,0.5);
        \draw[wipe] (0,-0.5) to[out=up,in=200] (1.2,0);
        \draw (0,-0.5) to[out=up,in=200] (1.2,0);
        \draw (-1.2,0) to[out=20,in=down] (0,0.5);
        \draw[wipe] (-0.4,-0.5) -- (-0.4,0.5);
        \draw (-0.4,-0.5) -- (-0.4,0.5);
        \draw[wipe] (-0.8,-0.5) -- (-0.8,0.5);
        \draw (-0.8,-0.5) -- (-0.8,0.5);
        \identify{-1.2}{-0.5}{1.2}{0.5};
    \end{tikzpicture}\; ,
\]
where the dotted red lines represent the red faces of Diagram \eqref{dia:torus_3D_diagram}. In $H_5^\affX(A,z)$ we have, viewing from the front red face,
\[
    \begin{tikzpicture}[centerzero]
    	\draw (-0.8,-0.5) -- (-0.8,0.5);
        \draw (-0.4,-0.5) -- (-0.4,0.5);
        \draw (0,-0.5) -- (0,0.5);
        \draw (0.4,-0.5) -- (0.4,0.5);
        \draw (0.8,-0.5) -- (0.8,0.5);
        \greydot{0,0} node[anchor=west] {\dotlabel{1}};
    \end{tikzpicture}
    \ = \
    \begin{tikzpicture}[centerzero]
    	\draw (0.8,-0.5) -- (0.8,0.5);
        \draw (0.4,-0.5) -- (0.4,0.5);
        \draw[wipe] (0,-0.5) to[out=up,in=200] (1.2,0);
        \draw (0,-0.5) to[out=up,in=200] (1.2,0);
        \draw (-1.2,0) to[out=20,in=down] (0,0.5);
        \draw[wipe] (-0.4,-0.5) -- (-0.4,0.5);
        \draw (-0.4,-0.5) -- (-0.4,0.5);
        \draw[wipe] (-0.8,-0.5) -- (-0.8,0.5);
        \draw (-0.8,-0.5) -- (-0.8,0.5);
        \identifyX{-1.2}{-0.5}{1.2}{0.5};
    \end{tikzpicture}\;
    \qquad \text{ and } \qquad
    \begin{tikzpicture}[centerzero]
    	\draw (-0.8,-0.5) -- (-0.8,0.5);
        \draw (-0.4,-0.5) -- (-0.4,0.5);
        \draw (0,-0.5) -- (0,0.5);
        \draw (0.4,-0.5) -- (0.4,0.5);
        \draw (0.8,-0.5) -- (0.8,0.5);
        \greydot{0,0} node[anchor=west,xshift=-4pt,yshift=4pt] {\dotlabel{-1}};
    \end{tikzpicture}
    \ = \
    \begin{tikzpicture}[centerzero]
        \draw (0.8,-0.5) -- (0.8,0.5);
        \draw (0.4,-0.5) -- (0.4,0.5);
        \draw (0,-0.5) to[out=up,in=-20] (-1.2,0);
        \draw[wipe] (1.2,0) to[out=160,in=down] (0,0.5);
        \draw (1.2,0) to[out=160,in=down] (0,0.5);
        \draw[wipe] (-0.4,-0.5) -- (-0.4,0.5);
        \draw (-0.4,-0.5) -- (-0.4,0.5);
        \draw[wipe] (-0.8,-0.5) -- (-0.8,0.5);
        \draw (-0.8,-0.5) -- (-0.8,0.5);
        \identifyX{-1.2}{-0.5}{1.2}{0.5};
    \end{tikzpicture}\;.
\]
For any $\lambda = (\lambda_1,\dotsc,\lambda_n) \in \ZZ^n$, let
\begin{equation} \label{eq:notation_Y_lambda}
Y^\lambda : = \begin{tikzpicture}[centerzero]
        \draw (0.3,-0.3) node[anchor=north] {\dotlabel{1}} -- (0.3,0.3);
        \opendot{0.3,0} node[anchor=west] {\dotlabel{\lambda_1}};
    \end{tikzpicture}
    \ \hdots \
    \begin{tikzpicture}[centerzero]
    	\draw (-1,-0.3) node[anchor=north] {\dotlabel{i-1}} -- (-1,0.3);
    	\opendot{-1,0} node[anchor=west] {\dotlabel{\lambda_{i-1}}};
		\draw (0,-0.3) node[anchor=north] {\dotlabel{i}} -- (0,0.3);
		\opendot{0,0} node[anchor=west] {\dotlabel{\lambda_i}};
        \draw (1,-0.3) node[anchor=north] {\dotlabel{i+1}} -- (1,0.3);
        \opendot{1,0} node[anchor=west] {\dotlabel{\lambda_{i+1}}};
    \end{tikzpicture}
    \ \hdots \
    \begin{tikzpicture}[centerzero]
        \draw (0.3,-0.3) node[anchor=north] {\dotlabel{n}} -- (0.3,0.3);
        \opendot{0.3,0} node[anchor=west] {\dotlabel{\lambda_n}};
    \end{tikzpicture}\; ,
\end{equation}
For any $\mu = (\mu_1,\dotsc,\mu_n) \in \ZZ^n$, let
\begin{equation} \label{eq:notation_X_mu}
X^\mu : = \begin{tikzpicture}[centerzero]
        \draw (0.3,-0.3) node[anchor=north] {\dotlabel{1}} -- (0.3,0.3);
        \greydot{0.3,0} node[anchor=west] {\dotlabel{\mu_1}};
    \end{tikzpicture}
    \ \hdots \
    \begin{tikzpicture}[centerzero]
    	\draw (-1,-0.3) node[anchor=north] {\dotlabel{i-1}} -- (-1,0.3);
    	\greydot{-1,0} node[anchor=west] {\dotlabel{\mu_{i-1}}};
		\draw (0,-0.3) node[anchor=north] {\dotlabel{i}} -- (0,0.3);
		\greydot{0,0} node[anchor=west] {\dotlabel{\mu_i}};
        \draw (1,-0.3) node[anchor=north] {\dotlabel{i+1}} -- (1,0.3);
        \greydot{1,0} node[anchor=west] {\dotlabel{\mu_{i+1}}};
    \end{tikzpicture}
    \ \hdots \
    \begin{tikzpicture}[centerzero]
        \draw (0.3,-0.3) node[anchor=north] {\dotlabel{n}} -- (0.3,0.3);
        \greydot{0.3,0} node[anchor=west] {\dotlabel{\mu_n}};
    \end{tikzpicture}\; .
\end{equation}
Diagrams \labelcref{eq:notation_tilde_a,eq:notation_tilde_T_i,eq:notation_Y_lambda,eq:notation_X_mu} give a full diagrammatic interpretation of the generators from \cref{thm:DAFHA_bernstein_presentation}. Composition corresponds to gluing the inside of one thickened punctured torus to the outside of another, such that the punctures align, and the braids within them are vertically stacked, bottom to top. For example, in $\HH_5(A,z,q)$,
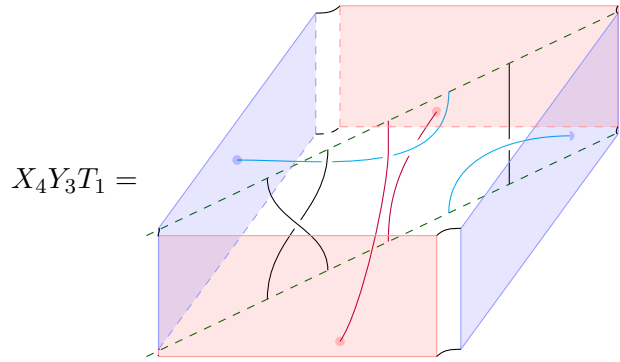
\[
X_4 Y_3 T_1 =
\begin{tikzpicture}[anchorbase,scale=0.58]
        \draw (\torwidth/2,\torheight/2) to[out=up,in=down] (\torwidth/2,\torheight/2+\torthickness);
        \filldraw[blue,opacity=0.2] (\torwidth-\torskew,\torheight/2+2*\torthickness/5) circle (2pt);
		\draw[wipe] (\torwidth/4,\torheight/4) to[out=up,in=left,looseness=1] (\torwidth-\torskew,\torheight/2+2*\torthickness/5);        
        \draw[cyan] (\torwidth/4,\torheight/4) to[out=up,in=left,looseness=1] (\torwidth-\torskew,\torheight/2+2*\torthickness/5);
        \filldraw[red,opacity=0.2] (\torwidth/5,\torheight+\torthickness/8) circle (2pt);
        \draw[purple] (0,0) to[out=up,in=-90-\torangle] (\torwidth/5,\torheight+\torthickness/8);
        \draw[wipe] (-\torwidth+1.5*\torskew,\torheight/2+\torthickness/5) to[out=right,in=down,looseness=1] (\torwidth/4,\torheight/4+\torthickness);
        \draw[cyan] (-\torwidth+1.5*\torskew,\torheight/2+\torthickness/5) to[out=right,in=down,looseness=1] (\torwidth/4,\torheight/4+\torthickness);
        \filldraw[blue,opacity=0.2] (-\torwidth+1.5*\torskew,\torheight/2+\torthickness/5) circle (2pt);
        \draw[wipe] (-\torwidth/5,-\torheight+\torthickness/8) to[out=90-\torangle,in=down, looseness=0.5] (0,3*\torthickness/4) to[out=up,in=down] (0,\torthickness);
        \draw[purple] (-\torwidth/5,-\torheight+\torthickness/8) to[out=90-\torangle,in=down, looseness=0.5] (0,3*\torthickness/4) to[out=up,in=down] (0,\torthickness);
        \filldraw[red,opacity=0.2] (-\torwidth/5,-\torheight+\torthickness/8) circle (2pt);
        \draw[wipe] (-\torwidth/2,-\torheight/2) to[out=up,in=down] (-\torwidth/4,-\torheight/4+\torthickness);
        \draw (-\torwidth/2,-\torheight/2) to[out=up,in=down] (-\torwidth/4,-\torheight/4+\torthickness);
        \draw[wipe] (-\torwidth/4,-\torheight/4) to[out=up,in=down] (-\torwidth/2,-\torheight/2+\torthickness);
        \draw (-\torwidth/4,-\torheight/4) to[out=up,in=down] (-\torwidth/2,-\torheight/2+\torthickness);
        \draw[DarkGreen,dashed] (-\torwidth,-\torheight+\torthickness) -- (\torwidth,\torheight+\torthickness);
        \draw[DarkGreen,dashed] (-\torwidth,-\torheight) -- (\torwidth,\torheight);
		\puncturedtorus
    \end{tikzpicture}
\]
Any strings which route around the meridian or longitude of the torus are coloured for clarity. The generators from \cref{thm:double_affine_Frobenius_Hecke_algebra} can also be written diagrammatically, using the identification \eqref{eq:Y_i_formula}. Indeed, for any given $n \geq 2$, in $\HH_n(A,z,q)$ we have
    \begin{equation} \label{str:pi_on_cylinder}
    \pi = Y_1 T_1^{-1} \dotsb T_{n-1}^{-1} \ =  \
	\begin{tikzpicture}[centerzero]
		\draw (0.8,-1) node[anchor=135 ,yshift=-1] {\dotlabel{n}} to[out=up,in=down] (-0.8,-0.2) to[out=up,in=180] (1.2,0.6);
		\draw (-1.2,0.6) to[out=0,in=down] (-0.8,1);
		\draw[wipe] (-0.8,-1) node[anchor=north] {\dotlabel{1}} to[out=up,in=down] (-0.4,-0.2) to[out=up,in=down] (-0.4,1);
        \draw (-0.8,-1) node[anchor=north] {\dotlabel{1}} to[out=up,in=down] (-0.4,-0.2) to[out=up,in=down] (-0.4,1);
        \draw[wipe] (0.4,-1) node[anchor=north] {\dotlabel{n-1}} to[out=up,in=down] (0.8,-0.2) to[out=up,in=down] (0.8,1);
        \draw (0.4,-1) node[anchor=north] {\dotlabel{n-1}} to[out=up,in=down] (0.8,-0.2) to[out=up,in=down] (0.8,1);
        \identify{-1.2}{-1}{1.2}{1};
        \node at (0.2,0) {$\cdots$};
    \end{tikzpicture}
    \ = \
    \begin{tikzpicture}[centerzero]
        \draw (-0.8,-1) node[anchor=north] {\dotlabel{1}} to[out=up,in=down] (-0.4,1);
        \draw (0.4,-1) node[anchor=north] {\dotlabel{n-1}} to[out=up,in=down] (0.8,1);
		\draw (0.8,-1) node[anchor=135 ,yshift=-1] {\dotlabel{n}} to[out=up,in=225] (1.2,0);
		\draw (-1.2,0) to[out=45,in=down] (-0.8,1);
        \identify{-1.2}{-1}{1.2}{1};
        \node at (0,0) {$\cdots$};
    \end{tikzpicture}\; ,
    \end{equation}
 	and
    \begin{align}
    T_0 &:= Y_n^{-1}Y_1 T_1^{-1} \dotsb T_{n-1}^{-1} \dotsb T_1^{-1}  \label{str:T_0_on_cylinder}\pjmbr
    = \begin{tikzpicture}[centerzero]
    	\draw (1,-1) node[anchor=135 ,yshift=-1] {\dotlabel{n}} to[out=up,in=down,looseness=0.5] (1,-0.6) to[out=up,in=down,looseness=0.5] (-1,-0.2) to[out=up,in=left,looseness=0.5] (1.4,0.2);
    	\draw (-1.4,0.2) to[out=right,in=down] (-1,1);
    	\draw[wipe] (0.6,-1) node[anchor=north] {\dotlabel{n-1}} to[out=up,in=down] (0.2,-0.6) to[out=up,in=down] (0.6,-0.2) to[out=up,in=down] (0.6,1);
    	\draw (0.6,-1) node[anchor=north] {\dotlabel{n-1}} to[out=up,in=down] (0.2,-0.6) to[out=up,in=down] (0.6,-0.2) to[out=up,in=down] (0.6,1);
    	\draw[wipe] (-0.6,-1) node[anchor=north] {\dotlabel{2}} to[out=up,in=down] (-1,-0.6) to[out=up,in=down] (-0.6,-0.2) to[out=up,in=down] (-0.6,1);
    	\draw (-0.6,-1) node[anchor=north] {\dotlabel{2}} to[out=up,in=down] (-1,-0.6) to[out=up,in=down] (-0.6,-0.2) to[out=up,in=down] (-0.6,1);
    	\draw[wipe] (-1,-1) node[anchor=north] {\dotlabel{1}} to[out=up,in=195,looseness=0.5] (0.6,-0.6) to[out=15,in=down,looseness=1] (1,-0.2) to[out=up,in=down,looseness=0.5] (1,0.2) to[out=up,in=right,looseness=0.5] (-1.4,0.6);
    	\draw (-1,-1) node[anchor=north] {\dotlabel{1}} to[out=up,in=195,looseness=0.5] (0.6,-0.6) to[out=15,in=down,looseness=1] (1,-0.2) to[out=up,in=down,looseness=0.5] (1,0.2) to[out=up,in=right,looseness=0.5] (-1.4,0.6);
    	\draw (1.4,0.6) to[out=left,in=down] (1,1);
        \identify{-1.4}{-1}{1.4}{1};
        \node at (0,0) {$\cdots$};
    \end{tikzpicture}\;
    \ = \
    \begin{tikzpicture}[centerzero]
    	\draw (1,-1) node[anchor=135 ,yshift=-1] {\dotlabel{n}} to[out=up,in=down,looseness=0.5] (1,-0.2) to[out=up,in=left] (1.4,0.2);
    	\draw (-1.4,0.2) to[out=right,in=down] (-1,1);
    	\draw (0.6,-1) node[anchor=north] {\dotlabel{n-1}} to[out=up,in=down] (0.6,1);
    	\draw (-0.6,-1) node[anchor=north] {\dotlabel{2}} to[out=up,in=down] (-0.6,1);
		\draw[wipe] (-1,-1) node[anchor=north] {\dotlabel{1}} to[out=up,in=down,looseness=0.5] (-1,0.2) to[out=up,in=right,looseness=1] (-1.4,0.6);
    	\draw (-1,-1) node[anchor=north] {\dotlabel{1}} to[out=up,in=down,looseness=0.5] (-1,0.2) to[out=up,in=right,looseness=1] (-1.4,0.6);
    	\draw (1.4,0.6) to[out=left,in=down] (1,1);
        \identify{-1.4}{-1}{1.4}{1};
        \node at (0,0) {$\cdots$};
    \end{tikzpicture}\; .
    \end{align}
	This diagrammatic description motivates both the identity \eqref{eq:Y_i_formula} and the existence of two presentations \cref{thm:DAFHA_bernstein_presentation} and \cref{thm:double_affine_Frobenius_Hecke_algebra} of $\HH_n(A,z,q)$. In three dimensions, for $\HH_5(A,z,q)$, we have
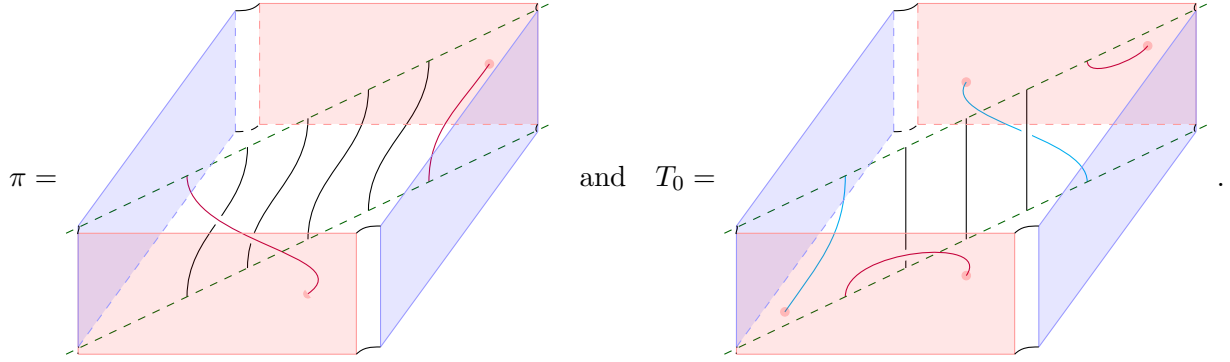
\[
\pi =
\begin{tikzpicture}[anchorbase,scale=0.58]
		\filldraw[red,opacity=0.2] (3*\torwidth/4,\torheight+\torthickness/2) circle (2pt);
        \draw[purple] (\torwidth/2,\torheight/2) to[out=up,in=-90-\torangle] (3*\torwidth/4,\torheight+\torthickness/2);
        \draw (\torwidth/4,\torheight/4) to[out=up,in=down] (\torwidth/2,\torheight/2+\torthickness);
        \draw (0,0) to[out=up,in=down] (\torwidth/4,\torheight/4+\torthickness);
        \draw (-\torwidth/4,-\torheight/4) to[out=up,in=down] (0,\torthickness);
        \draw (-\torwidth/2,-\torheight/2) to[out=up,in=down] (-\torwidth/4,-\torheight/4+\torthickness);
		\filldraw[red,opacity=0.2] (0,-\torheight+\torthickness/2) circle (2pt);
		\draw[wipe] (0,-\torheight+\torthickness/2) to[out=\torangle,in=down] (-\torwidth/2,-\torheight/2+\torthickness);
        \draw[purple] (0,-\torheight+\torthickness/2) to[out=\torangle,in=down] (-\torwidth/2,-\torheight/2+\torthickness);
        \draw[DarkGreen,dashed] (-\torwidth,-\torheight+\torthickness) -- (\torwidth,\torheight+\torthickness);
        \draw[DarkGreen,dashed] (-\torwidth,-\torheight) -- (\torwidth,\torheight);
		\puncturedtorus
    \end{tikzpicture}
    \quad \text{ and } \quad
T_0 =
\begin{tikzpicture}[anchorbase,scale=0.58]
		\filldraw[red,opacity=0.2] (3*\torwidth/4-0.75*\torwidth,\torheight+0.35*\torthickness) circle (2pt);
        \draw[cyan] (\torwidth/2,\torheight/2) to[out=up,in=-90-\torangle,looseness=0.75] (3*\torwidth/4-0.75*\torwidth,\torheight+0.35*\torthickness);
        \draw[wipe] (\torwidth/4,\torheight/4) to[out=up,in=down] (\torwidth/4,\torheight/4+\torthickness);
        \draw (\torwidth/4,\torheight/4) to[out=up,in=down] (\torwidth/4,\torheight/4+\torthickness);
        \draw (0,0) to[out=up,in=down] (0,\torthickness);
        \draw (-\torwidth/4,-\torheight/4) to[out=up,in=down] (-\torwidth/4,-\torheight/4+\torthickness);
        \draw[wipe] (-\torwidth/2,-\torheight/2) to[out=up,in=90-\torangle] (0,-\torheight+0.65*\torthickness);
        \draw[purple] (-\torwidth/2,-\torheight/2) to[out=up,in=90-\torangle] (0,-\torheight+0.65*\torthickness);
        \draw[purple] (3*\torwidth/4,\torheight+0.65*\torthickness) to[out=-90-\torangle,in=down,looseness=0.75] (\torwidth/2,\torheight/2+\torthickness);
        \filldraw[red,opacity=0.2] (3*\torwidth/4,\torheight+0.65*\torthickness) circle (2pt);
        \filldraw[red,opacity=0.2] (0,-\torheight+0.65*\torthickness) circle (2pt);
		\draw[wipe] (-0.75*\torwidth,-\torheight+0.35*\torthickness) to[out=90-\torangle,in=down] (-\torwidth/2,-\torheight/2+\torthickness);
        \draw[cyan] (-0.75*\torwidth,-\torheight+0.35*\torthickness) to[out=90-\torangle,in=down] (-\torwidth/2,-\torheight/2+\torthickness);
        \filldraw[red,opacity=0.2] (-0.75*\torwidth,-\torheight+0.35*\torthickness) circle (2pt);
        \draw[DarkGreen,dashed] (-\torwidth,-\torheight+\torthickness) -- (\torwidth,\torheight+\torthickness);
        \draw[DarkGreen,dashed] (-\torwidth,-\torheight) -- (\torwidth,\torheight);
		\puncturedtorus
    \end{tikzpicture}\; .
\]

All relations of $\HH_n(A,z,q)$ of either presentation can be written diagrammatically. For example, using \eqref{str:pi_on_cylinder}, Relations \eqref{dafha_rel_10} and \eqref{dafha_rel_10b} can be interpreted on the punctured torus as
\begin{equation} \label{str:DAFHA_rel_1}
\begin{tikzpicture}[centerzero]
        \draw (-0.8,-0.5) node[anchor=north] {\dotlabel{1}} to[out=up,in=down] (-0.4,0.5);
        \draw (0.4,-0.5) node[anchor=north] {\dotlabel{n-1}} to[out=up,in=down] (0.8,0.5);
		\draw (0.8,-0.5) node[anchor=135 ,yshift=-1] {\dotlabel{n}} to[out=up,in=225] (1.2,0);
		\draw (-1.2,0) to[out=45,in=down] (-0.8,0.5);
        \greydot{-0.8,-0.35} node[anchor=-45] {\dotlabel{\mu_1}};
        \greydot{0.4,-0.35} node[anchor=-30] {\dotlabel{\mu_{n-1}}};
        \greydot{0.85,-0.35} node[anchor=-100,yshift=2] {\dotlabel{\mu_n}};
        \identify{-1.2}{-0.5}{1.2}{0.5};
        \node at (0,0.2) {$\cdots$};
    \end{tikzpicture}
    \ = \
    q^{-\mu_n} \; \begin{tikzpicture}[centerzero]
        \draw (-0.8,-0.5) node[anchor=north] {\dotlabel{1}} to[out=up,in=down] (-0.4,0.5);
        \draw (0.4,-0.5) node[anchor=north] {\dotlabel{n-1}} to[out=up,in=down] (0.8,0.5);
		\draw (0.8,-0.5) node[anchor=135 ,yshift=0] {\dotlabel{n}} to[out=up,in=225] (1.2,0);
		\draw (-1.2,0) to[out=45,in=down] (-0.8,0.5);
        \greydot{-0.85,0.35} node[anchor=-45,yshift=2] {\dotlabel{\mu_n}};
        \greydot{-0.4,0.35} node[anchor=south,yshift=2] {\dotlabel{\mu_1}};
        \greydot{0.8,0.35} node[anchor=south,yshift=2] {\dotlabel{\mu_{n-1}}};
        \identify{-1.2}{-0.5}{1.2}{0.5};
        \node at (0,0) {$\cdots$};
    \end{tikzpicture}\; ,
\end{equation}
for $\mu = (\mu_1,\dotsc,\mu_n) \in \ZZ^n$. Using \eqref{str:T_0_on_cylinder}, Relations \eqref{dafha_rel_8} and \eqref{dafha_rel_8b} can be written diagrammatically as
\begin{equation} \label{str:DAFHA_rel_2}
\begin{tikzpicture}[centerzero]
		\draw (-0.6,-0.5) node[anchor=north] {\dotlabel{2}} -- (-0.6,0.5);
		\draw (0.6,-0.5) node[anchor=north] {\dotlabel{n-1}} -- (0.6,0.5);
        \draw (1,-0.5) node[anchor=north] {\dotlabel{n}} to[out=up,in=225] (1.4,-0.2);
        \draw (-1.4,-0.2) to[out=45,in=down] (-1,0.5);
        \draw[wipe] (-1,-0.5) to[out=up,in=-45] (-1.4,0.2);
        \draw (-1,-0.5) node[anchor=north] {\dotlabel{1}} to[out=up,in=-45] (-1.4,0.2);
        \draw (1.4,0.2) to[out=135,in=down] (1,0.5);
        \greydot{1.06,-0.35} node[anchor=south] {\dotlabel{1}};
        \identify{-1.4}{-0.5}{1.4}{0.5};
        \node at (0,0) {$\cdots$};
    \end{tikzpicture}
    \ = \
    q^{-1} \;
    \begin{tikzpicture}[centerzero]
		\draw (-0.6,-0.5) node[anchor=north] {\dotlabel{2}} -- (-0.6,0.5);
		\draw (0.6,-0.5) node[anchor=north] {\dotlabel{n-1}} -- (0.6,0.5);
		\draw (-1,-0.5) node[anchor=north] {\dotlabel{1}} to[out=up,in=-45] (-1.4,0.2);
		\draw (1.4,0.2) to[out=135,in=down] (1,0.5);
        \draw (1,-0.5) node[anchor=north] {\dotlabel{n}} to[out=up,in=225] (1.4,-0.2);
        \draw[wipe] (-1.4,-0.2) to[out=45,in=down] (-1,0.5);
        \draw (-1.4,-0.2) to[out=45,in=down] (-1,0.5);
        \greydot{-1,0.35} node[anchor=west] {\dotlabel{1}};
        \identify{-1.4}{-0.5}{1.4}{0.5};
        \node at (0,0) {$\cdots$};
    \end{tikzpicture} \; .
\end{equation}
In particular, taking $(\mu_1,\dotsc,\mu_n)=(0,\dotsc,0,1)$, Relations \eqref{str:DAFHA_rel_1} and \eqref{str:DAFHA_rel_2} imply
\begin{gather*}
\begin{tikzpicture}[centerzero]
        \draw (-0.8,-1) node[anchor=north] {\dotlabel{1}} to[out=up,in=down] (-0.4,0);
        \draw (0.4,-1) node[anchor=north] {\dotlabel{n-1}} to[out=up,in=down] (0.8,0);
		\draw (0.8,-1) node[anchor=135 ,yshift=-1] {\dotlabel{n}} to[out=up,in=225] (1.2,-0.5);
		\draw (-1.2,-0.5) to[out=45,in=down] (-0.8,0);
        \greydot{0.85,-0.85} node[anchor=-100,yshift=2] {\dotlabel{1}};
        \draw (-0.8,0)to[out=up,in=-45] (-1.2,0.5);
        \draw (-0.4,0) to[out=up,in=down] (-0.8,1);
		\draw (0.8,0) to[out=up,in=down] (0.4,1);
		\draw (1.2,0.5) to[out=135,in=down] (0.8,1);
        \greydot{-0.8,0} node[anchor=-100,yshift=3] {\dotlabel{-1}};
        \identify{-1.2}{-1}{1.2}{1};
        \node at (0.2,0) {$\cdots$};
    \end{tikzpicture}
	\ = \
    q^{-1} \;
    \begin{tikzpicture}[centerzero]
    	\draw (-1,-0.5) node[anchor=north] {\dotlabel{1}} -- (-1,0.5);
		\draw (-0.6,-0.5) node[anchor=north] {\dotlabel{2}} -- (-0.6,0.5);
		\draw (0.6,-0.5) node[anchor=north] {\dotlabel{n-1}} -- (0.6,0.5);
		\draw (1,-0.5) node[anchor=north] {\dotlabel{n}} -- (1,0.5);
        \identify{-1.4}{-0.5}{1.4}{0.5};
        \node at (0,0) {$\cdots$};
    \end{tikzpicture} \pjmand
	\begin{tikzpicture}[centerzero]
		\draw (-0.6,-1) node[anchor=north] {\dotlabel{2}} -- (-0.6,0);
		\draw (0.6,-1) node[anchor=north] {\dotlabel{n-1}} -- (0.6,0);
        \draw (1,-1) node[anchor=north] {\dotlabel{n}} to[out=up,in=225] (1.4,-0.7);
        \draw (-1.4,-0.7) to[out=45,in=down] (-1,0);
        \draw[wipe] (-1,-1) to[out=up,in=-45] (-1.4,-0.3);
        \draw (-1,-1) node[anchor=north] {\dotlabel{1}} to[out=up,in=-45] (-1.4,-0.3);
        \draw (1.4,-0.3) to[out=135,in=down] (1,0);
        \greydot{1.06,-0.85} node[anchor=south] {\dotlabel{1}};
        \draw (-0.6,0) -- (-0.6,1);
		\draw (0.6,0) -- (0.6,1);
		\draw (1.4,0.7) to[out=135,in=down] (1,1);
        \draw (1,0) to[out=up,in=225] (1.4,0.3);
        \draw (-1.4,0.3) to[out=45,in=down] (-1,1);
        \draw[wipe] (-1,0) to[out=up,in=-45] (-1.4,0.7);
        \draw (-1,0) to[out=up,in=-45] (-1.4,0.7);
        \greydot{-1,0} node[anchor=-120,yshift=3] {\dotlabel{-1}};
        \identify{-1.4}{-1}{1.4}{1};
        \node at (0,0) {$\cdots$};
    \end{tikzpicture}
    \ = \
    q^{-1} \;
    \begin{tikzpicture}[centerzero]
    	\draw (-1,-0.5) node[anchor=north] {\dotlabel{1}} -- (-1,0.5);
		\draw (-0.6,-0.5) node[anchor=north] {\dotlabel{2}} -- (-0.6,0.5);
		\draw (0.6,-0.5) node[anchor=north] {\dotlabel{n-1}} -- (0.6,0.5);
		\draw (1,-0.5) node[anchor=north] {\dotlabel{n}} -- (1,0.5);
        \identify{-1.4}{-0.5}{1.4}{0.5};
        \node at (0,0) {$\cdots$};
    \end{tikzpicture} \;,
\end{gather*}
respectively. The left-hand-side of either diagram can be interpreted as a single loop of the $n$-th string around the puncture of the torus. Hence, a string can be moved through the puncture, at the cost of introducing a factor of $q$ or $q^{-1}$, and still represent the same element of $\HH_n(A,z,q)$.

\section{Symmetries for \texorpdfstring{$\HH_n(A,z,q)$}{ℍₙ(A,z,q)}} \label{sec:symmetries_for_DAFHA}

The diagrammatic interpretation of $\HH_n(A,z,q)$ from \cref{sec:diagrammatic_interpretation} provides geometrically-motivated involutions and anti-involutions of $\HH_n(A,z,q)$, up to a change in parameters.

\begin{figure}
    \centering
    \begin{subfigure}[b]{0.32\linewidth}\centering
	$\begin{tikzpicture}[anchorbase,scale=\torscale]
    	\draw (-\torwidth/4,-\torheight/4) to[out=up,in=down] (-\torwidth/4,-\torheight/4+\torthickness);
        \draw (-\torwidth/2,-\torheight/2) to[out=up,in=down] (-\torwidth/2,-\torheight/2+\torthickness);
        \draw (0,0) to[out=up,in=down] (0,\torthickness);
        \draw (\torwidth/2,\torheight/2) to[out=up,in=down] (\torwidth/2,\torheight/2+\torthickness);
        \draw (\torwidth/4,\torheight/4) to[out=up,in=down] (\torwidth/4,\torheight/4+\torthickness);
		\puncturedtorus
		\path[fill=DarkGreen,opacity=0.1] (-\torwidth-\torpunc,-\torheight-\torpunc) -- (\torwidth+\torpunc,\torheight-\torpunc) -- (\torwidth+\torpunc,\torheight+\torthickness+\torpunc) -- (-\torwidth-\torpunc,-\torheight+\torthickness+\torpunc)--cycle;
		\draw[DarkGreen] (\torwidth,\torheight+\torthickness/2) node[xshift=0.75em,rotate=-45] {$\leftrightarrow$};
    \end{tikzpicture}$
	\caption{$\epsilontransform$}\label{fig:eps}
	\end{subfigure}%
	\begin{subfigure}[b]{0.32\linewidth}\centering
	$\begin{tikzpicture}[anchorbase,scale=\torscale]
    	\draw (-\torwidth/4,-\torheight/4) to[out=up,in=down] (-\torwidth/4,-\torheight/4+\torthickness);
        \draw (-\torwidth/2,-\torheight/2) to[out=up,in=down] (-\torwidth/2,-\torheight/2+\torthickness);
        \draw (0,0) to[out=up,in=down] (0,\torthickness);
        \draw (\torwidth/2,\torheight/2) to[out=up,in=down] (\torwidth/2,\torheight/2+\torthickness);
        \draw (\torwidth/4,\torheight/4) to[out=up,in=down] (\torwidth/4,\torheight/4+\torthickness);
		\puncturedtorus
		\path[fill=DarkGreen,opacity=0.1] (-\torwidth-\torpunc,-\torheight+\torthickness/2-\torpunc) -- (\torskew+\torpunc,-\torheight+\torthickness/2-\torpunc) -- (\torwidth+\torpunc,\torheight+\torthickness/2+\torpunc) -- (-\torskew-\torpunc,\torheight+\torthickness/2+\torpunc)--cycle;
		\draw[DarkGreen] (\torwidth,\torheight+\torthickness/2) node[xshift=0.65em,rotate=90] {$\leftrightarrow$};
    \end{tikzpicture}$
	\caption{$\Rdaff$}\label{fig:R}
	\end{subfigure}%
	\begin{subfigure}[b]{0.32\linewidth}\centering
	$\begin{tikzpicture}[anchorbase,scale=\torscale]
    	\draw (-\torwidth/4,-\torheight/4) to[out=up,in=down] (-\torwidth/4,-\torheight/4+\torthickness);
        \draw (-\torwidth/2,-\torheight/2) to[out=up,in=down] (-\torwidth/2,-\torheight/2+\torthickness);
        \draw (0,0) to[out=up,in=down] (0,\torthickness);
        \draw (\torwidth/2,\torheight/2) to[out=up,in=down] (\torwidth/2,\torheight/2+\torthickness);
        \draw (\torwidth/4,\torheight/4) to[out=up,in=down] (\torwidth/4,\torheight/4+\torthickness);
		\puncturedtorus
		\path[fill=DarkGreen,opacity=0.1] (-\torskew-\torpunc,\torheight-\torpunc) -- (-\torskew-\torpunc,\torheight+\torthickness+\torpunc) -- (\torskew+\torpunc,-\torheight+\torthickness+\torpunc) -- (\torskew+\torpunc,-\torheight-\torpunc)--cycle;
		\draw[DarkGreen] (\torskew+-\torpunc,-\torheight-\torpunc) node[xshift=0.75em,rotate=45] {$\leftrightarrow$};
    \end{tikzpicture}$
	\caption{$D$}\label{fig:D}
	\end{subfigure}

	\bigskip
	
    \begin{subfigure}[b]{0.24\linewidth}
        \centering
    	$\begin{tikzpicture}[anchorbase,scale=\torscale]
    	\draw (-\torwidth/4,-\torheight/4) to[out=up,in=down] (-\torwidth/4,-\torheight/4+\torthickness);
        \draw (-\torwidth/2,-\torheight/2) to[out=up,in=down] (-\torwidth/2,-\torheight/2+\torthickness);
        \draw (0,0) to[out=up,in=down] (0,\torthickness);
        \draw (\torwidth/2,\torheight/2) to[out=up,in=down] (\torwidth/2,\torheight/2+\torthickness);
        \draw (\torwidth/4,\torheight/4) to[out=up,in=down] (\torwidth/4,\torheight/4+\torthickness);
		\puncturedtorus
		\draw[DarkGreen,dashed] (-\torwidth,-\torheight+\torthickness/2) -- (\torwidth,\torheight+\torthickness/2) node[anchor=west] {$\mathbf{\circlearrowleft}$};
    \end{tikzpicture}$
    	\caption{$\Rdaff \circ \epsilontransform$}
        \label{fig:R_eps}
    \end{subfigure}%
    \begin{subfigure}[b]{0.24\linewidth}
        \centering
    	$\begin{tikzpicture}[anchorbase,scale=\torscale]
    	\draw (-\torwidth/4,-\torheight/4) to[out=up,in=down] (-\torwidth/4,-\torheight/4+\torthickness);
        \draw (-\torwidth/2,-\torheight/2) to[out=up,in=down] (-\torwidth/2,-\torheight/2+\torthickness);
        \draw (0,0) to[out=up,in=down] (0,\torthickness);
        \draw (\torwidth/2,\torheight/2) to[out=up,in=down] (\torwidth/2,\torheight/2+\torthickness);
        \draw (\torwidth/4,\torheight/4) to[out=up,in=down] (\torwidth/4,\torheight/4+\torthickness);
		\puncturedtorus
		\draw[DarkGreen,dashed] (0,-\torheight-\torpunc) -- (0,\torheight+\torthickness+\torpunc) node[anchor=south] {$\mathbf{\circlearrowleft}$};
    	\end{tikzpicture}$
    	\caption{$\epsilontransform \circ D$}
        \label{fig:eps_D}
    \end{subfigure}%
    \begin{subfigure}[b]{0.24\linewidth}
        \centering
        $\begin{tikzpicture}[anchorbase,scale=\torscale]
    	\draw (-\torwidth/4,-\torheight/4) to[out=up,in=down] (-\torwidth/4,-\torheight/4+\torthickness);
        \draw (-\torwidth/2,-\torheight/2) to[out=up,in=down] (-\torwidth/2,-\torheight/2+\torthickness);
        \draw (0,0) to[out=up,in=down] (0,\torthickness);
        \draw (\torwidth/2,\torheight/2) to[out=up,in=down] (\torwidth/2,\torheight/2+\torthickness);
        \draw (\torwidth/4,\torheight/4) to[out=up,in=down] (\torwidth/4,\torheight/4+\torthickness);
		\puncturedtorus
		\draw[DarkGreen,dashed] (-\torskew,\torheight+\torthickness/2) -- (\torskew,-\torheight+\torthickness/2) node[anchor=west,xshift=10] {$\mathbf{\circlearrowleft}$};
    \end{tikzpicture}$
        \caption{$\Rdaff \circ D$}
        \label{fig:R_D}
    \end{subfigure}%
    \begin{subfigure}[b]{0.24\linewidth}
        \centering
        $\begin{tikzpicture}[anchorbase,scale=\torscale]
    	\draw (-\torwidth/4,-\torheight/4) to[out=up,in=down] (-\torwidth/4,-\torheight/4+\torthickness);
        \draw (-\torwidth/2,-\torheight/2) to[out=up,in=down] (-\torwidth/2,-\torheight/2+\torthickness);
        \draw (0,0) to[out=up,in=down] (0,\torthickness);
        \draw (\torwidth/2,\torheight/2) to[out=up,in=down] (\torwidth/2,\torheight/2+\torthickness);
        \draw (\torwidth/4,\torheight/4) to[out=up,in=down] (\torwidth/4,\torheight/4+\torthickness);
		\puncturedtorus
		\fill[DarkGreen] (0,\torheight/2) circle (5pt);
    \end{tikzpicture}$
        \caption{$\Rdaff \circ \epsilontransform \circ D$}
        \label{fig:R_eps_D}
    \end{subfigure}
    \caption{Punctured torus transformations generated by reflections $\epsilontransform$, $\Rdaff$ and $D$.}
    \label{fig:four_transformations}
\end{figure}

\begin{proposition} \label{thm:DAFHA_geometric_isomorphisms}
Let $n \in \NN, n \geq 2$. The following maps are isomorphisms of $\kk$-superalgebras:
\begin{enumerate}[i)]
\item Reflecting diagrams in the green plane shown in \cref{fig:eps} gives the isomorphism
	\begin{gather*}
	\epsilontransform \colon \HH_n(A,z,q) \to \HH_n(A,-z,q^{-1}),\\
	\mathbf{a} \mapsto \mathbf{a}, \qquad T_i \mapsto T_i^{-1}, \qquad X_j \mapsto Y_j, \qquad Y_j \mapsto X_j,
	\end{gather*}
	for $\mathbf{a} \in A^{\otimes n}$, $1 \leq i \leq n-1$ and $1 \leq j \leq n$, which for $A=\kk$ recovers the Fourier--Cherednik transform $\epsilon$ from \cite[(9.8)]{DAHA_2004}. It is the isomorphism \eqref{eq:XY_automorphism} from \cref{thm:XY_automorphisms}.

\item Reflecting diagrams in the green plane shown in \cref{fig:R} gives the isomorphism
	\begin{gather*}
	\Rdaff \colon \HH_n(A,z,q) \to \HH_n(A^\textup{op},-z,q^{-1})^\textup{op},\\
	\mathbf{a} \mapsto \mathbf{a}^\textup{op}, \qquad T_i \mapsto T_i^{-1}, \qquad X_j \mapsto X_j^{-1}, \qquad Y_j \mapsto Y_j^{-1},
	\end{gather*}
	for $\mathbf{a} \in A^{\otimes n}$, $1 \leq i \leq n-1$ and $1 \leq j \leq n$, which for $A=\kk$ recovers the anti-involution $\star$ from \cite[(2.16)]{DAHADFT_2003}.

	\item Reflecting diagrams in the green plane shown in \cref{fig:D} gives the isomorphism
	\begin{gather*}
	D \colon \HH_n(A,z,q) \to \HH_n(A,-z,q^{-1}),\\
	\mathbf{a} \mapsto w_0(\mathbf{a}), \qquad T_i \mapsto T_{n-i}^{-1}, \qquad X_j \mapsto Y_{n+1-j}^{-1}, \qquad Y_j \mapsto X_{n+1-j}^{-1},
	\end{gather*}
	for $\mathbf{a} \in A^{\otimes n}$, $1 \leq i \leq n-1$ and $1 \leq j \leq n$, where $w_0$ is the longest element of $S_n$, acting by superpermutation on $A^{\otimes n}$, by $w_0(a_1 \otimes \dotsb \otimes a_n) = (-1)^{\sum_{1 \leq k < l \leq n}\bar{a}_k\bar{a}_l}\, a_n \otimes \dotsb \otimes a_1$ for homogeneous $a_1,\dotsc,a_n \in A$.

	\item The isomorphism $\Rdaff \circ \epsilontransform \colon \HH_n(A,z,q) \to \HH_n(A^\textup{op},z,q)^\textup{op}$ corresponds to rotating diagrams $180^\circ$ about the dashed green line shown in \cref{fig:R_eps}. It is the isomorphism $\sigmaop$ from \cref{thm:XY_automorphisms}.

\item The isomorphism $\epsilontransform \circ D \colon \HH_n(A,z,q) \to \HH_n(A,z,q)$ corresponds to rotating diagrams $180^\circ$ about the dashed green line shown in \cref{fig:eps_D}.

\item The isomorphism $\Rdaff \circ D \colon \HH_n(A,z,q) \to \HH_n(A^\textup{op},z,q)^\textup{op}$ corresponds to rotating diagrams $180^\circ$ about the dashed green line shown in \cref{fig:R_D}.

\item The isomorphism $\Rdaff \circ \epsilontransform \circ D \colon \HH_n(A,z,q) \to \HH_n(A^\textup{op},-z,q^{-1})^\textup{op}$ corresponds to the spatial inversion/point reflection centred at the green point shown in \cref{fig:R_eps_D}.
\end{enumerate}
\end{proposition}
\begin{proof}
The map $\epsilontransform$ is verified to be an isomorphism in \cref{thm:XY_automorphisms}. For $\Rdaff$ and $D$ one can check the relations of \cref{thm:DAFHA_bernstein_presentation} directly, and the remaining maps follow by composing.
\end{proof}

\begin{proposition} \label{thm:minus_z_or_minus_T_DAHA}
Let $n \in \NN, n \geq 2$ and let $u \in Z(A)$ be invertible and even. For $1 \leq j \leq n$, let $u_j = 1^{\otimes(j-1)} \otimes u \otimes 1^{\otimes(n-j)} \in A^{\otimes n}$. There is an isomorphism of $\kk$-superalgebras, given by
\begin{equation} \label{eq:iota}
\begin{gathered}
\iota \colon \HH_n(A,z,q) \to \HH_n(A,-z,q),\\
\mathbf{a} \mapsto \mathbf{a}, \qquad T_i \mapsto -T_i, \qquad X_j \mapsto X_j, \qquad Y_j \mapsto Y_j,
\end{gathered}
\end{equation}
for $\mathbf{a} \in A^{\otimes n}$, $1 \leq i \leq n-1$ and $1 \leq j \leq n$.
\details{Diagrammatically,
\begin{gather*}
\iota \colon \HH_n(A,z,q) \to \HH_n(A,-z,q),\\
	\begin{tikzpicture}[centerzero]
        \draw (-0.3,-0.3) -- (0.3,0.3);
        \draw[wipe] (0.3,-0.3) -- (-0.3,0.3);
        \draw (0.3,-0.3) -- (-0.3,0.3);
    \end{tikzpicture} \mapsto - \begin{tikzpicture}[centerzero]
        \draw (-0.3,-0.3) -- (0.3,0.3);
        \draw[wipe] (0.3,-0.3) -- (-0.3,0.3);
        \draw (0.3,-0.3) -- (-0.3,0.3);
    \end{tikzpicture}\;, \qquad \begin{tikzpicture}[centerzero]
        \draw (0.3,-0.3) -- (-0.3,0.3);
        \draw[wipe] (-0.3,-0.3) -- (0.3,0.3);
        \draw (-0.3,-0.3) -- (0.3,0.3);
    \end{tikzpicture} \mapsto - \begin{tikzpicture}[centerzero]
        \draw (0.3,-0.3) -- (-0.3,0.3);
        \draw[wipe] (-0.3,-0.3) -- (0.3,0.3);
        \draw (-0.3,-0.3) -- (0.3,0.3);
    \end{tikzpicture}\;, \qquad \begin{tikzpicture}[centerzero]
        \draw (0,-0.3) -- (0,0.3);
        \token{0,0}{west}{a};
    \end{tikzpicture} \mapsto \; \begin{tikzpicture}[centerzero]
        \draw (0,-0.3) -- (0,0.3);
        \token{0,0}{west}{a};
    \end{tikzpicture}\; , \\
    \begin{tikzpicture}[centerzero]
        \draw (0,-0.3) -- (0,0.3);
        \opendot{0,0} node[anchor=west] {\dotlabel{1}};
    \end{tikzpicture} \mapsto \; \begin{tikzpicture}[centerzero]
        \draw (0,-0.3) -- (0,0.3);
        \opendot{0,0} node[anchor=west] {\dotlabel{1}};
    \end{tikzpicture}\; , \qquad \begin{tikzpicture}[centerzero]
        \draw (0,-0.3) -- (0,0.3);
        \opendot{0,0} node[anchor=west] {\dotlabel{-1}};
    \end{tikzpicture} \mapsto \; \begin{tikzpicture}[centerzero]
        \draw (0,-0.3) -- (0,0.3);
        \opendot{0,0} node[anchor=west] {\dotlabel{-1}};
    \end{tikzpicture}\; , \qquad \begin{tikzpicture}[centerzero]
        \draw (0,-0.3) -- (0,0.3);
        \greydot{0,0} node[anchor=west] {\dotlabel{1}};
    \end{tikzpicture} \mapsto \; \begin{tikzpicture}[centerzero]
        \draw (0,-0.3) -- (0,0.3);
        \greydot{0,0} node[anchor=west] {\dotlabel{1}};
    \end{tikzpicture}\; , \qquad \begin{tikzpicture}[centerzero]
        \draw (0,-0.3) -- (0,0.3);
        \greydot{0,0} node[anchor=west] {\dotlabel{-1}};
    \end{tikzpicture} \mapsto \; \begin{tikzpicture}[centerzero]
        \draw (0,-0.3) -- (0,0.3);
        \greydot{0,0} node[anchor=west] {\dotlabel{-1}};
    \end{tikzpicture}\; ,
\end{gather*}
for $a \in A$.} For $A=\kk$, we recover the outer involution $\iota$ of $\HH_n(\kk,\tau-\tau^{-1},q)$ mentioned in \cite[p.~48]{DAHA_2004}. There is an automorphism of $\kk$-superalgebras, given by
\begin{equation} \label{eq:zeta_aX}
\begin{gathered}
\varsigma_{uX} \colon \HH_n(A,z,q) \to \HH_n(A,z,q),\\
\mathbf{a} \mapsto \mathbf{a}, \qquad T_i \mapsto T_i, \qquad X_j \mapsto u_jX_j, \qquad Y_j \mapsto Y_j, \end{gathered}
\end{equation}
for $\mathbf{a} \in A^{\otimes n}$, $1 \leq i \leq n-1$ and $1 \leq j \leq n$.
\details{Diagrammatically, for example,
\begin{gather*}
\varsigma_{(-1)X} \colon \HH_n(A,z,q) \to \HH_n(A,z,q),\\
	\begin{tikzpicture}[centerzero]
        \draw (-0.3,-0.3) -- (0.3,0.3);
        \draw[wipe] (0.3,-0.3) -- (-0.3,0.3);
        \draw (0.3,-0.3) -- (-0.3,0.3);
    \end{tikzpicture} \mapsto \begin{tikzpicture}[centerzero]
        \draw (-0.3,-0.3) -- (0.3,0.3);
        \draw[wipe] (0.3,-0.3) -- (-0.3,0.3);
        \draw (0.3,-0.3) -- (-0.3,0.3);
    \end{tikzpicture}\;, \qquad \begin{tikzpicture}[centerzero]
        \draw (0.3,-0.3) -- (-0.3,0.3);
        \draw[wipe] (-0.3,-0.3) -- (0.3,0.3);
        \draw (-0.3,-0.3) -- (0.3,0.3);
    \end{tikzpicture} \mapsto \begin{tikzpicture}[centerzero]
        \draw (0.3,-0.3) -- (-0.3,0.3);
        \draw[wipe] (-0.3,-0.3) -- (0.3,0.3);
        \draw (-0.3,-0.3) -- (0.3,0.3);
    \end{tikzpicture}\;, \qquad \begin{tikzpicture}[centerzero]
        \draw (0,-0.3) -- (0,0.3);
        \token{0,0}{west}{a};
    \end{tikzpicture} \mapsto \; \begin{tikzpicture}[centerzero]
        \draw (0,-0.3) -- (0,0.3);
        \token{0,0}{west}{a};
    \end{tikzpicture}\; , \\
    \begin{tikzpicture}[centerzero]
        \draw (0,-0.3) -- (0,0.3);
        \opendot{0,0} node[anchor=west] {\dotlabel{1}};
    \end{tikzpicture} \mapsto \; \begin{tikzpicture}[centerzero]
        \draw (0,-0.3) -- (0,0.3);
        \opendot{0,0} node[anchor=west] {\dotlabel{1}};
    \end{tikzpicture}\; , \qquad \begin{tikzpicture}[centerzero]
        \draw (0,-0.3) -- (0,0.3);
        \opendot{0,0} node[anchor=west] {\dotlabel{-1}};
    \end{tikzpicture} \mapsto \; \begin{tikzpicture}[centerzero]
        \draw (0,-0.3) -- (0,0.3);
        \opendot{0,0} node[anchor=west] {\dotlabel{-1}};
    \end{tikzpicture}\; , \qquad \begin{tikzpicture}[centerzero]
        \draw (0,-0.3) -- (0,0.3);
        \greydot{0,0} node[anchor=west] {\dotlabel{1}};
    \end{tikzpicture} \mapsto \; u \; \begin{tikzpicture}[centerzero]
        \draw (0,-0.3) -- (0,0.3);
        \greydot{0,0} node[anchor=west] {\dotlabel{1}};
    \end{tikzpicture}\; , \qquad \begin{tikzpicture}[centerzero]
        \draw (0,-0.3) -- (0,0.3);
        \greydot{0,0} node[anchor=west] {\dotlabel{-1}};
    \end{tikzpicture} \mapsto \; u \; \begin{tikzpicture}[centerzero]
        \draw (0,-0.3) -- (0,0.3);
        \greydot{0,0} node[anchor=west] {\dotlabel{-1}};
    \end{tikzpicture}\; ,
\end{gather*}
for $a \in A$.} For $A=\kk$ and $u=-1$, we recover the outer involution $\varsigma_x$ of $\HH_n(\kk,\tau-\tau^{-1},q)$ mentioned in \cite[p.~48]{DAHA_2004}. There is an automorphism of $\kk$-superalgebras, given by
\begin{equation} \label{eq:zeta_aY}
\begin{gathered}
\varsigma_{uY} \colon \HH_n(A,z,q) \to \HH_n(A,z,q),\\
\mathbf{a} \mapsto \mathbf{a}, \qquad T_i \mapsto T_i, \qquad X_j \mapsto X_j, \qquad Y_j \mapsto u_jY_j, \end{gathered}
\end{equation}
for $\mathbf{a} \in A^{\otimes n}$, $1 \leq i \leq n-1$ and $1 \leq j \leq n$.
\details{Diagrammatically, for example,
\begin{gather*}
\varsigma_{(-1)Y} \colon \HH_n(A,z,q) \to \HH_n(A,z,q),\\
	\begin{tikzpicture}[centerzero]
        \draw (-0.3,-0.3) -- (0.3,0.3);
        \draw[wipe] (0.3,-0.3) -- (-0.3,0.3);
        \draw (0.3,-0.3) -- (-0.3,0.3);
    \end{tikzpicture} \mapsto \begin{tikzpicture}[centerzero]
        \draw (-0.3,-0.3) -- (0.3,0.3);
        \draw[wipe] (0.3,-0.3) -- (-0.3,0.3);
        \draw (0.3,-0.3) -- (-0.3,0.3);
    \end{tikzpicture}\;, \qquad \begin{tikzpicture}[centerzero]
        \draw (0.3,-0.3) -- (-0.3,0.3);
        \draw[wipe] (-0.3,-0.3) -- (0.3,0.3);
        \draw (-0.3,-0.3) -- (0.3,0.3);
    \end{tikzpicture} \mapsto \begin{tikzpicture}[centerzero]
        \draw (0.3,-0.3) -- (-0.3,0.3);
        \draw[wipe] (-0.3,-0.3) -- (0.3,0.3);
        \draw (-0.3,-0.3) -- (0.3,0.3);
    \end{tikzpicture}\;, \qquad \begin{tikzpicture}[centerzero]
        \draw (0,-0.3) -- (0,0.3);
        \token{0,0}{west}{a};
    \end{tikzpicture} \mapsto \; \begin{tikzpicture}[centerzero]
        \draw (0,-0.3) -- (0,0.3);
        \token{0,0}{west}{a};
    \end{tikzpicture}\; , \\
    \begin{tikzpicture}[centerzero]
        \draw (0,-0.3) -- (0,0.3);
        \opendot{0,0} node[anchor=west] {\dotlabel{1}};
    \end{tikzpicture} \mapsto \; u \; \begin{tikzpicture}[centerzero]
        \draw (0,-0.3) -- (0,0.3);
        \opendot{0,0} node[anchor=west] {\dotlabel{1}};
    \end{tikzpicture}\; , \qquad \begin{tikzpicture}[centerzero]
        \draw (0,-0.3) -- (0,0.3);
        \opendot{0,0} node[anchor=west] {\dotlabel{-1}};
    \end{tikzpicture} \mapsto \; u \; \begin{tikzpicture}[centerzero]
        \draw (0,-0.3) -- (0,0.3);
        \opendot{0,0} node[anchor=west] {\dotlabel{-1}};
    \end{tikzpicture}\; , \qquad \begin{tikzpicture}[centerzero]
        \draw (0,-0.3) -- (0,0.3);
        \greydot{0,0} node[anchor=west] {\dotlabel{1}};
    \end{tikzpicture} \mapsto \; \begin{tikzpicture}[centerzero]
        \draw (0,-0.3) -- (0,0.3);
        \greydot{0,0} node[anchor=west] {\dotlabel{1}};
    \end{tikzpicture}\; , \qquad \begin{tikzpicture}[centerzero]
        \draw (0,-0.3) -- (0,0.3);
        \greydot{0,0} node[anchor=west] {\dotlabel{-1}};
    \end{tikzpicture} \mapsto \; \begin{tikzpicture}[centerzero]
        \draw (0,-0.3) -- (0,0.3);
        \greydot{0,0} node[anchor=west] {\dotlabel{-1}};
    \end{tikzpicture}\; ,
\end{gather*}
for $a \in A$.} For $A=\kk$ and $u=-1$, we recover the outer involution $\varsigma_y$ of $\HH_n(\kk,\tau-\tau^{-1},q)$ mentioned in \cite[p.~48]{DAHA_2004}.
\end{proposition}
\begin{proof}
It is straightforward to verify that $\iota$, $\varsigma_{uX}$ and $\varsigma_{uY}$ verify the defining relations of $\HH_n(A,z,q)$ from \cref{thm:DAFHA_bernstein_presentation}. Clearly, $\iota$ is an involution up to replacement of $z$ by $-z$, the inverse of $\varsigma_{uX}$ is $\varsigma_{u^{-1}X}$ and the inverse of $\varsigma_{uY}$ is $\varsigma_{u^{-1}Y}$.
\end{proof}

\begin{lemma} \label{thm:A_automorphisms}
Let $n \in \NN, n \geq 2$, and let $\xi \colon A \to D$ be an even, trace-preserving isomorphism of $\kk$-superalgebras. Then there is an isomorphism of $\kk$-superalgebras, given by
\begin{equation} \label{eq:A_automorphisms}
\begin{gathered}
\bar{\xi} \colon \HH_n(A,z,q) \to \HH_n(D,z,q),\\
\mathbf{a} \mapsto \xi^{\otimes n}(\mathbf{a}), \qquad T_i \mapsto T_i, \qquad X_j \mapsto X_j, \qquad Y_j \mapsto Y_j,
\end{gathered}
\end{equation}
for $\mathbf{a} \in A^{\otimes n}$, $1 \leq i \leq n-1$ and $1 \leq j \leq n$. In particular, this defines an action of the group $\textup{Aut}_{\tr}(A) := \{ \xi \in \textup{Aut}(A) \mid \xi \textup{ is even}, \tr \circ \xi = \tr \}$ on $\HH_n(A,z,q)$ by automorphisms.
\end{lemma}
\begin{proof}
Since $\xi$ preserves the trace, $\xi^{\otimes n}$ maps the teleporter $t_{i,j}$ of $A^{\otimes n}$ to the corresponding teleporter $t_{i,j}^D$ of $D^{\otimes n}$ for all $1 \leq i,j \leq n$. It is then easily verified that the assignment \eqref{eq:A_automorphisms} preserves the defining relations of $\HH_n(A,z,q)$ from \cref{thm:DAFHA_bernstein_presentation}. Taking $D=A$ obtains the group action.
\end{proof}

Let $\SL_2(\ZZ)$ be the group of $2 \times 2$ matrices with integer coefficients and determinant $1$.

\begin{lemma} \label{thm:quotient_of_B_3}
There is a surjective group homomorphism,
\begin{gather*}
B_3 = \langle \sigma_1,\sigma_2 | \sigma_1 \sigma_2 \sigma_1 = \sigma_2 \sigma_1 \sigma_2 \rangle \to \SL_2(\ZZ),\\
\sigma_1 \mapsto \begin{pmatrix}
1 & 1\\
0 & 1
\end{pmatrix}, \quad \quad \sigma_2 \mapsto \begin{pmatrix}
1 & 0\\
-1 & 1
\end{pmatrix},
\end{gather*}
which has kernel generated by $(\sigma_1 \sigma_2 \sigma_1)^4$.
\end{lemma}
\begin{proof}
See \cite[Appendix A]{BG_2008}.
\end{proof}

The torus can be described topologically as the quotient of the space $\RR^2$ by the action of $\ZZ^2$ on $\RR^2$ by translation. Let $\pi_\textup{tor} \colon \RR^2 \to \TT^2$ be the corresponding quotient map. This map restricts to $\pi_\textup{punc} \colon \RR^2 \setminus \ZZ^2 \to \TT^2_\text{punc}$, where $\TT^2_\text{punc}$ is the punctured torus. The $\HH_n(A,z,q)$ diagrams lie in the `thickened punctured torus' $\TT^2_\text{punc} \times [0,1]$ that arises from the quotient map $\pi_\textup{punc} \times \textup{id} \colon (\RR^2 \setminus \ZZ^2) \times [0,1] \to \TT^2_\text{punc} \times [0,1]$ that is the identity on the $[0,1]$ component. As the orbits of the action of $\ZZ^2$ on $\RR^2$ are invariant under the canonical action of $\SL_2(\ZZ)$ on $\RR^2$, the latter induces an action of $\SL_2(\ZZ)$ on $\TT^2$, $\TT^2_\text{punc}$ and $\TT^2_\text{punc} \times [0,1]$. \Cref{thm:quotient_of_B_3} defines an action of $B_3$ on these three spaces.

For a geometric description of the action, let the \emph{horizontal cutting strip} be the strip $\pi_\textup{punc}(\{(x,0) \in \RR^2 \setminus \ZZ^2\}) \times [0,1]$, shaded red in the diagrams \eqref{str:rho_1_pi} and \eqref{str:rho_1_T_0}. Let the \emph{vertical cutting strip} be the strip $\pi_\textup{punc}(\{(0,y) \in \RR^2 \setminus \ZZ^2\}) \times [0,1]$, shaded blue in the diagrams. The action of $\sigma_1$ on $\TT^2_\text{punc} \times [0,1]$ amounts to cutting the thickened punctured torus along the horizontal cutting strip to get a thickened cylinder, then re-identifying the boundary strips with a $360^{\circ}$ twist. The action of $\sigma_2$ is analogous, cutting along the vertical cutting strip.

In order to define automorphisms of $\HH_n(A,z,q)$ corresponding to the Dehn twists, note that the $\sigma_1$ Dehn twist does not affect the homotopy class of horizontal loops around the torus. Accordingly, any element of $H_n^\affX(A,z) \subseteq \HH_n(A,z,q)$ should remain fixed by our automorphism. Vertical loops around the torus do get cut, twisted, and connected up again when acted on by $\sigma_1$, so elements $Y_i$, $\pi$ and $T_0$ of $\HH_n(A,z,q)$ are not fixed by the $\sigma_1$ Dehn twist. For example, in $\HH_5(A,z,q)$,
\begin{equation} \label{str:rho_1_pi}
\begin{tikzpicture}[anchorbase,scale=0.7]
		\filldraw[red,opacity=0.2] (3*\torwidth/4,\torheight+\torthickness/2) circle (2pt);
        \draw[purple] (\torwidth/2,\torheight/2) to[out=up,in=-90-\torangle] (3*\torwidth/4,\torheight+\torthickness/2);
        \draw (\torwidth/4,\torheight/4) to[out=up,in=down] (\torwidth/2,\torheight/2+\torthickness);
        \draw (0,0) to[out=up,in=down] (\torwidth/4,\torheight/4+\torthickness);
        \draw (-\torwidth/4,-\torheight/4) to[out=up,in=down] (0,\torthickness);
        \draw (-\torwidth/2,-\torheight/2) to[out=up,in=down] (-\torwidth/4,-\torheight/4+\torthickness);
		\filldraw[red,opacity=0.2] (0,-\torheight+\torthickness/2) circle (2pt);
		\draw[wipe] (0,-\torheight+\torthickness/2) to[out=\torangle,in=down] (-\torwidth/2,-\torheight/2+\torthickness);
        \draw[purple] (0,-\torheight+\torthickness/2) to[out=\torangle,in=down] (-\torwidth/2,-\torheight/2+\torthickness);
        \draw[DarkGreen,dashed] (-\torwidth,-\torheight+\torthickness) -- (\torwidth,\torheight+\torthickness);
        \draw[DarkGreen,dashed] (-\torwidth,-\torheight) -- (\torwidth,\torheight);
		\puncturedtorus
    \end{tikzpicture}
    \mapsto
    \begin{tikzpicture}[anchorbase,scale=0.7]
		\filldraw[red,opacity=0.2] (3*\torwidth/4,\torheight+\torthickness/2) circle (2pt);
        \draw[purple] (\torwidth/2,\torheight/2) to[out=up,in=-90-\torangle] (3*\torwidth/4,\torheight+\torthickness/2);
        \draw (\torwidth/4,\torheight/4) to[out=up,in=down] (\torwidth/2,\torheight/2+\torthickness);
        \draw (0,0) to[out=up,in=down] (\torwidth/4,\torheight/4+\torthickness);
        \draw (-\torwidth/4,-\torheight/4) to[out=up,in=down] (0,\torthickness);
        \draw (-\torwidth/2,-\torheight/2) to[out=up,in=down] (-\torwidth/4,-\torheight/4+\torthickness);
		\filldraw[red,opacity=0.2] (0,-\torheight+\torthickness/2) circle (2pt);
		\filldraw[blue,opacity=0.2] (-3*\torwidth/4,-\torheight+1.1*\torthickness) circle (2pt);
        \draw[purple] (-3*\torwidth/4,-\torheight+1.1*\torthickness) to[out=\torangle,in=down] (-\torwidth/2,-\torheight/2+\torthickness);
        \draw[purple] (0,-\torheight+\torthickness/2) to[out=\torangle,in=left] (\torwidth/2,-\torheight+1.1*\torthickness);
        \filldraw[blue,opacity=0.2] (\torwidth/2,-\torheight+1.1*\torthickness) circle (2pt);
        \draw[DarkGreen,dashed] (-\torwidth,-\torheight+\torthickness) -- (\torwidth,\torheight+\torthickness);
        \draw[DarkGreen,dashed] (-\torwidth,-\torheight) -- (\torwidth,\torheight);
		\puncturedtorus
    \end{tikzpicture},
\end{equation}
corresponding to $\pi \mapsto X_1 \pi$, and
\begin{equation} \label{str:rho_1_T_0}
\begin{tikzpicture}[anchorbase,scale=0.7]
		\filldraw[red,opacity=0.2] (3*\torwidth/4-0.75*\torwidth,\torheight+0.35*\torthickness) circle (2pt);
        \draw[cyan] (\torwidth/2,\torheight/2) to[out=up,in=-90-\torangle,looseness=0.75] (3*\torwidth/4-0.75*\torwidth,\torheight+0.35*\torthickness);
        \draw[wipe] (\torwidth/4,\torheight/4) to[out=up,in=down] (\torwidth/4,\torheight/4+\torthickness);
        \draw (\torwidth/4,\torheight/4) to[out=up,in=down] (\torwidth/4,\torheight/4+\torthickness);
        \draw (0,0) to[out=up,in=down] (0,\torthickness);
        \draw (-\torwidth/4,-\torheight/4) to[out=up,in=down] (-\torwidth/4,-\torheight/4+\torthickness);
        \draw[wipe] (-\torwidth/2,-\torheight/2) to[out=up,in=90-\torangle] (0,-\torheight+0.65*\torthickness);
        \draw[purple] (-\torwidth/2,-\torheight/2) to[out=up,in=90-\torangle] (0,-\torheight+0.65*\torthickness);
        \draw[purple] (3*\torwidth/4,\torheight+0.65*\torthickness) to[out=-90-\torangle,in=down,looseness=0.75] (\torwidth/2,\torheight/2+\torthickness);
        \filldraw[red,opacity=0.2] (3*\torwidth/4,\torheight+0.65*\torthickness) circle (2pt);
        \filldraw[red,opacity=0.2] (0,-\torheight+0.65*\torthickness) circle (2pt);
		\draw[wipe] (-0.75*\torwidth,-\torheight+0.35*\torthickness) to[out=90-\torangle,in=down] (-\torwidth/2,-\torheight/2+\torthickness);
        \draw[cyan] (-0.75*\torwidth,-\torheight+0.35*\torthickness) to[out=90-\torangle,in=down] (-\torwidth/2,-\torheight/2+\torthickness);
        \filldraw[red,opacity=0.2] (-0.75*\torwidth,-\torheight+0.35*\torthickness) circle (2pt);
        \draw[DarkGreen,dashed] (-\torwidth,-\torheight+\torthickness) -- (\torwidth,\torheight+\torthickness);
        \draw[DarkGreen,dashed] (-\torwidth,-\torheight) -- (\torwidth,\torheight);
		\puncturedtorus
    \end{tikzpicture}
    \mapsto
    \begin{tikzpicture}[anchorbase,scale=0.7]
		\filldraw[red,opacity=0.2] (3*\torwidth/4-0.75*\torwidth,\torheight+0.35*\torthickness) circle (2pt);
        \draw[cyan] (\torwidth/2,\torheight/2) to[out=up,in=-90-\torangle,looseness=0.75] (3*\torwidth/4-0.75*\torwidth,\torheight+0.35*\torthickness);
        \draw[wipe] (\torwidth/4,\torheight/4) to[out=up,in=down] (\torwidth/4,\torheight/4+\torthickness);
        \draw (\torwidth/4,\torheight/4) to[out=up,in=down] (\torwidth/4,\torheight/4+\torthickness);
        \draw (0,0) to[out=up,in=down] (0,\torthickness);
        \draw (-\torwidth/4,-\torheight/4) to[out=up,in=down] (-\torwidth/4,-\torheight/4+\torthickness);
        \draw[purple] (-\torwidth/2,-\torheight/2) to[out=up,in=right] (-\torwidth+\torskew/2,-\torheight+0.65*\torthickness);
        \draw[purple] (1.5*\torskew,-\torheight+0.65*\torthickness) to[out=left,in=90-\torangle] (0,-\torheight+0.65*\torthickness);
        \draw[purple] (3*\torwidth/4,\torheight+0.65*\torthickness) to[out=-90-\torangle,in=down,looseness=0.75] (\torwidth/2,\torheight/2+\torthickness);
        \filldraw[red,opacity=0.2] (3*\torwidth/4,\torheight+0.65*\torthickness) circle (2pt);
        \filldraw[red,opacity=0.2] (0,-\torheight+0.65*\torthickness) circle (2pt);
        \filldraw[blue,opacity=0.2] (-\torwidth+\torskew/2,-\torheight+0.65*\torthickness) circle (2pt);
        \filldraw[blue,opacity=0.2] (1.5*\torskew,-\torheight+0.65*\torthickness) circle (2pt);
		\draw[wipe] (-0.75*\torwidth,-\torheight+0.35*\torthickness) to[out=90-\torangle,in=left] (\torwidth/2+\torskew/4,\torheight/4);
        \draw[cyan] (-0.75*\torwidth,-\torheight+0.35*\torthickness) to[out=90-\torangle,in=left] (\torwidth/2+\torskew/4,\torheight/4);
        \draw[cyan] (-\torwidth/2-0.75*\torskew,\torheight/4) to[out=right,in=down] (-\torwidth/2,-\torheight/2+\torthickness);
        \filldraw[red,opacity=0.2] (-0.75*\torwidth,-\torheight+0.35*\torthickness) circle (2pt);
        \filldraw[blue,opacity=0.2] (\torwidth/2+\torskew/4,\torheight/4) circle (2pt);
        \filldraw[blue,opacity=0.2] (-\torwidth/2-0.75*\torskew,\torheight/4) circle (2pt);
        \draw[DarkGreen,dashed] (-\torwidth,-\torheight+\torthickness) -- (\torwidth,\torheight+\torthickness);
        \draw[DarkGreen,dashed] (-\torwidth,-\torheight) -- (\torwidth,\torheight);
		\puncturedtorus
    \end{tikzpicture},
\end{equation}
corresponding to $T_0 \mapsto X_1 T_0 X_1^{-1}$. In both the above examples, we have cut the thickened punctured torus at the strip formed by the red faces to get a thickened cylinder, then twisted any strings near the front red face by $360^{\circ}$ anti-clockwise, and re-identified the red faces to recover the thickened punctured torus. Using this, for $1 \leq i \leq n$, we find that $Y_i = T_{i-1}^{-1} \dotsb T_1^{-1} \pi T_{n-1} \dotsb T_i$ is mapped to
\begin{multline} \label{eq:rho_1_Y_i}
T_{i-1}^{-1} \dotsb T_1^{-1} X_1 \pi T_{n-1} \dotsb T_i\pjmbreak
= (T_{i-1}^{-1} \dotsb T_1^{-1})(T_1^{-1} \dotsb T_{i-1}^{-1}) X_i T_{i-1}^{-1} \dotsb T_1^{-1} \pi T_{n-1} \dotsb T_i\\
= (T_{i-1}^{-1} \dotsb T_1^{-1})(T_1^{-1} \dotsb T_{i-1}^{-1}) X_i Y_i = P_i^{-1}X_iY_i,
\end{multline}
where
\begin{equation} \label{eq:P_i}
P_i := (T_{i-1} \dotsb T_1)(T_1 \dotsb T_{i-1}),
\end{equation}
motivating the following proposition:

\begin{proposition} \label{thm:rho_1}
Let $n \in \NN, n \geq 2$. The assignment $\rho_1 \colon \HH_n(A,z,q) \to \HH_n(A,z,q)$ given by
\[
\mathbf{a} \mapsto \mathbf{a}, \qquad X_i \mapsto X_i, \qquad Y_i \mapsto P_i^{-1} X_i Y_i, \qquad T_j \mapsto T_j,
\]
for $\mathbf{a} \in A^{\otimes n}$, $1 \leq i \leq n$ and $1 \leq j \leq n-1$ defines a $\kk$-superalgebra isomorphism with inverse
\[
\mathbf{a} \mapsto \mathbf{a}, \qquad X_i \mapsto X_i, \qquad Y_i \mapsto X_i^{-1}P_iY_i, \qquad T_j \mapsto T_j.
\]
Diagrammatically,
\begin{gather*}
\rho_1 \left( \begin{tikzpicture}[centerzero]
        \draw (-0.8,-0.5) node[anchor=north] {\dotlabel{1}} to[out=up,in=down] (-0.4,0.5);
        \draw (0.4,-0.5) node[anchor=north] {\dotlabel{n-1}} to[out=up,in=down] (0.8,0.5);
		\draw (0.8,-0.5) node[anchor=135 ,yshift=-1] {\dotlabel{n}} to[out=up,in=225] (1.2,0);
		\draw (-1.2,0) to[out=45,in=down] (-0.8,0.5);
        \identify{-1.2}{-0.5}{1.2}{0.5};
        \node at (0,0) {$\cdots$};
    \end{tikzpicture} \right)
    \ = \
    \begin{tikzpicture}[centerzero]
        \draw (-0.8,-0.5) node[anchor=north] {\dotlabel{1}} to[out=up,in=down] (-0.4,0.5);
        \draw (0.4,-0.5) node[anchor=north] {\dotlabel{n-1}} to[out=up,in=down] (0.8,0.5);
		\draw (0.8,-0.5) node[anchor=135 ,yshift=0] {\dotlabel{n}} to[out=up,in=225] (1.2,0);
		\draw (-1.2,0) to[out=45,in=down] (-0.8,0.5);
        \greydot{-0.85,0.35} node[anchor=-45,yshift=2] {\dotlabel{1}};
        \identify{-1.2}{-0.5}{1.2}{0.5};
        \node at (0,0) {$\cdots$};
    \end{tikzpicture} \pjmand
\rho_1 \left(\begin{tikzpicture}[centerzero]
		\draw (-0.6,-0.5) node[anchor=north] {\dotlabel{2}} -- (-0.6,0.5);
		\draw (0.6,-0.5) node[anchor=north] {\dotlabel{n-1}} -- (0.6,0.5);
        \draw (1,-0.5) node[anchor=north] {\dotlabel{n}} to[out=up,in=225] (1.4,-0.2);
        \draw (-1.4,-0.2) to[out=45,in=down] (-1,0.5);
        \draw[wipe] (-1,-0.5) to[out=up,in=-45] (-1.4,0.2);
        \draw (-1,-0.5) node[anchor=north] {\dotlabel{1}} to[out=up,in=-45] (-1.4,0.2);
        \draw (1.4,0.2) to[out=135,in=down] (1,0.5);
        \identify{-1.4}{-0.5}{1.4}{0.5};
        \node at (0,0) {$\cdots$};
    \end{tikzpicture} \right)
    \ = \
    \begin{tikzpicture}[centerzero]
		\draw (-0.6,-0.5) node[anchor=north] {\dotlabel{2}} -- (-0.6,0.5);
		\draw (0.6,-0.5) node[anchor=north] {\dotlabel{n-1}} -- (0.6,0.5);
        \draw (1,-0.5) node[anchor=north] {\dotlabel{n}} to[out=up,in=225] (1.4,-0.2);
        \draw (-1.4,-0.2) to[out=45,in=down] (-1,0.5);
        \draw[wipe] (-1,-0.5) to[out=up,in=-45] (-1.4,0.2);
        \draw (-1,-0.5) node[anchor=north] {\dotlabel{1}} to[out=up,in=-45] (-1.4,0.2);
        \draw (1.4,0.2) to[out=135,in=down] (1,0.5);
        \greydot{-1,0.35} node[anchor=west] {\dotlabel{1}};
        \greydot{-1,-0.35} node[anchor=235] {\dotlabel{-1}};
        \identify{-1.4}{-0.5}{1.4}{0.5};
        \node at (0,0) {$\cdots$};
    \end{tikzpicture} \; .
\end{gather*}
\end{proposition}
\begin{proof}
Since $P_1 = 1$, if $\rho_1$ is a homomorphism then $\rho_1(\pi) = \rho_1(Y_1T_1^{-1} \dotsb T_{n-1}^{-1}) = X_1\pi$, and hence $\rho_1(T_0) = \rho_1(\pi T_{n-1}\pi^{-1}) = X_1T_0X_1^{-1}$, as suggested by the diagrams \labelcref{str:rho_1_T_0}. It therefore suffices to verify the relations under the assignment $\mathbf{a} \mapsto \mathbf{a}$, $X_j \mapsto X_j$, $T_i \mapsto T_i$ for $1 \leq i \leq n-1$, $\pi \mapsto X_1\pi$ and $T_0 \mapsto X_1T_0X_1^{-1}$, since by \eqref{eq:rho_1_Y_i} the resulting homomorphism sends $Y_i \mapsto P_i^{-1}X_iY_i$. Elements of $H_n^\affX(A,z) \subseteq \HH_n(A,z,q)$ are unchanged, so all cases of Relations \labelcref{dafha_rel_1,dafha_rel_2,dafha_rel_3,dafha_rel_4,dafha_rel_7,dafha_rel_8,dafha_rel_9} not involving $T_0$ are trivially satisfied. The $T_0$ versions of these relations are preserved by $T_0 \mapsto X_1T_0X_1^{-1}$, since every element other than $T_0$ appearing in them commutes with $X_1$. For example,
\[
\rho_1(T_0)^2=X_1T_0^2X_1^{-1}=X_1(zt_{n,1}T_0+1)X_1^{-1}=zt_{n,1}\rho_1(T_0)+1.
\]
Relations \labelcref{dafha_rel_5,dafha_rel_6} are easily verified. Relations \labelcref{dafha_rel_10,dafha_rel_10b} are clearly preserved, since for $\mu = (\mu_1,\dotsc,\mu_n) \in \ZZ^n$,
\[
\rho_1(\pi)\rho_1(X^\mu) = X_1 \pi X^\mu \underset{\eqref{dafha_rel_10b}}{\overset{\eqref{dafha_rel_10}}{=}} q^{-\mu_n}X^{\sigma(\mu)}X_1\pi = q^{-\mu_n}\rho_1(X^{\sigma(\mu)})\rho_1(\pi),
\]
where $\sigma \in S_n$ is the cycle $(1 \; 2 \; \dotsb \; n)$. \details{Diagrammatically,
\begin{equation*}
	\rho_1(\pi)\rho_1(X^\mu)
    \ = \
    \begin{tikzpicture}[centerzero]
        \draw (-0.8,-0.5) node[anchor=north] {\dotlabel{1}} to[out=up,in=down] (-0.4,0.5);
        \draw (0.4,-0.5) node[anchor=north] {\dotlabel{n-1}} to[out=up,in=down] (0.8,0.5);
		\draw (0.8,-0.5) node[anchor=135 ,yshift=-1] {\dotlabel{n}} to[out=up,in=225] (1.2,0);
		\draw (-1.2,0) to[out=45,in=down] (-0.8,0.5);
        \greydot{-0.8,-0.35} node[anchor=-45] {\dotlabel{\mu_1}};
        \greydot{0.4,-0.35} node[anchor=-30] {\dotlabel{\mu_{n-1}}};
        \greydot{0.85,-0.35} node[anchor=-100,yshift=2] {\dotlabel{\mu_n}};
        \greydot{-0.85,0.35} node[anchor=-45,yshift=2] {\dotlabel{1}};
        \identify{-1.2}{-0.5}{1.2}{0.5};
        \node at (0,0) {$\cdots$};
    \end{tikzpicture}
    \ = \
    q^{-\mu_n} \; \begin{tikzpicture}[centerzero]
        \draw (-0.8,-0.5) node[anchor=north] {\dotlabel{1}} to[out=up,in=down] (-0.4,0.5);
        \draw (0.4,-0.5) node[anchor=north] {\dotlabel{n-1}} to[out=up,in=down] (0.8,0.5);
		\draw (0.8,-0.5) node[anchor=135 ,yshift=0] {\dotlabel{n}} to[out=up,in=225] (1.2,0);
		\draw (-1.2,0) to[out=45,in=down] (-0.8,0.5);
        \greydot{-0.85,0.35} node[anchor=-45,yshift=2,xshift=-2] {\dotlabel{\mu_n+1}};
        \greydot{-0.4,0.35} node[anchor=south,yshift=2] {\dotlabel{\mu_1}};
        \greydot{0.8,0.35} node[anchor=south,yshift=2] {\dotlabel{\mu_{n-1}}};
        \identify{-1.2}{-0.5}{1.2}{0.5};
        \node at (0,0) {$\cdots$};
    \end{tikzpicture}
    \ = \
    q^{-\mu_n}\rho_1(X^{\sigma(\mu)})\rho_1(\pi),
\end{equation*}}
Relation \eqref{dafha_rel_8b} follows:
\[
\rho_1(T_0)\rho_1(X_n) = X_1 T_0 X_1^{-1} X_n = q^{-1}X_1^2 T_0^{-1} X_1^{-1} = q^{-1}\rho_1(X_1) \rho_1(T_0)^{-1}.
\]
The inverse is similarly shown to be a well-defined homomorphism, so $\rho_1$ is a bijection.
\end{proof}

The Dehn twist corresponding to $\sigma_2 \in B_3$ corresponds to an automorphism of $\HH_n(A,z,q)$ that leaves elements of $H_n^\aff(A,z) \subseteq \HH_n(A,z,q)$ (such as $\pi$) unchanged. 

\begin{proposition} \label{thm:rho_2}
Let $n \in \NN, n \geq 2$. The assignment $\rho_2 \colon \HH_n(A,z,q) \to \HH_n(A,z,q)$ given by
\[
\mathbf{a} \mapsto \mathbf{a}, \qquad X_i \mapsto P_i Y_iX_i, \qquad Y_i \mapsto Y_i, \qquad T_j \mapsto T_j,
\]
for $\mathbf{a} \in A^{\otimes n}$, $1 \leq i \leq n$ and $1 \leq j \leq n-1$ defines a $\kk$-superalgebra isomorphism with inverse
\[
\mathbf{a} \mapsto \mathbf{a}, \qquad X_i \mapsto Y_i^{-1}P_i^{-1}X_i, \qquad Y_i \mapsto Y_i, \qquad T_j \mapsto T_j.
\]
\end{proposition}
\begin{proof}
This follows from observing that $\rho_2=\epsilontransform^{-1} \rho_1 \epsilontransform$, for the $\kk$-superalgebra automorphism $\rho_1$ of $\HH_n(A,-z,q^{-1})$.
\end{proof}

\begin{proposition} \label{thm:B_3_action}
Let $n \in \NN, n \geq 2$. There is an action of the three-stranded braid group ${B_3 = \langle \sigma_1, \sigma_2 | \sigma_1\sigma_2\sigma_1=\sigma_2\sigma_1\sigma_2 \rangle}$ on $\HH_n(A,z,q)$ by automorphisms of $\kk$-superalgebras, where $\sigma_1$ acts by $\rho_1$ and $\sigma_2$ acts by $\rho_2^{-1}$.
\end{proposition}
\begin{proof}
The relation $\rho_1\rho_2^{-1}\rho_1=\rho_2^{-1}\rho_1\rho_2^{-1}$ is verified by evaluating it on the generators of $\HH_n(A,z,q)$. Both sides of the relation fix any $\mathbf{a} \in A^{\otimes n}$, and $T_i$ for $1 \leq i \leq n-1$, so it needs only to be evaluated on $X_i$ and $Y_i$ for $1 \leq i \leq n$ to prove equality. Indeed, the first two of the following expressions are equal, as are the last two.
\begin{align*}
\rho_1\rho_2^{-1}\rho_1(X_i)&=\rho_1\rho_2^{-1}(X_i) = \rho_1(Y_i^{-1}P_i^{-1}X_i) = Y_i^{-1}X_i^{-1}P_iP_i^{-1}X_i = Y_i^{-1},\\
\rho_2^{-1}\rho_1\rho_2^{-1}(X_i) &= \rho_2^{-1}\rho_1(Y_i^{-1}P_i^{-1}X_i) = \rho_2^{-1}(Y_i^{-1}X_i^{-1}P_iP_i^{-1}X_i)\\
&= Y_i^{-1}X_i^{-1}P_iY_iP_iP_i^{-1}Y_i^{-1}P_i^{-1}X_i = Y_i^{-1},\\
\rho_1\rho_2^{-1}\rho_1(Y_i)&=\rho_1\rho_2^{-1}(P_i^{-1}X_iY_i) = \rho_1(P_i^{-1}Y_i^{-1}P_i^{-1}X_iY_i)\\
&= P_i^{-1}Y_i^{-1}X_i^{-1}P_iP_i^{-1}X_iP_i^{-1}X_iY_i = P_i^{-1}Y_i^{-1}P_i^{-1}X_iY_i,\\
\rho_2^{-1}\rho_1\rho_2^{-1}(Y_i) &= \rho_2^{-1}\rho_1(Y_i) = \rho_2^{-1}(P_i^{-1}X_iY_i)=P_i^{-1}Y_i^{-1}P_i^{-1}X_iY_i. \qedhere
\end{align*}
\end{proof}

Let
\begin{equation} \label{eq:sigma}
\sigmaaut := \rho_1\rho_2^{-1}\rho_1 \overset{\text{Prop.~}\labelcref{thm:B_3_action}}{=} \rho_2^{-1}\rho_1\rho_2^{-1},
\end{equation}
be defined as in \cite[(2.13)]{DAHADFT_2003}. Let $w_0$ be the longest element of $S_n$, which is an involution satisfying $w_0 s_i w_0 = s_{n-i}$ for any $1 \leq i \leq n-1$. Let $v:=w_0 s_i$, so that $w_0=v s_i = s_{n-i} v$ and since by definition $l(v) = l(w_0)-1$, given a reduced expression for $v$, both $v s_i$ and $s_{n-i}v$ can be used to write reduced expressions for $w_0$. Hence, by Matsumoto's theorem, $T_{w_0} = T_vT_i = T_{n-i}T_v$, and so for all $1 \leq i \leq n-1$,
\begin{equation} \label{eq:T_w_0}
T_{w_0}T_iT_{w_0}^{-1} = (T_{n-i}T_v)T_i(T_i^{-1}T_v^{-1}) = T_{n-i}.
\end{equation}
Thus, $T_{w_0}^2$ commutes with $T_i$ for $1 \leq i \leq n-1$. As $T_{w_0}\mathbf{a}=w_0(\mathbf{a})T_{w_0}$ by Relation \eqref{qwa_rel_4} and $w_0$ is an involution, $T_{w_0}^2$ commutes with all $\mathbf{a} \in A^{\otimes n}$, and so is central in $H_n(A,z)$.

\begin{proposition} \label{thm:sigma4}
In $\HH_n(A,z,q)$, the automorphism $\sigmaaut$ satisfies
\[
\sigmaaut^4=\Ad(T_{w_0}^{-2}).
\]
\end{proposition}
\begin{proof}
One particular reduced expression for $w_0$ is
\[
w_0 = (s_1)(s_2s_1)(s_3s_2s_1) \dotsb (s_{n-1} \dotsb s_1).
\]
This implies $T_{w_0} = T_{w_0'}(T_{n-1} \dotsb T_1)$, where $w_0'$ is the longest element of the symmetric group $S_{n-1} \cong \langle s_1, \dotsc, s_{n-2} \rangle \subsetneq S_n$. Since $T_{w_0'}$ commutes with $X_n^{-1}$ by Relation \eqref{dafha_rel_7}, we have
\begin{align*}
\sigmaaut^2(X_1) &= \sigmaaut(Y_1^{-1}) = Y_1^{-1}X_1^{-1}Y_1\\
&\overset{\eqref{eq:Y_i_formula}}{=} (T_1^{-1} \dotsb T_{n-1}^{-1} \pi^{-1})X_1^{-1}(\pi T_{n-1} \dotsb T_1)\\
&\overset{\eqref{dafha_rel_10b}}{=} q^{-1}T_1^{-1} \dotsb T_{n-1}^{-1} X_n^{-1} T_{n-1} \dotsb T_1 = q^{-1}T_{w_0}^{-1}X_n^{-1}T_{w_0},\\
\sigmaaut^2(Y_1) &= \sigmaaut(Y_1^{-1}X_1Y_1) = (Y_1^{-1}X_1^{-1}Y_1)Y_1^{-1}(Y_1^{-1}X_1Y_1)\\
&= Y_1^{-1}X_1^{-1}Y_1^{-1}X_1Y_1\\
&\overset{\eqref{eq:Y_i_formula}}{=} Y_1^{-1}X_1^{-1}(T_1^{-1} \dotsb T_{n-1}^{-1} \pi^{-1})X_1(\pi T_{n-1} \dotsb T_1)\\
&\overset{\eqref{dafha_rel_10b}}{=} qY_1^{-1}X_1^{-1}T_1^{-1} \dotsb T_{n-1}^{-1} X_n T_{n-1} \dotsb T_1\\
&\overset{\eqref{daqwa_sym_rel_2}}{=} qT_1^{-1} \dotsb T_{n-1}^{-1} Y_n^{-1}X_n^{-1}X_n T_{n-1} \dotsb T_1\\
&= qT_{w_0}^{-1}Y_n^{-1}T_{w_0}.
\end{align*}
Therefore, the automorphism $\Theta:=\Ad(T_{w_0}) \circ \sigmaaut^2$ of $\HH_n(A,z,q)$ maps $X_1$ to $q^{-1}X_n^{-1}$ and $Y_1$ to $qY_n^{-1}$. For $1 \leq i \leq n-1$, assume towards induction that
\begin{equation} \label{eq:Theta_IH}
\Theta(X_i)=q^{-1}X_{n-i+1}^{-1}, \qquad \textup{and} \qquad \Theta(Y_i)=qY_{n-i+1}^{-1}.
\end{equation}
As each $T_i$ is fixed by $\sigmaaut^2$, we have
\begin{multline*}
\Theta(X_{i+1}) \overset{\eqref{daqwa_sym_rel_2}}{=} \Theta(T_iX_iT_i) = (\Ad(T_{w_0})(T_i))\Theta(X_i)(\Ad(T_{w_0})(T_i))\\
\underset{\eqref{eq:Theta_IH}}{\overset{\eqref{eq:T_w_0}}{=}} T_{n-i}(q^{-1}X_{n-i+1}^{-1})T_{n-i} \overset{\eqref{daqwa_sym_rel_2}}{=} q^{-1}X_{n-i}^{-1}.
\end{multline*}
By induction on $i$, the first part of Equation \eqref{eq:Theta_IH} holds for all $1 \leq i \leq n$. The second part is similarly proved by induction. Therefore, $\Theta^2$ is clearly the identity. Moreover, as $\Theta$ restricts to $\Ad(T_{w_0})$ on $H_n(A,z)$, we have $\Theta(T_{w_0})=\Ad(T_{w_0})(T_{w_0})=T_{w_0}$, so $\Theta \circ \Ad(T_{w_0}^{-1}) = \Ad(\Theta(T_{w_0}^{-1})) \circ \Theta = \Ad(T_{w_0}^{-1}) \circ \Theta$, whence
\[
\begin{multlined}
\sigmaaut^4 = (\sigmaaut^2)^2 = (\Ad(T_{w_0}^{-1}) \Theta)(\Ad(T_{w_0}^{-1}) \Theta)\pjmbreak
= \Ad(T_{w_0}^{-1})\Ad(\Theta(T_{w_0}^{-1})) \Theta^2 = \Ad(T_{w_0}^{-2}).
\end{multlined}\qedhere
\]
\end{proof}

Therefore, whilst the $B_3$ action from \cref{thm:B_3_action} does not in general factor to an $\SL_2(\ZZ)$ action by automorphisms, due to \cref{thm:sigma4}, we obtain an action modulo conjugations by elements of $H_n(A,z) \subsetneq \HH_n(A,z,q)$, and thus an action by outer automorphisms.

\begin{corollary} \label{thm:SL_action}
There is a group homomorphism defined by
\begin{equation}\label{eq:SL_2_action_outer}
\begin{gathered}
\SL_2(\ZZ) \to \textup{Aut}(\HH_n(A,z,q))/\textup{Inn}(\HH_n(A,z,q)),\\
\begin{pmatrix} 1 & 1 \\ 0 & 1 \end{pmatrix} \mapsto [\rho_1], \qquad \begin{pmatrix} 1 & 0 \\ 1 & 1 \end{pmatrix} \mapsto [\rho_2]. 
\end{gathered}
\end{equation}
If $z=0$, there is an action of $\SL_2(\ZZ)$ on $\HH_n(A,z,q)$ by $\kk$-superalgebra automorphisms,
\begin{equation} \label{eq:SL_2_action}
\begin{gathered}
\SL_2(\ZZ) \to \textup{Aut}(\HH_n(A,z,q)),\\
\begin{pmatrix} 1 & 1 \\ 0 & 1 \end{pmatrix} \mapsto \rho_1, \qquad \begin{pmatrix} 1 & 0 \\ 1 & 1 \end{pmatrix} \mapsto \rho_2.
\end{gathered}
\end{equation}
\end{corollary}
\begin{proof}
This is a consequence of \cref{thm:quotient_of_B_3}, \cref{thm:B_3_action} and \cref{thm:sigma4}. In the case $z=0$, we have $T_i^2 = 1$ for all $1 \leq i \leq n-1$, so $T_{w_0}^{-2}=1$, by which $\sigmaaut^4=\Ad(T_{w_0}^{-2})=\id$, so $\SL_2(\ZZ)$ acts on $\HH_n(A,z,q)$ by $\kk$-superalgebra automorphisms.
\end{proof}

\begin{remark} \label{rem:GL_action}
If we consider $z$ and $q$ as central even generators of $\HH_n(A,z,q)$, rather than parameters from $\kk$, then $\epsilontransform$ and $D$ from \cref{thm:DAFHA_geometric_isomorphisms} become genuine $\kk$-automorphisms of $\HH_n(A,z,q)$. Both are involutions, and it is easily verified that $\epsilontransform \rho_1 \epsilontransform = \rho_2$ and $\epsilontransform \rho_2 \epsilontransform = \rho_1$. Thus, there is a group homomorphism
\begin{equation} \label{eq:GL_2_action_outer}
\begin{gathered}
\GL_2(\ZZ) \to \textup{Aut}(\HH_n(A,z,q))/\textup{Inn}(\HH_n(A,z,q)),\\
\begin{pmatrix} 1 & 1 \\ 0 & 1 \end{pmatrix} \mapsto [\rho_1], \qquad \begin{pmatrix} 1 & 0 \\ 1 & 1 \end{pmatrix} \mapsto [\rho_2], \qquad \begin{pmatrix} 0 & 1 \\ 1 & 0 \end{pmatrix} \mapsto [\epsilontransform].
\end{gathered}
\end{equation}
If $z=0$ and $q^2=1$, there is an action of $\GL_2(\ZZ)$ on $\HH_n(A,z,q)$ by $\kk$-superalgebra automorphisms,
\begin{equation} \label{eq:GL_2_action}
\begin{gathered}
\GL_2(\ZZ) \to \textup{Aut}(\HH_n(A,z,q)),\\
\begin{pmatrix} 1 & 1 \\ 0 & 1 \end{pmatrix} \mapsto \rho_1, \qquad \begin{pmatrix} 1 & 0 \\ 1 & 1 \end{pmatrix} \mapsto \rho_2, \qquad \begin{pmatrix} 0 & 1 \\ 1 & 0 \end{pmatrix} \mapsto \epsilontransform.
\end{gathered}
\end{equation}
\end{remark}

\begin{remark} \label{rem:A_action_B_3_commute}
As $A^{\otimes n} \subsetneq \HH_n(A,z,q)$ is fixed pointwise by $\rho_1$ and $\rho_2$ and $\epsilontransform$, the action of $\textup{Aut}_{\tr}(A)$ from \cref{thm:A_automorphisms} commutes with the actions from \cref{thm:B_3_action}, \cref{thm:SL_action} and \cref{rem:GL_action}.
\end{remark}

\begin{example} \label{ex:combining_modular_actions}
Let $V$ be a free $\kk$-module of rank $2$ and ordered basis $\{v_1,v_2\}$ and let $A=\Lambda(V)$ be the exterior algebra of $V$. We equip $A$ with its usual superalgebra structure in which the even part is spanned by $1$ and $v_1 \wedge v_2$ and the odd part is spanned by $v_1$ and $v_2$, and the symmetric trace map defined by $\tr(v_1 \wedge v_2)=1$ and $\tr(1)=\tr(v_1)=\tr(v_2)=0$. Then the standard action of $\SL_2(\ZZ)$ on $V$ with respect to the chosen basis extends to an action on $A$ by even and trace-preserving  automorphisms, since it fixes $v_1 \wedge v_2$. Thus, by \cref{thm:A_automorphisms}, it gives an action of $\SL_2(\ZZ)$ on $\HH_n(A,z,q)$. By \cref{rem:A_action_B_3_commute}, this action commutes with the action from \cref{thm:SL_action}, hence together they give an action of $\SL_2(\ZZ) \times \SL_2(\ZZ)$ on $\HH_n(A,z,q)$ by outer automorphisms. Precomposing with the diagonal map $\SL_2(\ZZ) \to \SL_2(\ZZ) \times \SL_2(\ZZ)$ gives an action combining the modular symmetries of the supersymmetric Frobenius algebra $A$ with those of the DAHA, by acting both on the Frobenius tokens and through Dehn twists of torus diagrams.
\end{example}

\section{The \texorpdfstring{$\textup{SL}_n$}{SLₙ} Double Affine Frobenius Hecke Algebra \texorpdfstring{$\HH_{\textup{SL}_n}(A,z,q)$}{ℍₛₗₙ(A,z,q)}} \label{sec:GLn_to_SLn_DAFHA}

As explained in \cite[Sec.~4.5]{RRDAHAESWD_2021}, defining an $\SL_n$ version of the double affine Hecke algebra is not as simple as it would be for the affine Hecke algebra, where it suffices to take a quotient by an ideal generated by $\pi^n-1$. This is because Relations \eqref{dafha_rel_10} and \eqref{dafha_rel_10b} together imply $\pi^nX_i\pi^{-n}=q^{-1}X_i$ for all $1 \leq i \leq n$, which would only be compatible with the specialisation $q=1$. Instead, we give \cref{def:DAFHA_SLn} based on \cite[358]{DAHA_2005}, which was elaborated on in \cite[Sec.~4.5]{RRDAHAESWD_2021}.

Fix an $n$-th root $q^{\tfrac{1}{n}} \in \kk$ of $q$. The ring $\kk$ can always be replaced with an extension to ensure that such a root exists. Let $\tilde{\HH}_n(A,z,q)$ be the $\kk$-superalgebra that is defined as the quotient of $\HH_n(A,z,q) \ast \langle \tilde{X} \rangle$ by the relations 
\begin{equation} \label{eq:tildeX_definition}
\begin{gathered}[c]
[\tilde{X},\mathbf{a}]=0,\quad [\tilde{X},T_i]=0,\quad [\tilde{X},X_j]=0,\ifpjm\\\else\quad\fi
\tilde{X}^n=(X_1 \dotsb X_n)^{-1},\quad \pi \tilde{X} \pi^{-1} = q^{\tfrac{1}{n}}\tilde{X},
\end{gathered}
\end{equation}
for all $\mathbf{a} \in A^{\otimes n}$, $1 \leq i < n$ and $1 \leq j \leq n$. Let $\mathcal{Y}$ be the $\kk$-subalgebra of $\tilde{\HH}_n(A,z,q)$ generated by $A^{\otimes n}$, $\pi^{\pm 1}$, $T_0,\dotsc,T_{n-1}$, and $(X_1\tilde{X})^{\pm 1}, \dotsc, (X_n\tilde{X})^{\pm 1}$. In particular, $\mathcal{Y}$ is the degree-zero subalgebra of $\tilde{\HH}_n(A,z,q)$ with respect to the grading in which $\tilde{X}$ has degree $-1$ and $X_i$ has degree $1$ for each $1 \leq i \leq n$.

\begin{definition} \label{def:DAFHA_SLn}
Let $z \in \kk$, $q^{\tfrac{1}{n}} \in \kk^\times$ and $n \geq 2$. The \emph{$\SL_n$ double affine Frobenius Hecke algebra} $\HH_{\SL_n}(A,z,q)$ is the $\kk$-superalgebra quotient of $\mathcal{Y}$ by the ideal generated by $\pi^n-1$.
\end{definition}

\begin{proposition} \label{thm:DAFHA_SLn_basis}
The set
\[
\begin{multlined}
\{\mathbf{b} (X_1\tilde{X})^{\mu_1} \dotsb (X_{n-1}\tilde{X})^{\mu_{n-1}} \pi^k T_w:\pjmbreak
\mathbf{b} \in B^{\otimes n}, (\mu_1,\dotsc,\mu_{n-1}) \in \ZZ^{n-1}, k \in \ZZ/n\ZZ, w \in S_n^{\affine}\},
\end{multlined}
\]
is a $\kk$-basis of $\HH_{\SL_n}(A,z,q)$.
\end{proposition}
\begin{proof}
\Cref{thm:DAFHA_basis_theorem} gives a basis 
\[
\mathcal{B}_2 := \{\mathbf{b} X_1^{\mu_1} \dotsb X_n^{\mu_n} \pi^k T_w: \mathbf{b} \in B^{\otimes n}, (\mu_1,\dotsc,\mu_n) \in \ZZ^n, k \in \ZZ, w \in S_n^{\affine}\}
\]
of $\HH_{n}(A,z,q)$, which is a $\kk$-submodule of $\tilde{\HH}_{n}(A,z,q)$. From Relations \eqref{eq:tildeX_definition}, the set
\[
\begin{multlined}
\tilde{\mathcal{B}}_2 := \{\mathbf{b} X_1^{\mu_1} \dotsb X_n^{\mu_n} \tilde{X}^l \pi^k T_w:\pjmbreak
\mathbf{b} \in B^{\otimes n}, (\mu_1,\dotsc,\mu_n) \in \ZZ^n, 0 \leq l \leq n-1, k \in \ZZ, w \in S_n^{\affine}\}
\end{multlined}
\]
is clearly a spanning set of $\tilde{\HH}_{n}(A,z,q)$. It is simple to verify that 
\begin{gather*}
\tilde{\sigma} \colon \HH_{n}(A,z,q) \to \HH_{n}(A,z,q),\quad \pjmbreak
\mathbf{a} \mapsto \mathbf{a}, \quad X_{j}^{\pm 1} \mapsto X_{j}^{\pm 1}, \quad \pi \mapsto q^{-\tfrac{1}{n}}\pi, \quad T_i \mapsto T_i,
\end{gather*}
defines a $\kk$-superalgebra automorphism of $\HH_{n}(A,z,q)$. Then $\tilde{\HH}_{n}(A,z,q)$ is the quotient of the skew Laurent polynomial ring $\HH_{n}(A,z,q)[\tilde{X}^{\pm 1};\tilde{\sigma},0]$, as in \cite[Sec.~2.4]{NNR_1987}, by the two-sided ideal generated by $\tilde{X}^n-(X_1 \dotsb X_n)^{-1}$. Thus, by \cite[Sec.~2.3]{NNR_1987}, as a free left $\HH_{n}(A,z,q)$-module, $\tilde{\HH}_n(A,z,q)$ has a basis $\{\tilde{X}^l: 0 \leq l \leq n-1\}$. Therefore, $\tilde{\mathcal{B}}_2$ is a $\kk$-basis of $\tilde{\HH}_{n}(A,z,q)$. Since $\mathcal{Y}$ is generated as an algebra by $A^{\otimes n}$, $\pi^{\pm 1}$, $T_0,\dotsc,T_{n-1}$, and $(X_1\tilde{X})^{\pm 1}, \dotsc, (X_n\tilde{X})^{\pm 1}$, the subset
\[
\begin{multlined}
\tilde{\mathcal{B}}_\mathcal{Y} := \{\mathbf{b} (X_1\tilde{X})^{\mu_1} \dotsb (X_{n-1}\tilde{X})^{\mu_{n-1}} \pi^k T_w:\pjmbreak
\mathbf{b} \in B^{\otimes n}, (\mu_1,\dotsc,\mu_{n-1}) \in \ZZ^{n-1}, k \in \ZZ, w \in S_n^{\affine}\}
\end{multlined}
\]
of $\tilde{\mathcal{B}}_2$ is linearly independent and contained in $\mathcal{Y}$. To show that $\tilde{\mathcal{B}}_\mathcal{Y}$ spans $\mathcal{Y}$, it suffices to demonstrate that the product of an arbitrary element $\mathbf{b} (X_1\tilde{X})^{\mu_1} \dotsb (X_{n-1}\tilde{X})^{\mu_{n-1}} \pi^k T_w \in \tilde{\mathcal{B}}_\mathcal{Y}$ with any one of the given generators of $\mathcal{Y}$ must lie in the span of $\tilde{\mathcal{B}}_\mathcal{Y}$.
\details{
We will repeatedly use
\begin{equation} \label{eq:X_n_equals}
X_n\tilde{X} \overset{\eqref{eq:tildeX_definition}}{=} (X_1\tilde{X})^{-1} \dotsb (X_{n-1}\tilde{X})^{-1} \in \vectspan \tilde{\mathcal{B}}_\mathcal{Y}.
\end{equation}
\begin{enumerate}
\item This is obvious for multiplication by elements of $A^{\otimes n}$. 
\item By Relations \eqref{dafha_rel_10} and \eqref{eq:tildeX_definition}, we have $\pi (X_i \tilde{X}) \pi^{-1} = q^{\tfrac{1}{n}}(X_{i+1} \tilde{X})$ for $1 \leq i \leq n-1$, and by Relations \eqref{dafha_rel_10b} and \eqref{eq:tildeX_definition}, we have $\pi (X_n \tilde{X}) \pi^{-1} = q^{\tfrac{1}{n}-1}(X_1 \tilde{X})$. Thus,
\begin{multline*}
\pi \mathbf{b} (X_1\tilde{X})^{\mu_1} \dotsb (X_{n-1}\tilde{X})^{\mu_{n-1}} \pi^k T_w \overset{\eqref{dafha_rel_6}}{=} \sigma(\mathbf{b}) \pi (X_1\tilde{X})^{\mu_1} \dotsb (X_{n-1}\tilde{X})^{\mu_{n-1}} \pi^k T_w \\
= q^{\tfrac{1}{n}(\mu_1 + \dotsb + \mu_{n-1})}\sigma(\mathbf{b})(X_2 \tilde{X})^{\mu_1} \dotsb (X_n \tilde{X})^{\mu_{n-1}} \pi^{k+1} T_w \overset{\eqref{eq:X_n_equals}}{\in} \vectspan \tilde{\mathcal{B}}_\mathcal{Y}.
\end{multline*}
Similarly, the span of $\tilde{\mathcal{B}}_\mathcal{Y}$ is closed under left multiplication by $\pi^{-1}$. \item Multiplying by $\pi$ on the right,
\begin{equation*}
\mathbf{b} (X_1\tilde{X})^{\mu_1} \dotsb (X_{n-1}\tilde{X})^{\mu_{n-1}} \pi^k T_w \pi \overset{\eqref{dafha_rel_5}}{=} \mathbf{b} (X_1\tilde{X})^{\mu_1} \dotsb (X_{n-1}\tilde{X})^{\mu_{n-1}} \pi^{k+1} T_{\pi^{-1}w\pi} \in \vectspan \tilde{\mathcal{B}}_\mathcal{Y}.
\end{equation*}
Similarly, the span of $\tilde{\mathcal{B}}_\mathcal{Y}$ is closed under right multiplication by $\pi^{-1}$. \item For any $w \in S_n^{\affine}$ and $0 \leq i \leq n-1$, the element $T_w T_i$ can be written as a finite linear combination of elements $\sum_{w' \in S_n^{\affine}} \mathbf{a}_{w'} T_{w'}$, for some $\mathbf{a}_{w'} \in A^{\otimes n}$, by \cref{thm:AFHA_coxeter_basis} and the comment after \cref{thm:double_affine_Frobenius_Hecke_algebra}. Thus, for any $1 \leq i \leq n-1$,
\begin{multline*}
\mathbf{b} (X_1\tilde{X})^{\mu_1} \dotsb (X_{n-1}\tilde{X})^{\mu_{n-1}} \pi^k T_w T_i = \mathbf{b} (X_1\tilde{X})^{\mu_1} \dotsb (X_{n-1}\tilde{X})^{\mu_{n-1}} \pi^k \left( \sum_{w' \in S_n^{\affine}} \mathbf{a}_{w'} T_{w'} \right) \\
= \sum_{w' \in S_n^{\affine}} \left(\mathbf{b} \sigma^k(\mathbf{a}_{w'}) (X_1\tilde{X})^{\mu_1} \dotsb (X_{n-1}\tilde{X})^{\mu_{n-1}} \pi^k  T_{w'} \right) \in \vectspan \tilde{\mathcal{B}}_\mathcal{Y}.
\end{multline*}
Similarly, the span of $\tilde{\mathcal{B}}_\mathcal{Y}$ is closed under left multiplication by $T_i$ -- the key observation used to show this is that for all $1 \leq i \leq n-1$ and $1 \leq j \leq n-1$,
\begin{equation} \label{eq:z_iandT_i}
T_i (X_j\tilde{X}) \quad \begin{cases} \overset{\eqref{dafha_rel_7}}{=} (X_j\tilde{X}) T_i &\text{ if } j \notin \{i,i+1\},\\
\overset{\eqref{dafha_rel_8}}{=} (X_{i+1}\tilde{X}) T_i^{-1} \overset{\eqref{dafha_rel_3}}{=} (X_{i+1}\tilde{X}) (T_i-z t_{i,i+1}) &\text{ if } j = i,\\
\overset{\eqref{dafha_rel_3}}{=}(T_{i}^{-1}+z t_{i,i+1})(X_{i+1}\tilde{X}) \overset{\eqref{dafha_rel_8}}{=} (X_i\tilde{X}T_i + z t_{i,i+1}X_{i+1}\tilde{X}) &\text{ if } j = i+1. \end{cases}
\end{equation}
An analogous observation can be made for $i=0$, using Relation \eqref{dafha_rel_8b} instead of Relation \eqref{dafha_rel_8}, and including the appropriate factors of $q$. By \eqref{eq:z_iandT_i} and the equivalent for $i=0$, we can write $T_i \mathbf{b} (X_1\tilde{X})^{\mu_1} \dotsb (X_{n-1}\tilde{X})^{\mu_{n-1}} \pi^k T_w$ as a linear combination of elements of $\tilde{\mathcal{B}}_\mathcal{Y}$. A similar proof applies for left and right multiplication by $T_0$.
\item Relations \eqref{dafha_rel_10}, \eqref{dafha_rel_10b} and \eqref{eq:tildeX_definition} imply that for any $k \in \ZZ$ and $1 \leq i \leq n-1$,
\[
\pi^k(X_i \tilde{X})=q^{\tfrac{k}{n}-k'}(X_r \tilde{X})\pi^k,
\]
where $k'$ is the unique integer satisfying $i+k=k'n+r$ for some $1 \leq r \leq n$. Together with Relation \eqref{eq:z_iandT_i} and the observation \eqref{eq:X_n_equals}, this can be applied to show that the element $\mathbf{b} (X_1\tilde{X})^{\mu_1} \dotsb (X_{n-1}\tilde{X})^{\mu_{n-1}} \pi^k T_w (X_i \tilde{X})$ is in the span of $\tilde{\mathcal{B}}_\mathcal{Y}$ for any $1 \leq i \leq n$. This proves closure of $\tilde{\mathcal{B}}_\mathcal{Y}$ under right multiplication by $(X_i \tilde{X})$. Finally, $\tilde{\mathcal{B}}_\mathcal{Y}$ is clearly closed under left multiplication by $(X_i \tilde{X})$ for any $1 \leq i \leq n$, since $(X_i \tilde{X})$ commutes with $\mathbf{b} \in B^{\otimes n}$ and with $(X_1 \tilde{X}),\dotsc,(X_{n-1} \tilde{X})$. For $i=n$ we also use the observation \eqref{eq:X_n_equals}. Closure by left or right multiplication by $(X_i \tilde{X})^{-1}$ is analogous.
\end{enumerate}}
Thus $\tilde{\mathcal{B}}_\mathcal{Y}$ is a $\kk$-basis of $\mathcal{Y}$. The element $\pi^n$ is central in $\mathcal{Y}$, since $\pi^n(X_j\tilde{X})=q^{\frac{n}{n}-1}X_{\sigma^n(j)}\tilde{X}\pi^n=(X_j\tilde{X})\pi^n$ for all $1 \leq j \leq n$. It follows that the two-sided ideal of $\mathcal{Y}$ generated by $\pi^n-1$ is $(\pi^n-1)\mathcal{Y}$ and has a $\kk$-basis
\[
\begin{multlined}
\{\mathbf{b} (X_1\tilde{X})^{\mu_1} \dotsb (X_{n-1}\tilde{X})^{\mu_{n-1}} \pi^k(\pi^n-1) T_w:\pjmbreak
\mathbf{b} \in B^{\otimes n}, (\mu_1,\dotsc,\mu_{n-1}) \in \ZZ^{n-1}, k \in \ZZ, w \in S_n^{\affine}\},
\end{multlined}
\]
so $\HH_{\SL_n}(A,z,q)=\mathcal{Y}/\langle \pi^n-1 \rangle$ has the basis given in the statement of \cref{thm:DAFHA_SLn_basis}.
\end{proof}

\begin{proposition} \label{thm:double_affine_Frobenius_Hecke_algebra_SLn}
Let $z \in \kk$, $q^{\tfrac{1}{n}} \in \kk^\times$ and $n \geq 2$. The $\SL_n$ double affine Frobenius Hecke algebra $\HH_{\SL_n}(A,z,q)$ is isomorphic as a $\kk$-superalgebra to the quotient $\HH_{\SL_n}(A,z,q)'$ of the $\kk$-superalgebra
\[
A^{\otimes n} \ast \kk [Z_1^{\pm 1}, \dotsc ,Z_n^{\pm 1}] \ast \langle \pi^{\pm 1} \rangle  \ast \langle T_i \colon 0 \leq i \leq n-1 \rangle,
\]
by the following relations, where the indices $i,j$ range over $0 \leq i \leq n-1$ and $1 \leq j \leq n$, with the convention $t_{0,1}=t_{n,1}$, $T_0=T_n$, $Z_0=Z_n$:
\begin{align}
T_iT_j&=T_jT_i, && 0 \leq j \leq n-1,\; j \not\equiv {i\pm1} \bmod n, \label{dafhaSLn_rel_1}\\
T_iT_{i+1}T_i&=T_{i+1}T_iT_{i+1}, && 1 \leq i \leq n-2, \label{dafhaSLn_rel_2}\\
T_i^2&=zt_{i,i+1}T_i+1, && \label{dafhaSLn_rel_3}\\
T_i\mathbf{a}&=s_i(\mathbf{a})T_i, && \mathbf{a} \in A^{\otimes n}, \label{dafhaSLn_rel_4}\\
\pi T_i \pi^{-1} &= T_{i+1}, && \label{dafhaSLn_rel_5}\\
\pi \mathbf{a} &= \sigma(\mathbf{a}) \pi, && \mathbf{a} \in A^{\otimes n}, \label{dafhaSLn_rel_6}\\
T_iZ_j&=Z_jT_i, && j \not\equiv i, i+1 \bmod n, \label{dafhaSLn_rel_7}\\
T_iZ_iT_i&=Z_{i+1}, && 1 \leq i \leq n-1, \label{dafhaSLn_rel_8}\\
T_0Z_nT_0&=q^{-1}Z_1, &&\label{dafhaSLn_rel_8b}\\
Z_j\mathbf{a}&=\mathbf{a}Z_j, && \mathbf{a} \in A^{\otimes n}, \label{dafhaSLn_rel_9}\\
\pi Z_j \pi^{-1} &= q^{\tfrac{1}{n}}Z_{j+1}, && 1 \leq j \leq n-1,\label{dafhaSLn_rel_10}\\
\pi Z_n \pi^{-1} &= q^{\tfrac{1}{n}-1}Z_1, &&\label{dafhaSLn_rel_10b}\\
\pi^n &= 1, && \label{dafhaSLn_rel_11}\\
Z_1 \dotsb Z_n &= 1. && \label{dafhaSLn_rel_12}
\end{align}
In the above, we define $\sigma = (1 \; 2 \; \dotsb \; n) \in S_n$. The isomorphism from $\HH_{\SL_n}(A,z,q)'$ as presented above to $\HH_{\SL_n}(A,z,q)=\mathcal{Y}/\langle \pi^n-1 \rangle$ is given by
\begin{equation} \label{eq:double_affine_Frobenius_Hecke_algebra_SLn}
\begin{gathered}
\HH_{\SL_n}(A,z,q)' \to \HH_{\SL_n}(A,z,q),\\
\mathbf{a} \mapsto \mathbf{a},\quad \pi \mapsto \pi,\quad T_i \mapsto T_i,\quad Z_j \mapsto X_j\tilde{X}, 
\end{gathered}
\end{equation}
for $\mathbf{a} \in A^{\otimes n}$, $0 \leq i \leq n-1$, $1 \leq j \leq n$, where cosets are denoted only by their representatives, for clarity.
\end{proposition}
\begin{proof}
It is simple to use the relations from \cref{thm:double_affine_Frobenius_Hecke_algebra} to verify that \eqref{eq:double_affine_Frobenius_Hecke_algebra_SLn} gives a well-defined homomorphism $\HH_{\SL_n}(A,z,q)' \to \HH_{\SL_n}(A,z,q)$ of $\kk$-superalgebras. Four relations are markedly different from the $\GL_n$ version of this presentation, from \cref{thm:double_affine_Frobenius_Hecke_algebra}, and should be checked here. The first two of these are Relations \eqref{dafhaSLn_rel_10} and \eqref{dafhaSLn_rel_10b}. By Relation \eqref{eq:tildeX_definition} in $\mathcal{Y}$, we have $\pi \tilde{X} \pi^{-1} = q^{\tfrac{1}{n}} \tilde{X}$. Thus,
\[
\pi X_i \tilde{X} \pi^{-1} = \pi X_i \pi^{-1} \pi \tilde{X} \pi^{-1} = \begin{cases}q^{\tfrac{1}{n}} X_{i+1} \tilde{X} & \text{ if } 1 \leq i \leq n-1,\\q^{\tfrac{1}{n}-1} X_1 \tilde{X} & \text{ if } i=n. \end{cases}
\]
Relation \eqref{dafhaSLn_rel_11} is satisfied by definition of $\HH_{\SL_n}(A,z,q)$, and Relation \eqref{dafhaSLn_rel_12} follows immediately from the definition of $\tilde{X}$. The homomorphism is parity-preserving because it maps $\mathbf{a}$ to $\mathbf{a}$. Thus, the $\kk$-superalgebra homomorphism from \eqref{eq:double_affine_Frobenius_Hecke_algebra_SLn} is well-defined. It is easily verified from the presentation of $\HH_{\SL_n}(A,z,q)'$ that the set
\[
\begin{multlined}
\{\mathbf{b} Z_1^{\mu_1} \dotsb Z_{n-1}^{\mu_{n-1}} \pi^k T_w:\pjmbreak
\mathbf{b} \in B^{\otimes n}, (\mu_1,\dotsc,\mu_{n-1}) \in \ZZ^{n-1}, 0 \leq k < n, w \in S_n^{\affine}\}
\end{multlined}
\]
spans $\HH_{\SL_n}(A,z,q)'$ as a $\kk$-module. The homomorphism from \eqref{eq:double_affine_Frobenius_Hecke_algebra_SLn} maps this spanning set bijectively to the basis of $\HH_{\SL_n}(A,z,q)$ given in \cref{thm:DAFHA_SLn_basis}, and is hence a linear isomorphism, and a superalgebra isomorphism.
\end{proof}

Henceforth we identify $\HH_{\SL_n}(A,z,q)$ and $\HH_{\SL_n}(A,z,q)'$ by  \cref{thm:double_affine_Frobenius_Hecke_algebra_SLn}.

\begin{proposition} \label{thm:DAFHA_bernstein_presentation_SLn}
Let $z \in \kk$, $q^{\tfrac{1}{n}} \in \kk^\times$ and $n \geq 2$. The $\SL_n$ double affine Frobenius Hecke algebra $\HH_{\SL_n}(A,z,q)$ is isomorphic as a $\kk$-superalgebra to the quotient $\HH_{\SL_n}(A,z,q)^\ast$ of the $\kk$-superalgebra
\[
A^{\otimes n} \ast \kk [Z_1^{\pm 1}, \dotsc ,Z_n^{\pm 1}] \ast \kk [Y_1^{\pm 1}, \dotsc ,Y_n^{\pm 1}] \ast \langle T_i \colon 1 \leq i \leq n-1 \rangle
\]
by the following relations, where the indices $i,j$ range over $1 \leq i \leq n-1$ and $1 \leq j \leq n$:
\begin{align}
T_iT_j&=T_jT_i, && 1 \leq j \leq n-1,\; |i-j|>1,\\
T_iT_{i+1}T_i&=T_{i+1}T_iT_{i+1}, && 1 \leq i \leq n-2,\\
T_i^2&=zt_{i,i+1}T_i+1, &&\\
T_i\mathbf{a}&=s_i(\mathbf{a})T_i, && \mathbf{a} \in A^{\otimes n},\\
T_iZ_j=Z_jT_i &, \quad T_iY_j=Y_jT_i, && j \notin \{i,i+1\}, \label{daqwa_sym_rel_1_SLn}\\
T_iZ_iT_i=Z_{i+1} &, \quad T_iY_{i+1}T_i=Y_i, && \label{daqwa_sym_rel_2_SLn}\\
Z_j\mathbf{a}=\mathbf{a}Z_j &, \quad Y_j\mathbf{a}=\mathbf{a}Y_j, && \mathbf{a} \in A^{\otimes n}, \label{daqwa_sym_rel_3_SLn}\\
Z_1^{-1}Y_2Z_1Y_2^{-1}&=q^{\tfrac{1}{n}}T_1^2, && \label{daqwa_sym_rel_4_SLn}\\
Y_1 \dotsb Y_n &= 1, && \label{daqwa_sym_rel_5_SLn}\\
Z_1 \dotsb Z_n &= 1, && \label{daqwa_sym_rel_6_SLn}
\end{align}
via the homomorphism $\HH_{\SL_n}(A,z,q) \to \HH_{\SL_n}(A,z,q)^\ast$ of $\kk$-superalgebras sending $Z_1^{\pm 1}, \dotsc, Z_n^{\pm 1}$, $T_1^{\pm 1}, \dotsc, T_{n-1}^{\pm 1}$ and $\mathbf{a} \in A^{\otimes n}$ to themselves, and sending
\[
\pi \mapsto Y_1 T_1^{-1} \dotsb T_{n-1}^{-1} \qquad \text{ and } \qquad T_0 \mapsto Y_n^{-1}Y_1 T_1^{-1} \dotsb T_{n-1}^{-1} \dotsb T_1^{-1}.
\]
The inverse is a $\kk$-superalgebra homomorphism $\HH_{\SL_n}(A,z,q)^\ast \to \HH_{\SL_n}(A,z,q)$ sending
\[
Y_i \mapsto T_{i-1}^{-1} \dotsb T_1^{-1}\pi T_{n-1} \dotsb T_{i}.
\]
\end{proposition}
\begin{proof}
This is analogous to the proof of \cref{thm:DAFHA_bernstein_presentation}. The main differences are in the verifications of Relations \eqref{daqwa_sym_rel_4_SLn} and \eqref{daqwa_sym_rel_5_SLn}, required to prove that a well-defined homomorphism $\HH_{\SL_n}(A,z,q)^\ast \to \HH_{\SL_n}(A,z,q)$ can be defined. To check Relation \eqref{daqwa_sym_rel_4_SLn} holds in $\HH_{\SL_n}(A,z,q)$, in $\tilde{\HH}_n(A,z,q)$, Relations \eqref{eq:tildeX_definition} and \eqref{eq:Y_i_formula} imply $Y_2\tilde{X} = q^{\tfrac{1}{n}} \tilde{X} Y_2$, so
\[
(X_1\tilde{X})^{-1}Y_2(X_1\tilde{X})Y_2^{-1} = q^{\tfrac{1}{n}} X_1^{-1}Y_2X_1Y_2^{-1} \overset{\eqref{daqwa_sym_rel_4}}{=} q^{\tfrac{1}{n}} T_1^2,
\]
as required. Relation \eqref{daqwa_sym_rel_5_SLn} follows immediately from \cref{thm:pi_n}, by which $Y_1 \dotsb Y_n=\pi^n$ in $\HH_n(A,z,q)$ and hence in $\HH_{\SL_n}(A,z,q)$.
\end{proof}


Henceforth, we identify $\HH_{\SL_n}(A,z,q)$ with the algebra $\HH_{\SL_n}(A,z,q)^\ast$ defined in \cref{thm:DAFHA_bernstein_presentation_SLn}. Moreover, motivated by \cref{thm:DAFHA_SLn_basis}, for $\mu=(\mu_1,\dotsc,\mu_{n-1}) \in \ZZ^{n-1}$, let $Z^\mu := Z_1^{\mu_1} \dotsb Z_{n-1}^{\mu_{n-1}}$.

\begin{corollary} \label{thm:DAFHA_SLn_bases}
The set
\begin{equation} \label{eq:DAFHA_SLn_bases_1}
\begin{aligned}
&\{\mathbf{b} Z^\mu \pi^k T_w: \mathbf{b} \in B^{\otimes n}, \mu \in \ZZ^{n-1}, 0 \leq k < n, w \in S_n^{\affine}\} \\
={}&\{\mathbf{b} Z^\mu T_w: \mathbf{b} \in B^{\otimes n}, \mu \in \ZZ^{n-1}, w \in W_{\SL_n}^{\extaffine}\},
\end{aligned}
\end{equation}
is a $\kk$-basis of $\HH_{\SL_n}(A,z,q)$. The set
\begin{equation} \label{eq:DAFHA_SLn_bases_2}
\{\mathbf{b} Z^\mu Y^\lambda T_w: \mathbf{b} \in B^{\otimes n}, \mu \in \ZZ^{n-1},\lambda \in \ZZ^{n-1}, w \in S_n\}
\end{equation}
is a $\kk$-basis of $\HH_{\SL_n}(A,z,q)$.
\end{corollary}
\begin{proof}
Both sets in \eqref{eq:DAFHA_SLn_bases_1} are the basis from \cref{thm:DAFHA_SLn_basis} written using the generators of the presentation given in \cref{thm:double_affine_Frobenius_Hecke_algebra_SLn}, since $W_{\SL_n}^{\extaffine} = \langle \pi_{\SL} \rangle \ltimes S_n^{\affine} \cong \ZZ/n\ZZ \ltimes S_n^{\affine}$, with the convention $T_{\pi_{\SL}^{[k]} w}=\pi^kT_w$ for $[k] \in \ZZ/n\ZZ$ and $w \in S_n^{\affine}$. The set in \eqref{eq:DAFHA_SLn_bases_2} can be deduced to be a basis by an analogous proof, starting from the second basis given in \cref{thm:DAFHA_basis}.
\end{proof}

\details{Note that one difference between \cref{thm:DAFHA_SLn_bases} and the $\GL_n$ version from \cref{thm:DAFHA_basis}, is that $W_{\SL_n}^{\extaffine} \cong \ZZ/n\ZZ \ltimes S_n^{\affine}$ is used in the indexing of the basis elements rather than $W_{\GL_n}^{\extaffine} \cong \ZZ \ltimes S_n^{\affine}$. }If $\kk$ contains a square root $q^{\tfrac{1}{2n}}$ of $q^{\tfrac{1}{n}}$, then there are normalised automorphisms $\rho_1^{\SL}$ and $\rho_2^{\SL}$ of $\HH_{\SL_n}(A,z,q)$ corresponding to the automorphisms $\rho_1$ from \cref{thm:rho_1} and $\rho_2$ from \cref{thm:rho_2}. In particular,
\begin{equation} \label{eq:tau_1_SL}
\begin{gathered}
\rho_1^{\SL} \colon \mathbf{a} \mapsto \mathbf{a},\qquad\pjmbreak
T_i \mapsto T_i, \qquad Z_j \mapsto Z_j, \qquad \pi \mapsto q^{\tfrac{1-n}{2n}} Z_1\pi, \qquad T_0 \mapsto Z_1 T_0 Z_1^{-1},
\end{gathered}
\end{equation}
for $\mathbf{a} \in A^{\otimes n}$, $1 \leq i \leq n-1$ and $1 \leq j \leq n$, or, using $Y_i = T_{i-1}^{-1} \dotsb T_1^{-1}\pi T_{n-1} \dotsb T_{i}$,
\begin{multline}
\rho_1^{\SL}(Y_i) = T_{i-1}^{-1} \dotsb T_1^{-1} (q^{\tfrac{1-n}{2n}}Z_1 \pi) T_{n-1} \dotsb T_i\\
= q^{\tfrac{1-n}{2n}}(T_{i-1}^{-1} \dotsb T_1^{-1})(T_1^{-1} \dotsb T_{i-1}^{-1}) Z_i T_{i-1}^{-1} \dotsb T_1^{-1} \pi T_{n-1} \dotsb T_i \label{eq:rho_1_Y_i_SL}\\
= q^{\tfrac{1-n}{2n}}(T_{i-1}^{-1} \dotsb T_1^{-1})(T_1^{-1} \dotsb T_{i-1}^{-1}) Z_i Y_i = q^{\tfrac{1-n}{2n}}P_i^{-1}Z_iY_i.
\end{multline}
Similarly, $\rho_2^{\SL}$ fixes $\mathbf{a} \in A^{\otimes n}$, $T_1,\dotsc,T_{n-1}$ and $\pi$, with $\rho_2^{\SL}(Z_i)=q^{\tfrac{n-1}{2n}}P_iY_iZ_i$.

\begin{example}
In \cite[Sec.~2]{DAHADFT_2003}, $r$ indexes the non-zero minuscule weights of some irreducible root system. For $\SL_n$, these are the fundamental weights $\omega_1,\dotsc,\omega_{n-1}$, and setting $X_{\omega_i}=Z_1 \dotsb Z_i$ and $Y_{\omega_i}=Y_1 \dotsb Y_i$ with $(\omega_i,\omega_i)=\tfrac{i(n-i)}{n}$ for all $1 \leq i \leq n-1$ specialises $\rho_1^{\SL}$ and $\rho_2^{\SL}$ respectively to $\tau_+$ and $\tau_-$ of \cite[(2.9)--(2.13)]{DAHADFT_2003} when $A=\kk$.
\end{example}

\begin{remark}
A rescaling $\dot{Z}_i := q^{\tfrac{1-n}{2n}}Z_i$ for $1 \leq i \leq n$ removes the parameter $q^{\tfrac{1-n}{2n}}$ from both the definition of the map \eqref{eq:tau_1_SL} and the action by $Y_i$ calculated in \eqref{eq:rho_1_Y_i_SL}. However, this rescaling is not an automorphism like those of \cref{thm:minus_z_or_minus_T_DAHA}, and we rescale at the cost of Relations \eqref{dafhaSLn_rel_12}/\eqref{daqwa_sym_rel_6_SLn} becoming $\dot{Z}_1 \dotsb \dot{Z}_n=q^{\tfrac{1-n}{2}}$. Such rescaling is formalised in \cite[Remark~4.8]{RRDAHAESWD_2021}.
\end{remark}

\printbibliography

\end{document}